\documentclass[11pt]{article}
\usepackage[dvips]{graphicx}
\usepackage{bm}
\usepackage{amsmath,amsthm,amssymb}
\usepackage{amscd}

\newtheorem{thm}{Theorem}[section]
\newtheorem{theorem}[thm]{Theorem}
\newtheorem{proposition}[thm]{Proposition}
\newtheorem{lemma}[thm]{Lemma}

\newtheorem{claim}[thm]{Claim}
\newtheorem{corollary}[thm]{Corollary}

\theoremstyle{remark}
\newtheorem{example}{Example}[section]
\newtheorem{remark}[example]{Remark}

\newcommand{\dF}{\mathbb{F}}
\newcommand{\dK}{\mathbb{K}}
\newcommand{\dR}{\mathbb{R}}

\newcommand{\spa}{{\rm span}}
\newcommand{\llangle}{\langle\!\langle}
\newcommand{\rrangle}{\rangle\!\rangle}
\title{Matroids of Gain Graphs in Applied Discrete Geometry}
\author{Shin-ichi Tanigawa\footnote{
Research Institute for Mathematical Sciences, Kyoto University. 
Kyoto 606-8502, Japan. {\tt tanigawa@kurims.kyoto-u.ac.jp} 
Supported in part by 
JSPS Grant-in-Aid for Scientific Research (B) and 
JSPS Grant-in-Aid for Young Scientists (B).}} 

\begin{document}
\renewcommand{\thefootnote}{ }

\maketitle

\begin{abstract}
A $\Gamma$-gain graph is a graph whose oriented edges are labeled invertibly from a group $\Gamma$.
Zaslavsky proposed two matroids of $\Gamma$-gain graphs, called frame matroids and lift matroids,
and investigated linear representations of them. 
Each matroid has a canonical representation over a field $\mathbb{F}$ 
if $\Gamma$ is isomorphic to a subgroup of $\mathbb{F}^{\times}$ in the case of frame matroids 
or  $\Gamma$ is isomorphic to an additive subgroup of $\mathbb{F}$ in the case of lift matroids.
The canonical representation of the frame matroid of a complete graph is  also known as a Dowling geometry,
as it was first introduced by Dowling for finite groups $\Gamma$.
%On the other hand, a lift matroid is a special case elementary lifts of the graphic matroid of the underlying undirected graph.

In this paper, we extend these matroids in two ways.
% based on submodular functions over $\Gamma$
%and linear representaions of $\Gamma$.
The first one is extending the rank function of each matroid, based on submodular functions over $\Gamma$.
The resulting rank function generalizes that of the union of frame matroids or lift matroids.
Another one is extending the canonical linear representation of 
the union of $d$ copies of a frame matroid or a lift matroid, based on linear representations of $\Gamma$ on 
a $d$-dimensional vector space.
We show that linear matroids of the latter extension are indeed special cases of the first extensions, as in the relation between 
Dowling geometries and frame matroids.
We also discuss an attempt to unify the extension of frame matroids and that of lift matroids.

This work is motivated from recent research on the combinatorial rigidity of symmetric graphs.
As special cases, we give several new results on this topic,
including combinatorial characterizations of 
the symmetry-forced rigidity of generic body-bar frameworks with point group symmetries 
or  crystallographic symmetries
and   the symmetric parallel redrawability of generic bar-joint frameworks with point group symmetries 
or crystallographic symmetries.  
\end{abstract}

%\section{Introduction}
%Notation we need to mention at some point: 
%
%\begin{itemize}
%\item For a cyclic group ${\cal C}$, $\bar{\cal C}$ denotes the maximal cyclic group containing ${\cal C}$  in the underlying dihedral group ${\cal D}_h$.
%\item For two edges $e=av$ and $f=bv$, both directed to $v$, $e\cdot f^{-1}$ denotes the new edge from $a$ to $b$ with the label $\psi(e)\psi(f)^{-1}$.  
%\end{itemize}

\section{Introduction}

A {\em $\Gamma$-gain graph} $(G,\psi)$ is a pair of a graph $G=(V,E)$ 
and an assignment $\psi$ of an element of a group $\Gamma$ with each oriented edge
such that reversing  the direction inverts the assigned element.
Gain graphs are also known as group-labeled graphs
and appear in wide range of combinatorial problems and applications.
Zaslavsky~\cite{zaslavsky1989biased,zaslavsky1991biased,zaslavsky1994} 
studied a class of matroids of graphs, called {\em frame matroids} 
(formerly known as  {\em bias matroids}),
and as a principal subcase he considered a matroid of a $\Gamma$-gain graph $(G,\psi)$,
called the {\em frame matroid } $\mathbf{F}(G,\psi)$ of $(G,\psi)$.
Frame matroids include  several known matroids, 
such as graphic matroids, bicircular matroids, Dowling geometries, and matroids on signed graphs.
Zaslavsky~\cite{zaslavsky1991biased} also proposed another matroid on a $\Gamma$-gain graph $(G,\psi)$, 
called the {\rm lift matroid} $\mathbf{L}(G,\psi)$,
which can be constructed from the graphic matroid of $G$ by an elementary lift.

Each matroid has a canonical representation over a field $\mathbb{F}$ 
if $\Gamma$ is isomorphic to a subgroup of the multiplicative group $\mathbb{F}^{\times}$ of $\mathbb{F}$
in the case of $\mathbf{F}(G,\psi)$ or 
$\Gamma$ is isomorphic to an additive subgroup of $\mathbb{F}$ in the case of $\mathbf{L}(G,\psi)$.
The canonical representation of the frame matroid of a dense graph is  also known as a Dowling geometry,
as it was first introduced by Dowling~\cite{dowling1973class}.

As a further extension, Whittle~\cite{whittle1989generalisation} 
discussed a counterpart of frame matroids in general matroids,
by extending the construction of frame matroids from graphic matroids.

In this paper, we shall consider extensions, sticking to gain graphs.
We propose matroids of gain graphs, extending the constructions of frame matroids or lift matroids 
in the following two ways.
% based on submodular functions over $\Gamma$
%and linear representaions of $\Gamma$.
The first one is extending the rank function of each matroid, based on submodular functions over $\Gamma$.
The resulting rank function generalizes that of the union of frame matroids or lift matroids.
Another one is extending the canonical linear representation of 
the union of $d$ copies of a frame matroid or a lift matroid, based on linear representations of $\Gamma$ on 
a $d$-dimensional vector space.
We show that linear matroids of the latter extension are indeed special cases of the first extensions, as in the relation between 
Dowling geometries and frame matroids.

\subsection{Applications to rigidity theory}
This work is motivated from recent research on the combinatorial rigidity of symmetric graphs
and most parts of this paper are devoted to this application.

Characterizing generic rigidity of graphs is one of central problems in rigidity theory,
where a graph is identified with  a {\em bar-joint framework} 
by regarding  each vertex as a joint and each edge as a bar in the Euclidean space (see e.g.~\cite{Whitley:1997}).
In this context, a bar-joint framework is denoted by a pair $(G,p)$ of a graph $G=(V,E)$ 
and $p:V\rightarrow \mathbb{R}^d$.
For 2-dimensional rigidity, Laman's theorem~\cite{laman:Rigidity:1970} 
(along with a result by Asimov and Roth~\cite{asimov1978} or Gluck~\cite{gluck})
asserts that $(G,p)$ is minimally rigid on 
any generic $p:V\rightarrow\mathbb{R}^2$
if and only if $|E|=2|V|-3$ and $|F|\leq 2|V(F)|-3$ hold for any nonempty $F\subseteq E$, 
where $V(F)$ denotes the set of vertices incident to edges in $F$.
However, despite exhausting efforts so
far, the 3-dimensional counterpart has not been obtained yet.
 
Although characterizing generic 3-dimensional rigidity of graphs is recognized as 
one of the most difficult open problems in this field, there are solvable structural models even
in higher dimension. 
The most important case is a {\em body-bar framework} introduced by Tay~\cite{tay:84}.
A body-bar framework is a structural model consisting of disjoint rigid bodies articulated 
by bars, and the underlying graph is extracted by associating each body with a vertex and 
each bar with an edge. 
Tay~\cite{tay:84} proved that a generic body-bar framework (i.e., relative positions of bars are generic) 
is rigid  if only if the underlying graph has rank ${d+1\choose 2}(|V|-1)$ 
in the union of ${d+1\choose 2}$ copies of the graphic matroid.

Building up mathematical models of oscillations of chemical compounds or phase transitions of crystal materials is
one of main issues in theoretical physics,
and toward understanding topological impacts in such phenomena  
there are attempts to extend those theorems for generic rigidity to {\em symmetric frameworks} in the past few years.
Here, symmetric frameworks are those which are invariant with
an action of a point group in finite case or of a space group in infinite case.
The papers by Borcea and Streinu~\cite{borcea2010}, Power~\cite{power2012}, 
or Schulze et al.~\cite{Schulze2012protein} demonstrate 
applications of the theory to specific ideal crystals or proteins 
and discuss possible extensions.   

For a finite case, initiated 
by a combinatorial necessary condition~\cite{fowler2000symmetry,connelly2009symmetric}, 
Schulze~\cite{schulze,schulze2010symmetric} showed an extension of Laman's theorem of 
minimal 2-dimensional rigidity subject to certain point group symmetries.

Characterizing {\em symmetry-forced rigidity},  proposed
for finite frameworks in \cite{schulze2011orbit} and for
infinite periodic frameworks in \cite{borcea2010,borcea2011minimally}, 
is now recognized as an important initial step to understand the rigidity of symmetric frameworks,
where in this model each motion is also subject to the underlying symmetry.
(For other attempts to  capture the flexibility of periodic frameworks, 
see, e.g., \cite{owen2008frameworks,owen:2011}.)   
It was  proved that the  
{\em symmetry-forced generic rigidity} 
(i.e., symmetry-forced rigidity on generic configurations subject to the symmetry)
can be checked by computing  the rank of linear matroids defined on the edge sets 
of the underlying quotient gain graphs,
and thus can be analyzed as in a conventional manner.
After this concept has been emerged, 
characterizing in terms of 
the underlying quotient gain graphs were proved by 
Ross~\cite{ross2011,ross2012rigidity} for 
periodic 2-dimensional bar-joint frameworks and periodic 3-dimensional 
body-bar frameworks with fixed lattice metric and 
by Malestein and Theran~\cite{malestein2010,malestein2011generic}
for crystallographic 2-dimensional bar-joint frameworks with flexible lattice metric.

The result of this paper is indeed inspired by these previous results. 
As shown by Lov{\'a}sz and Yemini~\cite{lovasz:1982}, Tay~\cite{tay:84} and Whiteley~\cite{whiteley:88,Whitley:1997},
the union of copies of graphic matroids plays a central role in combinatorial rigidity theory,
that is, most combinatorial characterizations are written in terms of the union 
of copies of graphic matroids or its variants, called {\em count matroids} 
(see e.g., \cite{Frank2011} for count matroids). 
It is thus natural to investigate the union of copies of frame matroids or lift matroids 
to derive the symmetric analogues 
on gain graphs.
However, when compared with the canonical linear representation of the union of frame matroids 
(cf.~\textsection\ref{sec:poly}),
linear matroids of gain graphs proposed 
in the context of 
rigidity~\cite{schulze2011orbit,ross2012rigidity,malestein2011generic,borcea2011periodic,borcea2011minimally,borcea2011periodic}
much rely on algebraic structures of the underlying groups.
The primary motivation of this paper is to propose a new class of matroids of gain graphs, 
which forms the foundation
in the study of symmetry-forced rigidity, as does the union of graphic matroids in classical rigidity problem.

As another application, we shall also consider the symmetric version of the {\em parallel redrawing problem}
of graphs.
In the parallel redrawing problem, we are asked whether a given straight-line drawing of 
a graph admits a parallel redrawing,
that is, another straight-line drawing such that each edge is parallel to the corresponding one in the original 
drawing. 
Since any drawing admits a parallel redrawing by a translation or a dilation, we are asked 
whether all possible parallel redrawing are obtained in these trivial ways. 
%In the context of rigidity theory, the concept of parallel redrawability is known as the {\em direction-rigidity} of 
%bar-joint frameworks,
%where we are interested in direction-constraint, rather than 
%conventional length-constraint, and 
Whiteley~\cite{Whitley:1997} proved a combinatorial characterization 
for parallel redrawability of generic drawings.
Here, we shall discuss the symmetric counterpart, called the {\em symmetric parallel redrawing problem}, 
where both drawing and its redrawing are subject to symmetry. 
%Malestein and Theran~\cite{malestein2011generic} studied it in $d=2$ with rotation symmetry.

Whiteley~\cite{Whitley:1997} also gave similar matroids arose in scene analysis, 
which can be characterized by count matroids.
Replacing the union of graphic matroids with our new matroids, 
it is possible to extend the characterizations to the symmetric version. 

We list applications addressed in this papers:
the $d$-dimensional symmetric parallel redrawing problem 
with point group symmetry (\textsection\ref{subsec:drawing});
the $2$-dimensional symmetry-forced rigidity of 
bar-joint frameworks with rotational symmetry (\textsection\ref{subsec:rigidity});
the $d$-dimensional symmetry-forced rigidity of body-bar frameworks 
with point group symmetry or crystallographic symmetry 
with fixed lattice metric (\textsection\ref{subsec:symmetric_body_bar});
the $d$-dimensional symmetric parallel redrawing problem with crystallographic symmetry with flexible lattice metric
(\textsection\ref{subsec:cry_parallel});
the $2$-dimensional symmetry-forced rigidity of bar-joints frameworks with crystallographic symmetry whose linear 
part is a group of rotations (\textsection\ref{subsec:cry_parallel}).
The results provide alternative proofs of  existing works as well as new statements, 
which solve questions (explicitly/implicitly) posed 
in~\cite{ross2011,ross2012rigidity,malestein2011generic,schulze2011orbit}. 

%
%except \cite{gain_sparsity,malestein:reflection} and \cite{borcea2011minimally,borcea2011periodic}:
%in \cite{gain_sparsity,malestein:reflection} 
%combinatorial characterizations of the 2-dimensional symmetry-forced rigidity of bar-joint frameworks
%with point group symmetry are discussed, and  in \cite{borcea2011minimally,borcea2011periodic}
%combinatorial characterizations of the $d$-dimensional rigidity of periodic frameworks
%are discussed in terms of quotient graphs (but not quotient gain graphs).  
%We believe that matroids of gain graphs proposed in this paper 
%can be used to derive other symmetric analogues of combinatorial characterizations  appeared 
%in combinatorial geometry (see e.g.~\cite{Whitley:1997}).

\subsection{Organization}
The paper is organized as follows.
In \textsection\ref{sec:gain_graph} and \textsection\ref{sec:poly}, 
we briefly review fundamental facts on gain graphs and (poly)matroids, respectively.
In particular, we shall explain details of matroids 
induced by monotone submodular functions in \textsection\ref{sec:poly}, 
as our extensions belong to this class.

Extensions of frame matroids and lift matroids are described in 
\textsection\ref{sec:fractional_lifting}, \textsection\ref{sec:dowling_extension},
\textsection\ref{sec:lift} and \textsection\ref{sec:unified},
and the remaining sections are devoted to applications.
In \textsection\ref{sec:fractional_lifting}, we give 
an extension of rank functions of frame matroids via submodular functions over groups,
while in \textsection\ref{sec:dowling_extension} we
give an extension of Dowling geometries via group representations.
We give a combinatorial characterization (Theorem~\ref{thm:main_rank}) of the proposed linear matroids, 
which implies that these linear matroids are special cases of  
matroids combinatorially defined in \textsection\ref{sec:fractional_lifting}.
Proving such a characterization does not look an easy task at a glance,
but it turns out, by using the polymatroid theory discussed in \textsection\ref{sec:poly}, 
that the problem is as  easy as the case of frame matroids.  

As applications, we will discuss the parallel redrawing problem and 
the symmetry-forced rigidity of bar-joint frameworks with point group symmetry 
in \textsection\ref{sec:applications}.
In \textsection\ref{sec:action}, we also discuss an application to
the symmetry-forced rigidity of body-bar frameworks with crystallographic symmetry.  

In \textsection\ref{sec:lift} we give counterparts of those results for lift matroids.
In \textsection\ref{sec:unified}, 
we attempt to unify the extension of frame matroids and that of lift matroids,
based on the representation theory obtained so far. 
In \textsection\ref{sec:further}, we give further applications to 
the parallel redrawing problem or the rigidity problem of bar-joint frameworks with crystallographic symmetry.  

\subsection{Notations}
We conclude introduction by listing  notations used throughout the paper. 
A {\em partition} ${\cal P}$ of a finite set $E$ is a set of nonempty subsets of $E$
such that each element of $E$ belongs to exactly one subset of ${\cal P}$.
If $E=\emptyset$, the partition of $E$ is defined as the empty set.
A {\em subpartition} of $E$ is a partition of a subset of $E$.

For an undirected graph $G$, $V(G)$ and $E(G)$ denotes the vertex set and the edge set of $G$, respectively.
For $F\subseteq E(G)$, $V(F)$ denotes the set of endvertices of edges in $F$,
and let $G[F]=(V(F), F)$, that is, the graph edge-induced by $F$. 

For simplicity of description, 
we shall use some terminologies for referring edge subsets, which are conventionally used for subgraphs, as follows.
Let $F\subseteq E$.
$F$ is called {\em connected} if $G[F]$ is connected.
A {\em connected component} of $F$ is the edge set of a connected component of $G[F]$. 
$C(F)$ denotes the partition of $F$ into connected components of $F$,
and let $c(F)=|C(F)|$.
$F$ is called a {\em forest} if it contains no cycle
and  called a {\em tree} if it is connected and forest.
$F$ is called a {\em spanning tree} of a graph $G=(V,E)$ 
if $F$ is a tree with $F\subseteq E$ and $V(F)=V$.

A graph is called {\em simple} if it contains neither a loop nor parallel edges.
In a simple undirected graph, an edge between $i$ and $j$ is denoted by $\{i,j\}$.
Similarly, in a simple directed graph, an edge oriented from $i$ to $j$ is denoted by $(i,j)$.
Even though the graph is not simple, we sometimes denote $e=(i,j)$ to means that an edge $e$ is oriented from $i$ to $j$,  if it is clear from the context.
    
%Let ${\cal S}$ be a group.
%It is sometimes convinient to regard the empty set as a subgroup of ${\cal S}$.
%For a  cyclic subgroup ${\cal C}$ of ${\cal S}$,
%$\bar{\cal C}$ denotes the maximal cyclic subgroup containing ${\cal C}$.

Throughout the paper,  $\mathbb{K}$ denotes a field, which may be finite,
and $\mathbb{F}$ a subfield of $\mathbb{K}$ such that
$\mathbb{K}$ has transcendentals $\alpha_1,\dots, \alpha_k$ that form
an algebraically independent set over $\mathbb{F}$,
where $k$ is finite and will become clear from the context 
(i.e., it depends on the size of ground sets).
Then we assume that vector space $\mathbb{F}^d$ is contained in $\mathbb{K}^d$ by extension of scalars.
For a set $X\subseteq \dF^d$,
$\dim_{\mathbb{F}} X$ denotes the dimension of the linear subspace spanned by $X$ in $\dF^d$.

For a finite set $E$ and a vector space $W$, the set of linear maps from $E$ to $W$ is denoted by $W^E$,
i.e., $W^E=\{\psi\mid \psi\colon E\rightarrow W\}$. 
  
For a group $\Gamma$ and $X\subseteq \Gamma$, $\langle X\rangle$ denotes the subgroup of $\Gamma$ generated by $X$.

\section{Fundamentals on Gain Graphs}
\label{sec:gain_graph}
In this section we shall review properties of gain graphs.
See e.g.,~\cite{GrossTucker, zaslavsky1989biased,zaslavsky1991biased} for concrete explanations on this topic.
Propositions given in this section are rather straightforward,
and can be found in  \cite{gain_sparsity}. 
%All of them are basic from topology theory~\cite{Munkres199912}. 
\subsection{Gain graphs}
Let $G=(V,E)$ be a directed graph which may contain multiple edges and loops, and let $\Gamma$ be a group.
A pair is called a {\em $\Gamma$-gain graph} $(G,\psi)$, in which each edge is associated with an
element of $\Gamma$ by a {\em gain function} $\psi:E\rightarrow \Gamma$.
$G$ is a directed graph, but
its orientation is used only for the reference of the gain labeling.
Namely, we can change orientation of each edge as we like, 
by imposing a property to $\psi$ such that, if an edge in one direction has label $g$, 
then it has $g^{-1}$ in the other direction.
Thus, we often do not distinguish $G$ and the underlying undirected graph
and use notations in the introduction, which were introduced for undirected graphs, 
if it is clear from the context.
%implicitly referring the underlying graph.
%The identity is sometimes referred to as a zero gain.

%\begin{figure}[h]
%\centering
%\includegraphics[scale=1]{gain_graph.eps}
%\caption{An example of an $\Gamma$-gain graph, where $\Gamma$ is a group generated by $a$ and $b$.}
%\label{fig:gain_graph}
%\end{figure}

A {\em walk} is a sequence $W=v_0,e_1,v_1,e_2,v_2,\dots,v_{k-1}, e_k,v_k$ of 
vertices and edges such that $v_{i-1}$ and $v_i$ are endvertices of $e_i$ for $1\leq i\leq k$.
%The reversed walk of $W$ is $W^{-1}=v_k,e_{k},\dots, e_1,v_0$.
%We often denote a walk as a sequence of edges implicitly assuming the incidence at each vertex.
For two walks $W$ and $W'$ for which the end vertex of $W$ and the starting vertex of $W'$ coincide,
the concatenation of $W$ and $W'$ is  the walk $W$ followed by $W'$.
A walk is called {\em closed} if the starting vertex and the end vertex coincide.
%A {\em path} is a walk in which no vertex appears more than once.
The {\em gain} of a walk $W$ is defined as $\psi(W)=\psi(e_1)\cdot \psi(e_2)\cdots \psi(e_k)$ if each edge is oriented in the forward direction through $W$, 
and  for a backward edge $e_i$ we replace $\psi(e_i)$ with $\psi(e_i)^{-1}$ in the formulation.
%For example, in Figure~\ref{fig:gain_graph}, $W=e_2,e_5,e_7,e_4$ is a closed walk starting at $v_1$ and its gain is $b^{-1}a^3b^{-2}$. 
%Note that $\psi(W^{-1})=\psi(W)^{-1}$.

Let $(G,\psi)$ be a gain graph. 
For  $v\in V(G)$,
we denote by $\pi_1(G,v)$ the set of closed walks starting at $v$.
Similarly, for $X\subseteq E(G)$ and $v\in V(G)$, 
$\pi_1(X,v)$ denotes the set of closed walks starting at $v$ and using only edges of $X$,
where $\pi_1(X,v)=\emptyset$ if $v\notin V(X)$. 
For $X\subseteq E(G)$, the subgroup induced by $X$ relative to $v$ is defined as 
$\langle X\rangle_{\psi,v}=\{\psi(W)\mid W\in \pi_1(X,v)\}$.
%The subscription $\phi$ of $\langle X\rangle_{\phi,v}$ is sometimes omitted if it is clear from the context.
%The subscription $\psi$ of $\langle X\rangle_{\psi,v}$ is sometimes omitted if it is clear from the context.
\begin{proposition}
\label{prop:conjugate}
For any connected $X\subseteq E(G)$ and two vertices $u,v\in V(X)$, 
$\langle X\rangle_{\psi,u}$ is conjugate to $\langle X\rangle_{\psi,v}$.
\end{proposition}
%\begin{proof}
%Since $X$ is connected, there is a path $P$ starting at $u$ and ending at $v$.
%Then, for any $W\in \pi(X,u)$, $P^{-1}\ast W \ast P\in \pi(X,v)$ and hence $\psi(P)^{-1}\psi(W)\psi(P)\in \langle X\rangle_v$. 
%\end{proof}

\subsection{Switching operations}
For $v\in V(G)$ and $g\in \Gamma$, 
a {\em  switching} at $v$ with $g$  changes the gain function $\psi$ on $E(G)$ as follows:
\begin{equation*}
\label{eq:label_change}
 \psi'(e)=\begin{cases}
g\cdot \psi(e) & \text{if $e$ is directed from $v$} \\
\psi(e)\cdot g^{-1} & \text{if $e$ is directed to $v$} \\
\psi(e) & \text{otherwise}.
\end{cases}
\end{equation*}
By definition, $\psi'(e)=g\cdot \psi(e)\cdot g^{-1}$ if $e$ is a loop attached at $v$.
We say that a gain function $\psi'$ on $E(G)$ is {\em equivalent} to another gain function 
$\psi$ on $E(G)$ 
if $\psi'$ is obtained from $\psi$ by a sequence of switchings.
%$\langle X\rangle_{\psi',u}$ is conjugate to $\langle X\rangle_{\psi,u}$ for any  $X\subseteq E$ and $u\in V$. 

%\begin{figure}[h]
%\centering
%\includegraphics[width=0.65\textwidth]{gain_operation.eps}
%\caption{Switching.}
%\label{fig:gain_operation}
%\end{figure}

\begin{proposition}
\label{prop:fundamental_operation}
Let $(G,\psi)$ be a gain graph.
Let $\psi'$ be a gain function equivalent to $\psi$.
Then, for any $X\subseteq E(G)$ and any $v\in V(G)$, $\langle X\rangle_{\psi',v}$ 
is conjugate to $\langle X\rangle_{\psi,v}$.
\end{proposition}
%\begin{proof}
%Suppose that a switching is performed at $v\in V(G)$ with $g\in \Gamma$.
%Notice that 
%$\psi'(e)\psi'(f)=\psi(e)\psi(f)$ for any incoming edge $e$ to $v$ 
%and any outgoing edge $f$ from $v$.
%Also, $\psi'(e)=\psi(e)$ for any edge $e$ not incicident to $v$.
%Hence, for any closed walk $W$ starting at $u\in V(G)$,
%we have $\psi'(W)=\psi(W)$ if $u\neq v$ 
%and $\psi'(W)=g\cdot \psi(W)\cdot g^{-1}$ if $u=v$.
%Thus, for any $X\subseteq E(G)$, 
%we have $\langle X\rangle_{\psi',u}=\langle X\rangle_{\psi,u}$ if $u\neq v$,
%and $\langle X\rangle_{\psi',v}=g\cdot \langle X\rangle_{\psi,v} \cdot g^{-1}$.
%\end{proof}
%
%By an application of Proposition~\ref{prop:fundamental_operation} we have the following useful fact.

\begin{proposition}
\label{prop:tree_identity}
Let $(G,\psi)$ be a gain graph.
Then, for any forest $F\subseteq E(G)$, there is an equivalent gain function $\psi'$ on $E(G)$
such that $\psi'(e)$ is identity  for every $e\in F$.
\end{proposition}
%\begin{proof}
%Suppose that $G$ is connected.
%Let $T$ be  a spanning tree of $G$ with $F\subseteq T$.
%Take a vertex as a root and consider $T$ as a rooted tree, i.e.,
%edges of $T$ are oriented from the root to descendant.
%We then perform label-change operations from the root to descendant
%so that $\psi(e)={\sf id}$ for $e\in T$.
%
%More precisely, we first take a child $v$ of the root $u$ and perform a label-change operation
%at $v$ with $\psi(uv)$. We then take a child $w$ of $v$
%and perform a label-change operation at $w$ with $\psi''(vw)$, where $\psi''$ is the gain-labeling 
%obtained by the first label-change operation. We perform this process from the root to all leaves. 
%Each operation makes the label of an edge $e$ of $T$ identity,
%and after that the label of $e$ is never changed.
%Therefore, in the final gain labeling $\psi'$, we have $\psi'(e)={\sf id}$ for all $e\in T$. 
%Figure~\ref{fig:gain_graph2} illustrates an example.
%
%If $G$ is not connected, we can apply the same argument to each connected component.
%\end{proof}

%\begin{figure}[h]
%\centering
%\includegraphics[scale=1]{gain_graph2.eps}
%\caption{An example of Proposition~\ref{prop:tree_identity} 
%for the graph of Figure~\ref{fig:gain_graph}, where $F=\{e_2,e_3,e_5,e_7,e_8\}$.}
%\label{fig:gain_graph2}
%\end{figure}

Proposition~\ref{prop:tree_identity} suggests a simple way to compute $\langle F\rangle_{\psi,v}$ up to congruence,
in analogy with the fact that 
a cycle space of a graph is spanned by fundamental cycles. 
For a connected $X\subseteq E(G)$, 
take a spanning tree $T$ of the edge induced graph $G[X]$. 
By Proposition~\ref{prop:tree_identity} 
we can convert the gain function to an equivalent gain function such that $\phi(e)={\rm id}$ for all $e\in T$.
Then, observe that any closed walk $W\in \pi_1(X,v)$ 
can be considered as concatenations of closed walks $W_1, W_2, \dots, W_k$ 
such that $W_i$ is a closed walk in $\pi_1(X,v)$ that passes through only one edge of $X\setminus T$.
Since $\phi(e)$ is identity  for all $e\in T$, 
it follows that $\phi(W)$ is a product of elements in $\{\phi(e):e\in X\setminus T\}$, 
implying that $\langle X\rangle_{\phi,v}\subseteq \langle \phi(e):e\in X\setminus T\rangle$.
Conversely, 
$\phi(e)$  is contained in $\langle X\rangle_{\phi,v}$ for all $e\in X\setminus T$.
Thus, $\langle X\rangle_{\phi,v}=\langle \phi(e):e\in X\setminus T\rangle$.
In particular, we proved the following.
\begin{proposition}
\label{lem:checking_label}
For a connected $X\subseteq E(G)$ and a spanning tree $T$ of graph $(V(X),X)$,
suppose that $\psi(e)$ is identity for all $e\in T$.
Then, $\langle X\rangle_{\psi,v}=\langle \psi(e): e\in X\setminus T\rangle$. 
\end{proposition}

A connected edge subset $F$ in a gain graph $(G,\psi)$ is called {\em balanced} 
if $\langle F\rangle_{\psi,v}$ is the identity group for some $v\in V(F)$.
$F$ is called {\em unbalanced} if it is not balanced.
By Proposition~\ref{prop:conjugate}, this property is invariant under the choice of the base vertex  $v\in V(F)$,
and $F$ is unbalanced if and only if $F$ contains an unbalanced cycle.
%A cycle $C$ in $G$ (which may not be a directed cycle) is a {\em balanced cycle} 
%if $\langle C\rangle_v={\rm id}$ for any $v\in V(C)$.
Thus, we can extend this notion to any $F\subseteq E(G)$ (possibly disconnected sets) such that $F$ is {\em unbalanced} 
if and only if $F$ contains an unbalanced cycle.

\section{Matroids and Polymatroids}
\label{sec:poly}
\subsection{Polymatroids}
\label{subsec:polymatroids}

Let $E$ be a finite set. A function $\mu:2^E\rightarrow \mathbb{R}$ is called {\em submodular} 
if $\mu(X)+\mu(Y)\geq \mu(X\cup Y)+\mu(X\cap Y)$ for every $X,Y\subseteq E$. 
It is well known that $\mu:2^E\rightarrow \mathbb{R}$ is submodular if and only if
$\mu(X\cup\{e\})-\mu(X)\geq \mu(Y\cup \{e\})-\mu(Y)$ for any $X\subseteq Y\subseteq E$ and $e\in E\setminus Y$.
$\mu$ is called {\em monotone} if $\mu(X)\leq \mu(Y)$ for any $X\subseteq Y$.
$\mu$ is called {\em normalized} if $\mu(\emptyset)=0$.

Suppose that $\mu:2^E\rightarrow \mathbb{Z}$ is a normalized integer-valued function on $E$.
The pair $(E,\mu)$ is called an {\em integer polymatroid}  if $\mu$ is monotone and submodular,
and $\mu$ is called the {\em rank function} of $(E,\mu)$.
Throughout the paper, we shall refer to integer polymatroids as  {\em polymatroids}.
$(E,\mu)$ is called a {\em matroid} if $\mu$ further satisfies  $\mu(e)\leq 1$ for every $e\in E$.
%$F\subseteq E$ is called {\em independent} if $|F|=\mu(F)$, and a maximal independent set and a minimal dependent set are called a {\em base} and a {\em circuit}, respectively.

%Suppose $\mu:2^E\rightarrow \mathbb{Z}$ is a monotone submodular function such that $\mu(F)\geq 0$ for every nonempty $F\subseteq E$ (but $f(\emptyset)<0$ is allowed).
%Then, we define $\hat{\mu}:2^E\rightarrow \mathbb{Z}$ by
%\begin{equation}
%\label{eq:g_hat}
%\hat{\mu}(F)=\min\{\mbox{$\sum_{i=1}^k$}\mu(F_i)\} \qquad  (F\subseteq E)
%\end{equation}
%where the minimum is taken over all partitions $\{F_1,\dots,F_k\}$ of $F$ into non-empty subsets.
%It is known that $\hat{\mu}$ is a monotone submodular function satisfying $\hat{\mu}(\emptyset)=0$ (see, e.g.,\cite[Chapter 48]{Schriver} or \cite{fujishige}),
%and hence the pair $(E,\hat{\mu})$ forms a polymatroid.
%It is also known that $\hat{\mu}$ is the unique largest among all monotone submodular functions satisfying $0\leq \hat{\mu}(F)\leq \mu(F)$ for each $F\subseteq E$.

\subsection{Matroids induced by submodular functions}
\label{subsec:induced_matroids}

Let $E$ be a finite set.
An integer-valued monotone submodular function $\mu:2^E\rightarrow \mathbb{Z}$ {\em induces} a matroid on $E$, 
denoted by $\mathbf{M}(\mu)$,
where $F\subseteq E$ is independent if and only if $|X|\leq \mu(X)$ 
for every nonempty $X\subseteq F$~\cite{edmonds:1970}. 
This matroid can be understood through the following two polymatroid constructions, {\em Dilworth truncation} and {\em restriction}.

\subsubsection{Dilworth truncation}
Let $\mu:2^E\rightarrow \mathbb{Z}$ be monotone submodular.
Let us first assume that $\mu(e)\geq 0$ for every $e\in E$,
and  consider
\begin{equation}
\label{eq:poly_trun}
 \hat{\mu}(F)=\max\{\sum_{e\in F} x(e)\mid x\in \mathbb{R}_+^F \colon \emptyset\neq \forall X\subseteq F, \sum_{e\in X}x(e)\leq \mu(X)\}
\end{equation}
for $F\subseteq E$.
%, that is, the maximum value of the sum of coordinates in the polyhedron 
%\begin{equation}
%\label{eq:polyhedron}
% \{x\in \mathbb{R}^{|F|}_+\mid \emptyset\neq \forall I\subseteq F, \sum_{e\in I}x(e)\leq \mu(I)\}.
%\end{equation}
%
It is know that $\hat{\mu}:2^E\rightarrow \mathbb{R}$ is a monotone submodular function, written by
\begin{equation}
\label{eq:hat}
 \hat{\mu}(F)=\min\{\sum_{1\leq i\leq k}\mu(F_i)\mid \text{ a partition } \{F_1,\dots, F_k\} 
\text{ of } F \} \qquad (F\subseteq E),
\end{equation} 
where $\hat{\mu}(\emptyset)=0$
(see, e.g., \cite[Section 48.2]{Schrijver} or \cite[Theorem~2.6]{fujishige}).
It is easy to check that, even if $\mu(e)<0$ holds, $\hat{\mu}$ can be extended to be monotone submodular as follows:
\begin{equation}
\label{eq:hat2}
\hat{\mu}(F)=\min\{\sum_{1\leq i\leq k}\mu(F_i)\mid \text{ a partition } \{F_1,\dots, F_k\} \text{ of } F_+\} 
\qquad (F\subseteq E)
\end{equation}
where $F_+=\{e\in F\mid \mu(e)\geq 0\}$, and  $\hat{\mu}(F)=0$ if $F_+=\emptyset$.
Since $\hat{\mu}$ is nonnegative and normalized, $(E,\hat{\mu})$ is a polymatroid, 
which is called a {\em polymatroid induced by $\mu$}, denoted by $\mathbf{P}(\mu)$. 
$\hat{\mu}$ is called the {\em Dilworth truncation} (or the {\em lower truncation}) of $\mu$ in the literature.
See, e.g., \cite{Schrijver,Frank2011,fujishige} for more detail on Dilworth truncations and applications. 

\subsubsection{Restriction}
Another important operation we will use is the {\em restriction} of $\mu$ (to the hypercube). 
For a polymatroid $(E,\mu)$ (where $\mu(\emptyset)=0$ by definition), 
let $\mu^{\mathbf{1}}:2^E\rightarrow \mathbb{Z}$ be 
\begin{equation}
\label{eq:rist}
 \mu^{\mathbf{1}}(F)=\min\{|F\setminus X|+\mu(X) \mid X\subseteq F\} \qquad (F\subseteq E).
\end{equation}
Then, it can be seen that $\mu^{\mathbf{1}}$ is a monotone submodular function with $\mu^{\mathbf{1}}(F)\leq |F|$ 
(see e.g.,\cite[Section 3.1(b)]{fujishige}). 
In particular, $\mu^{\mathbf{1}}(e)\leq 1$ for every $e\in E$, 
which implies that $(E,\mu^{\mathbf{1}})$ is a matroid.
It is easy to see that $F\subseteq E$ is independent in $(E,\mu^{\mathbf{1}})$ if and only if 
$|X|\leq \mu(X)$ holds for any $X\subseteq F$.
%Indeed, if $F$ is dependent, then $|F|>\mu^{\mathbf{1}}(F)$ by definition and hence there exists $X\subseteq F$ such that $|X|>\mu(X)$ by (\ref{eq:rist}).
%Conversely, if $F$ is independent, then $|X|\leq \mu^{\mathbf{1}}(X)$ for any $X\subseteq F$, which implies $|X|=\mu(X)$ by $\mu^{\mathbf{1}}(X)\leq \mu(X)$.
%This operation corresponds to restricting the polyhedron 
%\begin{equation}
%\label{eq:polyhedron2}
%\{x\in \mathbb{R}^{|F|}\mid \forall I\subseteq F, \sum_{e\in I}x(e)\leq \mu(I)\}.
%\end{equation}
%to the unit hypercube. %Namely, 
%\begin{equation}
%\mu^{\mathbf{1}}(F)=\

\subsubsection{Rank formula of induced matroids}
Combining these two operations, 
we now check that $\hat{\mu}^{\mathbf{1}}$ (i.e., $(\hat{\mu})^{\mathbf{1}}$) is the rank of the matroid 
induced by an integer-valued monotone submodular function $\mu$.
Note that, by (\ref{eq:hat2}) and (\ref{eq:rist}),
\begin{equation}
\label{eq:rank}
\hat{\mu}^{\mathbf{1}}(F)=\min \Bigl\{|F_+\setminus \bigcup_{i=1}^k F_i|+\sum_{i=1}^k\mu(F_i)\mid 
\text{ a subpartition } \{F_1,\dots, F_k\} \text{ of } F_+ \Bigr\},
\end{equation}
and $(E,\hat{\mu}^{\mathbf{1}})$ is a matroid. 
Since $(E,\hat{\mu}^{\mathbf{1}})$ is obtained from $(E,\hat{\mu})$ by a restriction, 
$F\subseteq E$ is independent in $(E,\hat{\mu}^{\mathbf{1}})$ if and only if $|X|\leq \hat{\mu}(X)$ holds for any $X\subseteq F$.
The latter condition is equivalent to $|X|\leq \mu(X)$ for any nonempty $X\subseteq F$
by (\ref{eq:poly_trun}).
%The latter condition is equivalent to $|X|\leq \mu(X)$ for any non-empty $X\subseteq F$ 
%(see e.g.,\cite[Section 48.2]{Schriver}).
We thus have  $\mathbf{M}(\mu)=(E, \hat{\mu}^{\mathbf{1}})$.

%Dilworth truncation can be also used to truncate associated submodular functions.
%For a monotone submodular function $\mu$ with $\mu(e)\geq 1$ for $e\in E$, let $\nu=\mu-1$.
%Then, $\nu$ is monotone submodular with $\nu(e)\geq 0$ for $e\in E$ and induces the matroid $\mathbf{M}_{\nu}$.
%This matroid is referred to as the {\em Dilworth truncation} of $\mathbf{M}_{\mu}$ 
%since it can be shown $\hat{\nu}^{\mathbf{1}}(F)=\min\{|F\setminus \bigcup_i F_i|+\sum_{i}(\hat{\mu}^{\mathbf{1}}(F_i)-1)\mid \text{ a partial partition } \{F_1,\dots, F_k\} \text{ of } F\}$.  
%
%See e.g., \cite{Schriver,Frank2011} for more detail. 

\subsection{Matroid union}
\label{subsec:union}

Let us consider two monotone submodular functions $\mu_1$ and $\mu_2$ on a finite set $E$.
Since the monotonicity and the submodularity are preserved by taking summation, 
$\mu_1+\mu_2$ is  monotone and submodular.
Thus, for two polymatroids $\mathbf{P}_1=(E,\mu_1)$ and $\mathbf{P}_2=(E,\mu_2)$,
$(E,\mu_1+\mu_2)$ forms a polymatroid, which is called the {\em sum} of $\mathbf{P}_1$ and $\mathbf{P}_2$.

In a similar manner, suppose that we have two matroids $\mathbf{M}_1=(E,r_1)$ and $\mathbf{M}_2=(E,r_2)$ 
with the rank functions $r_1$ and $r_2$.
Their {\em union} $\mathbf{M}_1\vee \mathbf{M}_2$ is defined by $(E,(r_1+r_2)^{\mathbf{1}})$, i.e.,
$(r_1+r_2)^{\mathbf{1}}(F)=\min\{|F\setminus X|+r_1(X)+r_2(X)\mid X\subseteq F\}$ for $F\subseteq E$.
It is well known that $F$ is independent in $\mathbf{M}_1\vee \mathbf{M}_2$ if and only if 
$F$ can be partitioned into $F_1$ and $F_2$ such that $F_i$ is independent in $\mathbf{M}_i$ 
for $i=1,2$~\cite{edmonds:1968}. 
%
%Now we consider the union of two matroids induced by monotone submodular functions $\mu_1$ and $\mu_2$.
%If $\mu_1$ and $\mu_2$ are normalized, $(E,\mu_1+\mu_2)$ is a polymatroid, 
%and thus the restriction of $(E,\mu_1+\mu_2)$ is a matroid.
%In other words, $\mathbf{M}({\mu_1+\mu_2})=\mathbf{M}({\mu_1})\vee \mathbf{M}(\mu_2)$ as noted in \cite{pym}.
%This relation is however not true in general if $\mu_1$ or $\mu_2$ is not normalized.

\subsection{Linear Polymatroids}
Let $\mathbb{K}$ be a field and $\mathbb{F}$ be a subfield of $\mathbb{K}$ as defined in introduction.
For a finite set $E$, let us associate a linear subspace $A_e$ of $\mathbb{F}^d$ with each $e\in E$
by $\Phi:e\in E\mapsto A_e\subseteq \mathbb{F}^d$. 
Then, $\dim_{\Phi}:2^{E}\rightarrow \mathbb{Z}$, defined by $\dim_{\Phi}(F)=\dim_{\dF}\{A_e\mid e\in F\}$,  
is a set function on $E$, 
and $(E,\dim_{\Phi})$ forms a polymatroid, denoted by $\mathbf{LP}(E,\Phi)$. 
If a polymatroid $(E,\mu)$ is isomorphic to $\mathbf{LP}(E,\Phi)$ for some $\Phi$ 
(i.e., $\mu(F)=\dim_{\Phi}(F)$ for any $F\subseteq E$), 
$(E,\mu)$ is said to be a {\em linear polymatroid},
and  $\Phi$ is called a {\em linear representation} of $(E,\mu)$.

If $(E,\mu)$ is a matroid, a linear representation $\Phi$ is sometimes referred to as 
an assignment  of a vector,
rather than a $1$-dimensional linear space,  with each element in $E$.

%\begin{theorem}[Lov{\'a}sz~\cite{lovasz:1977}]
%\label{theorem:flat_matroid}
%Let ${\cal A}=\{A_e: e\in E\}$ be a family of flats, and $X$ be a set of representative points in generic position.
%Then, 
%\begin{equation}
%\label{eq:flat_matroid_rank}
%\rank_X(E)=\min_{F\subseteq E} \{|E\setminus F|+\rank(\overline{{\cal A}_{F}})\}.
%\end{equation}
%\end{theorem}

\subsubsection{Generic linear matroids}
\label{subsec:points}
In \textsection\ref{subsec:induced_matroids}, 
we have reviewed two polymatroid operations, restrictions and Dilworth truncations.  
Below, we shall take a look at geometric interpretations of these operations for linear polymatroids.

%As we have seen above, a polymatroid turns out to be a matroid  by bounding the rank of each element by one.
%Here we review a geometric method for getting a maximum linear matroid 
%from the linear polymatroid ${\cal LP}(E,\psi)$ over $\mathbb{F}$.

Let $\mathbf{LP}(E,\Phi)$ be a linear polymatroid with a linear representation  
$\Phi:e\in E\mapsto A_e\subseteq \mathbb{F}^d$.
For each $e\in E$, we shall pick a basis $v_1,\dots, v_{k_e}$ of $A_e$, where $k_e=\dim_{\dF} A_e$,
and define a {\em representative vector} by $x_e=\sum_i\alpha_e^iv_i$,
where $\alpha_e^i$ is a number in $\mathbb{K}$
such that $\{\alpha_e^i\colon e\in E, 1\leq i\leq k_e\}$ is algebraically independent over $\mathbb{F}$.
That is, by extension of the underlying field from $\mathbb{F}$ to $\mathbb{K}$, 
we have {\em generically} chosen  a representative vector $x_e$ from each $A_e$.

This gives us a linear matroid with a linear representation $e \mapsto x_e$ over $\mathbb{K}$.
Lov{\'a}sz~\cite{lovasz:1977} gave its rank formula.
\begin{theorem}[Lov{\'a}sz~\cite{lovasz:1977}]
\label{theorem:flat_matroid}
Let $\mathbb{K}$ be a field and $\mathbb{F}$ be a subfield of $\mathbb{K}$.
Let $\mathbf{LP}(E,\Phi)$ be a linear polymatroid with 
a linear representation $\Phi:e\in E\mapsto A_e\subseteq \mathbb{F}^d$,
and suppose that a representative vector $x_e$ is generically chosen from each $A_e$ over $\mathbb{K}$.
Then, 
\begin{equation}
\label{eq:flat_matroid_rank}
\dim_{\dK}\{x_e\mid e\in E\}=
\min \{|E\setminus F|+\dim_{\dF}\{A_e\mid e\in F\} \mid F\subseteq E\}.
\end{equation}
\end{theorem}

Note that the right hand side of (\ref{eq:flat_matroid_rank}) does not rely on the choice of representative vectors,
and hence this motivates us to define the {\em generic} matroid.
The {\em generic matroid obtained from $\mathbf{LP}(E,\Phi)$}, 
denoted by $\mathbf{LM}(E,\Phi)$, 
is defined to be
a matroid with a linear representation $e\mapsto x_e$ over $\mathbb{K}$.
Notice the coincidence of two formula (\ref{eq:rist}) and (\ref{eq:flat_matroid_rank}).
Namely, taking the generic matroid is the same meaning as 
the restriction  for linear polymatroids. 

Lov{\'a}sz actually proved Theorem~\ref{theorem:flat_matroid} under a much weaker assumption.
For a family $\{A_e\mid e\in E\}$ of linear subspaces in $\mathbb{K}^d$,
a set of vectors $x_e$ taken from each $A_e$ is said to be {\em in generic position} if 
\begin{equation}
\label{eq:generic_position}
x_f\in {\rm span}\{x_e\mid e\in X\} \  \Rightarrow \ A_f\subseteq {\rm span}\{x_e\mid e\in X\} \qquad 
\forall X\subseteq E, \forall f\in E\setminus X
\end{equation}
If $\{x_e\mid e\in E\}$ is in generic position, (\ref{eq:flat_matroid_rank}) holds (see \cite{lovasz:1977}).

\subsubsection{Dilworth truncation}
\label{subsec:linear_truncation}
We also have a geometric interpretation of Dilworth truncation.
For a linear polymatroid $\mathbf{LP}(E,\Phi)$ with $\Phi:e\mapsto A_e$, 
let ${\cal A}=\{A_e\mid e\in E\}$.
We now consider restricting ${\cal A}$ to a generic hyperplane (i.e., a $d-1$ dimensional linear subspace)
by extending the underlying field $\mathbb{F}$ to $\mathbb{K}$, again.
A hyperplane $H$ is called {\em generic} if it is expressed by
$H=\{x\in \mathbb{K}^d \mid \sum_{1\leq i\leq d} \alpha_i x(i)=0\}$
for some algebraically independent numbers $\{\alpha_1,\dots, \alpha_d\}$ over $\mathbb{F}$. 
Lov{\'a}sz~\cite{lovasz:1977} observed the following formula.

\begin{theorem}[Lov{\'a}sz~\cite{lovasz:1977}]
\label{theorem:truncation}
Let $\mathbb{K}$ be a field and $\mathbb{F}$ be a subfield of $\mathbb{K}$.
Let $\mathbf{LP}(E,\Phi)$ be a linear polymatroid with 
a linear representation $\Phi:e\in E\mapsto A_e\subseteq \mathbb{F}^d$,
and $H$ be a generic hyperplane of $\mathbb{K}^d$.
Then, 
\begin{equation}
\label{eq:truncation}
\dim_{\dK} \{A_e\cap H\mid e\in E\} = \min\{\mbox{$\sum_{i=1}^k$} (\dim_{\dF} \{A_e\mid e\in E_i\}-1)\},
\end{equation}
where the minimum is taken over all partitions $\{E_1,\dots, E_k\}$ of $E$ into nonempty subsets.
\end{theorem}
The same result was also obtained by Mason~\cite{Mason:1981,Mason:1976} 
from the view point of combinatorial geometry (projective matroids), see also \cite{brylawski}.

Setting $\mu(F)=\dim_{\dK} \{A_e\mid e\in F\}-1$ for $F\subseteq E$, we see that
the polymatroid induced by $\mu$, that is $(E,\hat{\mu})$, has linear representation 
$e\mapsto A_e\cap H$ from the coincidence of (\ref{eq:hat}) and (\ref{eq:truncation}).

\subsubsection{Linear matroid union}
\label{subsec:linear_union}
A linear representation of the sum of two polymatroids can be easily obtained in the following manner.
Suppose that we have two linear polymatroids $(E,\mu)$ and $(E,\mu')$ with linear representations 
$\Phi:e\in E\mapsto A_e\subseteq \mathbb{F}^s$ and 
$\Phi':e\in E\mapsto A_e'\subseteq \mathbb{F}^t$, respectively.
By definition, $(\mu+\mu')(F)=\mu(F)+\mu'(F)=\dim_{\dF} \{A_e\mid e\in F\}+\dim_{\dF} \{A_e'\mid e\in F\}$.
Hence, if we prepare $\mathbb{F}^{s+t}$ as  the underlying vector space, 
the polymatroid $(E,\mu+\mu')$ is represented by $e\mapsto A_e\oplus A_{e}'$.

Combining this with the discussions of \textsection\ref{subsec:union} and \textsection\ref{subsec:points}, 
it is now straightforward to see the following.
\begin{proposition}
\label{prop:linear_union}
Let $\mathbf{M}_i$ be a matroid on a finite set $E$ with a linear 
representation $e\mapsto x_e^i$ in a vector space $W_i$ for each $i=1,2$. 
Then, $\mathbf{M}_1\vee \mathbf{M}_2$ is represented by $e\mapsto x_e$,
where $x_e$ is a representative vector taken from  
$\spa \{x_e^1\}\oplus \spa \{x_e^2\}\in W_1\oplus W_2$ in generic position.
\end{proposition}
This fact is at least known from  \cite{Mason:1976}.
More detailed descriptions with examples can be found in \cite{Mason:1976,Mason:1981,brylawski}.
%
%In a similar manner, let us consider a linear representation of $\mathbf{M}_{\mu_1}\vee \mathbf{M}_{\mu_2}$,
%where ${\cal mu}_i$ may not be normalized.
%If $\mathbf{M}_\mathbf{M}_{\mu_1}\vee \mathbf{M}_{\mu_2}$
%if we have two matroids $\mathbf{M}_{\mu_1}$ and $\mathbf{M}_{\mu}_2$
%induced by monotone submodular functions
%----------------------------------------------------

%For a finite set $E$, a {\em partition} of $E$ means a partition of $E$ into nonempty subsets.
%A {\em subpartition} is a partition of a subset $F$ of $E$,
%where $F$ may be empty.

\section{Matroids Induced by Submodular Functions over Groups}
\label{sec:fractional_lifting}

\subsection{Frame matroids}
Let $\Theta$ be the graph with two vertices $u$ and $v$ and three parallel edges.
A subdivision of $\Theta$ is called a theta graph.
Hence, a theta graph consists of three openly disjoint paths between $u$ and $v$ and contains three cycles.

Consider an undirected multigraph, which may contain loops and parallel edges.  
A family ${\cal C}$ of cycles is called a {\em linear class} if it satisfies the following property. If two cycles in ${\cal C}$ form a theta subgraph, then 
the third cycle of the theta subgraph is also contained in ${\cal C}$.
For a graph $G=(V,E)$ and a linear class ${\cal C}$ of cycles, the {\em frame matroid} 
$\mathbf{F}(G,{\cal C})$ is defined such that
$F\subseteq E$ is independent if and only if each connected component of $F$ contains no cycle or just one cycle, which is not included in the linear class ${\cal C}$ \cite{zaslavsky1989biased,zaslavsky1991biased}.
Therefore, the rank of $F\subseteq E$ in $\mathbf{F}(G,{\cal C})$ is equal to
\begin{equation*}
% \label{eq:biased_rank} 
g_{\cal C}(F):=|V(F)|-c(F)+\sum_{X\in C(F)} \alpha_{\cal C}(X) \qquad (F\subseteq E)
\end{equation*}
where
\begin{equation*}
\label{eq:t}
\alpha_{\cal C}(X)=\begin{cases}
1 & \text{if $X$ contains a cycle not included in } {\cal C} \\
0 & \text{otherwise}.
\end{cases}
\end{equation*}
This also implies that $g_{\cal C}$ is monotone and submodular. 

In this paper we are interested in frame matroids on gain graphs.
Let $(G=(V,E),\psi)$ be a $\Gamma$-gain graph for a group $\Gamma$. 
%Recall that a cycle is called balanced if the total gain walking through the cycle is equal to identity.
%This property does not depend on the choice of a base vertex of the walk by Proposition~\ref{prop:conjugate},
%and hence it is indeed a property of each cycle.
Let ${\cal C}$ be the set of balanced cycles in $(G,\psi)$.
Then, ${\cal C}$ forms a linear class, and the associated frame matroid is defined. 
This matroid is called the {\em frame matroid of $(G,\psi)$},
denoted by $\mathbf{F}(G,\psi)$. 
If we define $g_{\Gamma}:2^E\rightarrow \mathbb{Z}$ by 
\begin{equation}
\label{eq:gain_rank}
g_{\Gamma}(F)=|V(F)|-c(F)+\sum_{X\in C(F)} \alpha_{\Gamma}(X) \qquad (F\subseteq E)
\end{equation}
where
\begin{equation}
\label{eq:a2}
\alpha_{\Gamma}(X)=\begin{cases}
1 & \text{if $X$ is unbalanced} \\
0 & \text{otherwise},
\end{cases}
\end{equation}
then we have $\mathbf{F}(G,\psi)=(E,g_{\Gamma})$.
%If we denote $\alpha_{\Gamma}(F)=\sum_{F_i\in C(F)}\alpha_{\Gamma}(F_i)$ for $F\subseteq E$,
%we can rewrite $g_{\Gamma}$ by
%\begin{equation}
%\label{eq:gain_rank}
%g_{\Gamma}(F)=|V(F)|-c(F)+\alpha_{\Gamma}(F) \qquad (F\subseteq E)
%\end{equation}
%where $c(F)=|C(F)|$.
%Throughout the paper we will omit the subscription ${\cal S}$ from $g_{\cal S}$ and $\alpha_{\cal S}$ if it is clear from the context.

For a positive integer $d$, the union of $d$ copies of $\mathbf{F}(G,\psi)$ is 
$(E,(dg_{\Gamma})^{\mathbf{1}})$ by definition. 
That is, it is the matroid induced by
%\begin{equation}
%\label{eq:g_count_d}
$dg_{\Gamma}(F)=d|V(F)|-dc(F)+\sum_{X\in C(F)}d\alpha_{\Gamma}(X)$ for $F\subseteq E$.
%\end{equation}

\subsection{Lifting based on submodular functions on groups}
\label{subsec:fractional_lifting}
We  now  extend the construction of the union of frame matroids 
by using structures of the underlying group.
The idea is to replace the term $\alpha_{\Gamma}$ by a function taking fractional values.
 
For a group $\Gamma$,
we consider a function $\mu:2^{\Gamma}\rightarrow \mathbb{R}_+$ satisfying the following properties: 
\begin{description}
\item[(Normalized)] $\mu(\emptyset)=0$;
\item[(Monotonicity)] $\mu(X)\leq \mu(Y)$ for any $X\subseteq Y\subseteq \Gamma$;
\item[(Submodularity)] $\mu(X)+\mu(Y)\geq \mu(X\cup Y)+\mu(X\cap Y)$ for any $X, Y \subseteq \Gamma$;
\item[(Invariance under closure)] $\mu(X)=\mu(\langle X\rangle)$ for any nonempty $X\subseteq \Gamma$;
\item[(Invariance under conjugate)] $\mu(X)=\mu(\gamma X\gamma^{-1})$ for any nonempty $X\subseteq \Gamma$ 
and $\gamma\in \Gamma$.
\end{description}
We say that  $\mu:2^{\Gamma}\rightarrow \mathbb{R}_+$ is 
a {\em symmetric polymatroidal function} over $\Gamma$ if $\mu$ satisfies these five conditions.
The submodularity implies that, for any $X\subseteq Y\subseteq \Gamma$ and $e\in \Gamma$,
\begin{equation}
\label{eq:submodular_ineq}
\mu(X\cup \{e\})-\mu(X)\geq \mu(Y\cup \{e\})-\mu(Y).
\end{equation}

Extending the rank function (\ref{eq:gain_rank}) of frame matroids, 
we now propose a submodular function based on a symmetric polymatroidal function $\mu$.
Let $(G=(V,E),\psi)$ be a $\Gamma$-gain graph.
%For each connected $F\subseteq E$ and $v\in V(F)$,
%we define
%\begin{equation}
%\lambda_{\mu,\psi}(F,v)=\mu(\langle F\rangle_v).
%\end{equation}
We consider $\mu(\langle F\rangle_{\psi,v})$ for a connected $F\subseteq E$ and $v\in V(F)$.
By Proposition~\ref{prop:conjugate}, $\langle F\rangle_{\psi,v}$ is 
conjugate to $\langle F\rangle_{\psi,u}$ for any $u,v\in V(F)$
for $F\subseteq E$,
and hence $\mu(\langle F\rangle_{\psi,u})=\mu(\langle F\rangle_{\psi,v})$ for any $u,v\in V(F)$.
Also, by Proposition~\ref{prop:fundamental_operation},
$\mu(\langle F\rangle_{\psi,v})$ is invariant with respect to the choice of equivalent gain functions $\psi$.
We hence simply denote $\mu(\langle F\rangle_{\psi,v})$  by $\mu \langle F\rangle$, 
implicitly assuming the gain function and the base vertex among $V(F)$.
We can then define a set function  $g_{\mu}:2^E\rightarrow \mathbb{R}$ by
\begin{align}
\label{eq:g}
g_{\mu}(F)&=|V(F)|-c(F)+\sum_{X\in C(F)}\mu\langle X\rangle \qquad (F\subseteq E).
\end{align}

%We first remark that, under some connected sets, $\mu$ satisfies the monotonicity and the subodularity condition.
%\begin{lemma}
%\label{lem:mu}
%Let $\mu:2^{\Gamma}\rightarrow \mathbb{R}_+$ be a symmetric polymatroidal function over a group $\Gamma$,
%and let $(G,\psi)$ be a $\Gamma$-gain graph.
%For $X, Y\subseteq E(G)$ such that all of $X, Y, X\cup Y$ and $X\cap Y$ are connected,
%we have 
%\begin{itemize}
%\item $\mu\langle X\rangle \leq \mu\langle Y\rangle$ if $X\subseteq Y$;
%\item $\mu\langle X\rangle+\mu\langle Y\rangle \geq \mu\langle X\cup Y\rangle+\mu\langle X\cap Y\rangle$.
%\end{itemize}
%\end{lemma}
%\begin{proof}
%We shall take $v\in X\cap Y$ as the base vertex of walks.
%Clearly, $\langle X\rangle_v\subseteq \langle Y\rangle_v$ if $X\subseteq Y$.
%This implies the former relation by the monotonicity of $\mu$ over $\Gamma$.
%
%Fot the latter relation, observe that $\langle X\cap Y\rangle_v \subseteq \langle X\rangle_v\cap \langle Y\rangle_v$
%and $\langle X\cup Y\rangle_v=\langle X\rangle_v\cup \langle Y\rangle_v$. 
%\end{proof}

Notice that, if $X$ and $Y$ are connected with $X\subseteq Y\subseteq E$, we have 
$\mu\langle X\rangle\leq \mu\langle Y\rangle$ by the monotonicity of $\mu$ over $\Gamma$.
However, the monotonicity and the submodularity of $\mu$ do not hold over $E$ in general.
The next theorem ensures these properties for $g_{\mu}$. 

\begin{theorem}
\label{thm:submodularity1}
Let $\mu:2^\Gamma\rightarrow [0,1]$ be a symmetric polymatroidal function over a group $\Gamma$ (with the upper bound $1$), 
and $(G=(V,E),\psi)$ a $\Gamma$-gain graph.
Then, $g_{\mu}$ is a monotone submodular function over $E$.
\end{theorem}
\begin{proof}
%Since $|V(\cdot)|-c(\cdot)$ is nonnegative, $g_{\mu}$ is nonnegative.
For each $X\subseteq E$ and $e=(i,j)\in E\setminus X$, 
let $\Delta(X,e)=g_{\mu}(X\cup \{e\})-g_{\mu}(X)$.
We denote by $X_i$ the connected component of $X$ for which $i\in V(X_i)$.
If such a component does not exist, let $X_i=\emptyset$.
Similarly, we denote by $X_j$ the component of $X$ for which $j\in V(X_j)$.
%Note  $X_i=X_j$ if $e$ is a loop.

By a simple calculation, we have the following relation:
%\begin{align}
%\label{eq:g1}
%\Delta(X,e)=\begin{cases}
%\mu\langle e\rangle & \text{ if } i\notin V(X) \\
%\mu\langle X'\cup \{e\}\rangle-\mu \langle X'\rangle & \text{ if $i\in V(X')$ for some $X'\in C(X)$} 
%\end{cases}
%\end{align}
%and if $e$ is not a loop
\begin{align}
\label{eq:g1}
\Delta(X,e)=\begin{cases}
\mu\langle X_i\cup\{e\}\rangle -\mu\langle X_i\rangle & \text{ if } e \text{ is a loop or } X_i=X_j\neq \emptyset \\
\mu\langle X_i\cup X_j\cup\{e\}\rangle+1-\mu \langle X_i\rangle-\mu\langle X_j\rangle & \text{ otherwise. }
\end{cases}
\end{align}

Let us check the monotonicity.
Suppose that $e$ is a loop or $X_i=X_j\neq \emptyset$.
Due to the monotonicity of $\mu$ over $\Gamma$, $\mu\langle X_i\cup \{e\}\rangle-\mu\langle X_i\rangle\geq 0$.
On the other hand,  suppose not.
Since $X_i$ and $X_i\cup X_j\cup \{e\}$ are connected,
we have $\mu\langle X_i\rangle \leq \mu\langle X_i\cup X_j\cup\{e\}\rangle$  by the monotonicity of $\mu$ over $\Gamma$.
Also, by the upper bound of $\mu$, $\mu\langle X_j\rangle\leq 1$.
We thus have $\Delta(X,e)=\mu\langle X_i\cup X_j\cup \{e\}\rangle+1-(\mu\langle X_i\rangle+\mu\langle X_j\rangle )\geq 0$.
This completes the proof of the monotonicity.

For the submodularity, we check 
\begin{equation}
\label{eq:submo_proof1}
\Delta(X,e)\geq \Delta(Y,e)
\end{equation} 
for any $X\subseteq Y\subseteq E$ and $e\in E\setminus Y$. 
We split the proof into two cases.

Case 1: Suppose that $e$ is a loop or $X_i=X_j\neq\emptyset$.
We then have $X_i\subseteq Y_i=Y_j$.
We take a tree $T\subseteq Y_i$  spanning $V(Y_i)$ such that 
$T\cap X_i$ forms a tree spanning $V(X_i)$.
%Let $v\in V(X_i)$. 
By using switching operations, we may assume 
by Proposition~\ref{prop:tree_identity} that
$\psi(f)={\rm id}$ for every $f\in T$.
Observe then that 
$\langle Y_i\cup\{e\}\rangle_{\psi,i}=\langle \langle Y_i\rangle_{\psi,i}\cup \{\psi(e)\}\rangle$ 
and $\langle X_i\cup\{e\}\rangle_{\psi,i}=
\langle \langle X_i\rangle_{\psi,i}\cup \{\psi(e)\}\rangle$ by Proposition~\ref{lem:checking_label}.
We thus have 
\begin{align*}
\Delta(X,e)&=\mu\langle X_i\cup\{e\}\rangle-\mu\langle X_i\rangle \\
&=\mu(\langle \langle X_i\rangle_{\psi,i}\cup \{\psi(e)\}\rangle)-\mu(\langle X_i\rangle_{\psi,i}) \\
&=\mu(\langle X_i\rangle_{\psi,i} \cup \{\psi(e)\})-\mu(\langle X_i\rangle_{\psi,i}) \\
&\geq \mu(\langle Y_i\rangle_{\psi,i} \cup \{\psi(e)\})-\mu(\langle Y_i\rangle_{\psi,i}) \\
&=\mu(\langle \langle Y_i\rangle_{\psi,i}\cup \{\psi(e)\}\rangle)-\mu(\langle Y_i\rangle_{\psi,i}) \\
&=\mu\langle Y_i\cup\{e\}\rangle-\mu\langle Y_i\rangle=\Delta(Y,e), 
\end{align*}
where we used (\ref{eq:submodular_ineq})(\ref{eq:g1}) and  
the invariance of $\mu$ under closures.

\medskip

Case 2. Suppose that $e$ is a non-loop edge and at least one of $X_i\neq X_j$ or $X_i=X_j=\emptyset$ holds.
We further split the proof into subcases.

(2-i) If $Y_i=Y_j\neq \emptyset$, then, by (\ref{eq:g1}),
we have $\Delta(X,e)-\Delta(Y,e)=
\mu\langle X_i\cup X_j\cup\{e\}\rangle+1+\mu\langle Y_i\rangle-\mu\langle X_i\rangle-
\mu\langle X_j\rangle-\mu\langle Y_i\cup \{e\}\rangle$.
%Notice that $\lambda_{\mu}(A)\leq \lambda(B)$ holds for any connected $A$ and $B$ with $A\subseteq B$.
Since all these sets are connected or empty,
 $\mu\langle X_i\cup X_j\cup\{e\}\rangle\geq \mu\langle X_j\rangle$,
$\mu\langle Y_i\rangle\geq \mu\langle X_i\rangle$, and
$1\geq \mu\langle Y_i\cup\{e\}\rangle$.
Thus, 
$\mu\langle X_i\cup X_j\cup\{e\}\rangle+\mu\langle Y_i\rangle+1\geq \mu\langle X_i\rangle+\mu\langle X_j\rangle+\mu\langle Y_i\cup \{e\}\rangle$,
implying (\ref{eq:submo_proof1}).

(2-ii) If $Y_i\neq Y_j$ or $Y_i=Y_j=\emptyset$ holds,
then $e$ is a bridge connecting $X_i$ and $X_j$ in 
$X_i\cup X_j\cup\{e\}$ and is also a bridge connecting $Y_i$ and $Y_j$ in $Y_i\cup Y_j\cup \{e\}$.
By a switch operation, we may assume that $\psi(e)$ is identity.
Then, $\langle X_i\cup X_j\cup\{e\}\rangle_{\psi,i}=\langle \langle X_i\rangle_{\psi,i}\cup \langle X_j\rangle_{\psi,j}\rangle$.
This implies $\mu\langle X_i\cup X_j\cup\{e\}\rangle
=\mu(\langle X_i\rangle_{\psi,i}\cup \langle X_j\rangle_{\psi,j})$ by the invariance under closure.
Symmetrically, we have $\mu\langle Y_i\cup Y_j\cup \{e\}\rangle=
\mu(\langle Y_i\rangle_{\psi,i}\cup \langle Y_j\rangle_{\psi,j})$.
By using the submodularity and the monotonicity of $\mu$ over $\Gamma$, 
along with $X_k\subseteq Y_k$ for $k=1,2$, we have 
\begin{align*}
& \mu\langle X_i\cup X_j\cup \{e\}\rangle+\mu\langle Y_i\rangle+\mu\langle Y_j\rangle \\
&= \mu(\langle X_i\rangle_{\psi,i}\cup \langle X_j\rangle_{\psi,j})
+\mu(\langle Y_i\rangle_{\psi,i}) +\mu(\langle Y_j\rangle_{\psi,j}) \\
&\geq \mu(\langle X_i\rangle_{\psi,i}\cup \langle X_j\rangle_{\psi,j}\cup \langle Y_i\rangle_{\psi,i})+
\mu((\langle X_i\rangle_{\psi,i}\cup \langle X_j\rangle_{\psi,j})\cap \langle Y_i\rangle_{\psi,i})
+ \mu(\langle Y_j\rangle_{\psi,j}) \\
&\geq \mu(\langle Y_i\rangle_{\psi,i}\cup \langle X_j\rangle_{\psi,j})+
\mu(\langle X_i\rangle_{\psi,i})+ \mu(\langle Y_j\rangle_{\psi,j}) \\
&\geq \mu(\langle Y_i\rangle_{\psi,i}\cup \langle X_j\rangle_{\psi,j}\cup \langle Y_j\rangle_{\psi,j})+
\mu((\langle Y_i\rangle_{\psi,i}\cup \langle X_j\rangle_{\psi,j})\cap \langle Y_j\rangle_{\psi,j})+
\mu(\langle X_i\rangle_{\psi,i}) \\
&\geq \mu(\langle Y_i\rangle_{\psi,i}\cup \langle Y_j\rangle_{\psi,j})+
\mu(\langle X_j\rangle_{\psi,j})+
\mu(\langle X_i\rangle_{\psi,i}) \\
&=\mu \langle Y_i\cup Y_j\cup \{e\} \rangle+
\mu\langle X_j\rangle+
\mu\langle X_i\rangle .
\end{align*}
This implies (\ref{eq:submo_proof1}) by (\ref{eq:g1}).
\end{proof}

The aim of this paper is to extend the concept of the union of frame matroids.
We shall thus concentrate on
a function $\mu$ taking fractional values, that is, 
$\mu:2^{\Gamma}\rightarrow \{0,\frac{1}{d},\dots, \frac{d-1}{d},1\}$
for some finite positive integer $d$. 
As it is not integer-valued, 
$g_{\mu}$ does not induce a matroid in general, but if we  define $f_{\mu}:2^E\rightarrow \mathbb{Z}$ by
\begin{equation}
\label{eq:gain_rank_ex}
f_{\mu}(F)=dg_{\mu}(F) \qquad (F\subseteq E),
\end{equation}
%For simplicity we shall denote $f_{\mu}$ by $f$ in this section.
$f_{\mu}$ is a normalized integer-valued monotone submodular function by Theorem~\ref{thm:submodularity1}.
Thus $f_{\mu}$ induces a polymatroid $\mathbf{P}(f_{\mu})=(E,f_{\mu})$ and 
a matroid $\mathbf{M}(f_{\mu})=(E,f_{\mu}^{\mathbf{1}})$ on $E$.
%We will see special classes of $\mathbf{M}(f_{\mu})$ in several applications.

%We  remark that the independence of $\mathbf{M}(f_{\mu})$ can be described in the following simple manner.
%\begin{lemma}
%\label{lem:simple_counting}
%$F\subseteq E$ is independent in $\mathbf{M}(f_{\mu})$ if and only if 
%$|X|\leq d(|V(X)|-1+\mu\langle X\rangle)$ for any nonempty connected $X\subseteq F$.
%\end{lemma}

\begin{example}
\label{ex:gain1}
The frame matroid (or the union of copies) is a special case of $\mathbf{M}(f_{\mu})$,
where $\mu$ is defined by $\mu(X)=0$ for $X=\emptyset$ or $X={\rm id}$,
and otherwise $\mu(X)=1$.
In this case, (\ref{eq:gain_rank_ex}) is equal to (\ref{eq:gain_rank}). 
\end{example}

\begin{example}
\label{ex:gain2}
Let us consider a group $\Gamma$ equipped with a linear representation $\rho:\Gamma\rightarrow GL(\mathbb{F}^d)$
over a field $\mathbb{F}$.
Let
$d_{\rho}:2^{\Gamma}\rightarrow \mathbb{Z}$ be a function defined by
\[
 d_{\rho}(X)=\dim_{\dF} \{{\rm image}(I_d-\rho(\gamma))\mid \gamma\in X\} \qquad (X\subseteq \Gamma),
\]
where $d_{\rho}(\emptyset)=0$, $I_d$ denotes the identity matrix of size $d\times d$,
and ${\rm image}(I_d-\rho(\gamma))=\{(I_d-\rho(\gamma))x\mid x\in \dF^d\}$

It is easy to see that $d_{\rho}$ is monotone submodular
and  is invariant under conjugate.
Also, for any $\gamma_1, \gamma_2\in \Gamma$, we have 
${\rm image}(I_d-\rho(\gamma_1\gamma_2))\subseteq 
{\rm image}(I_d-\rho(\gamma_1)) + {\rm image}(I_d-\rho(\gamma_2))$ since
$(I_d-\rho(\gamma_1\gamma_2))=-(I_d-\rho(\gamma_1))(I_d-\rho(\gamma_2))+(I_d-\rho(\gamma_1))+(I_d-\rho(\gamma_2))$
and ${\rm image}((I_d-\rho(\gamma_1))(I_d-\rho(\gamma_2)))\subseteq {\rm image}(I_d-\rho(\gamma_1))$.
This implies the invariance of $d_{\rho}$ under closure.
Therefore, by setting $\mu=d_{\rho}/d$, we have another example of a symmetric polymatroidal function $\mu$.
The corresponding matroid will be extensively discussed in the next section.
\end{example}

\section{Matroids Induced by Group Representations}
\label{sec:dowling_extension}
Dowling geometries~\cite{dowling1973class} are special cases of frame matroids for finite groups,
which admit linear representations over finite fields $\mathbb{F}$.
In this section, we shall extend the union of Dowling geometries based on group representations.

\subsection{Dowling Geometries}
Suppose that $\Gamma$ is a nontrivial finite group and $n$ is a positive integer.
Define a $\Gamma$-gain graph $(K_n^{\bullet}(\Gamma),\psi^{\bullet})$ on $V(K_n^{\bullet}(\Gamma))=\{1,2,\dots, n\}$
such that (i)  for every $i,j$ with $1\leq i<j\leq n$ and every $\gamma\in \Gamma$, 
it has an edge from $i$ to $j$ with the gain $\gamma$
and (ii)  for each vertex $i$, it has a loop attached to $i$ with a gain $\gamma_i$, 
where $\gamma_i$ is any non-identity element of $\Gamma$.
%which is not contained in $\Gamma$ (i.e., $\gamma_i$ is just a sign of nonidentity 
%and has nothing to do with $\Gamma$). 
The {\em Dowling geometry} $\mathbf{D}_n(\Gamma)$ is defined by $\mathbf{F}(K_n^{\bullet}(\Gamma),\psi^{\bullet})$,
the frame matroid of $(K_n^{\bullet}(\Gamma),\psi^{\bullet})$.

A remarkable property of $\mathbf{D}_n(\Gamma)$ is that, for $n\geq 3$, $\mathbf{D}_n(\Gamma)$ is representable over $\mathbb{F}$ if 
and only if $\Gamma$ is isomorphic to a subgroup of 
$\mathbb{F}^{\times}$~(see, e.g., \cite[Theorem 6.10.10]{oxley}).
The proof of one direction indicates the following explicit construction of representations.

Suppose that $\Gamma$ is isomorphic to a subgroup of $\mathbb{F}^{\times}$.
For a simpler description, we assume that $\Gamma$ is itself  a subgroup of $\mathbb{F}^{\times}$.
With each $e=(i,j)\in E(K_n^{\bullet}(\Gamma))$, we associate a vector  $x_e\in \mathbb{F}^{V}$
defined by 
\begin{equation*}
\label{eq:dowling_rep}
x_e(v)=\begin{cases}
-\psi(e) & \text{ if } v=i \\
1 & \text{ if } v=j \\
0 & \text{ otherwise}
\end{cases}
\end{equation*}
if $e$ is not a loop attached at $i$, and
\begin{equation*}
\label{eq:dowling_rep_loop}
x_e(v)=\begin{cases}
1 & \text{ if } v=i \\
0 & \text{ otherwise}
\end{cases}
\end{equation*}
if $e$ is a loop.
These give us a linear representation of $\mathbf{D}_n(\Gamma)$ over $\mathbb{F}$ 
(see e.g.,~\cite[Lemma 6.10.11]{oxley}),
which is called the {\em canonical representation}~\cite{zaslavsky2003biased}.
As each $\Gamma$-gain graph $(G,\psi)$ can be considered as a subgraph of $(K_n^{\bullet}(\Gamma),\psi^{\bullet})$, 
the restriction to $E(G)$ leads to the canonical representation of $\mathbf{F}(G,\psi)$.

Equivalently, instead of a vector assignment, we may associate a 1-dimensional linear space
\begin{equation}
D_e=\left\{x\in \mathbb{F}^V\Big\vert 
\begin{array}{l}
x(i)+\psi(e)x(j)=0, \\ 
x(V\setminus \{i,j\})=0 
\end{array}
\right\}
\end{equation}
with each non-loop edge $e=(i,j)$, 
and 
\begin{equation}
D_e=\{x\in \mathbb{F}^V\mid x(V\setminus \{i\})=0\}
\end{equation}
with a loop $e$ attached to $i$,
where, for $W\subseteq V$, $x(W)=0$ implies $x(k)=0$ for all $k\in W$.
Then, the union of $d$ copies of $\mathbf{D}_n(\Gamma)$ can be obtained 
in a systematic manner, by just following the technique mentioned in \textsection\ref{subsec:linear_union}.

To see more detail, 
let us consider the direct sum of $d$ copies of $\mathbb{F}^V$, which results in $(\mathbb{F}^d)^V$.
Then, the associated vector space with each edge $e=(i,j)$ becomes 
a $d$-dimensional space in $(\mathbb{F}^d)^V$ written by
\begin{equation}
\label{eq:dowling_union}
D_e^d= \left\{x\in (\mathbb{F}^d)^V\Big\vert 
\begin{array}{l}
x(i)+\psi(e)x(j)=0, \\
x(V\setminus \{i,j\})=0
\end{array}
\right\}
\end{equation}
and 
\begin{equation}
D_e^d=\{x\in (\mathbb{F}^d)^V\mid x(V\setminus \{i\})=0\}.
\end{equation}
By extension of scalars for the underlying vector space from $\mathbb{F}$ to $\mathbb{K}$, 
$D_e^d$ is contained in $\mathbb{K}^d$ and 
we can take a representative vector $x_e^d$  from each $D_e^d$ in generic position.
By Proposition~\ref{prop:linear_union}, we obtain a linear representation $e=(i,j)\mapsto x_e^d\in (\mathbb{K}^d)^V$ 
of the union of $d$ copies of Dowling geometry $\mathbf{D}_n(\Gamma)$, where each vector written by
\begin{equation}
\label{eq:dowling_union_rep}
x_e^d(v)=\begin{cases}
-\psi(e)\alpha_e & \text{ if } v=i \\
\alpha_e & \text{ if } v=j \\
0 & \text{ otherwise}
\end{cases}
\end{equation}
\begin{equation}
\label{eq:dowling_union_rep_loop}
x_e^d(v)=\begin{cases}
\alpha_e & \text{ if } v=i \\
0 & \text{ otherwise},
\end{cases}
\end{equation}
depending on whether $e=(i,j)$ is a non-loop or a loop, where $\alpha_e=(\alpha_e^1,\dots, \alpha_e^d)^{\top}\in \mathbb{K}^d$
such that $\{\alpha_e^i \colon 1\leq i\leq d, e\in E(K_n^{\bullet}(\Gamma))\}$ is algebraically 
independent over $\mathbb{F}$.
We shall extend this construction in the next subsection.

\medskip

In the subsequent discussion, we will frequently us the following (more or less) known fact about graphic matroids.
Consider a $\Gamma$-gain graph $(G=(V,E),\psi)$ such that $\psi(e)$ is identity for every $e\in E$.
Then, $\mathbf{F}(G,\psi)$ is just the graphic matroid of $G$.
%The sum of $d$ copies of the graphic matroids is the polymatroid $(E,dg)$ with the rank function $dg(F)=d|V(F)|-dc(F)$ for $F\subseteq E$.
The above result on linear representations of frame matroids implicitly implies the following  fact 
on the linear representation of the sum of $d$ copies of the graphic matroid, 
which will be frequently used in the subsequent discussion.
\begin{lemma}
\label{lem:graphic}
%Let $\Gamma$ be a subgroup of $\mathbb{F}^{\times}$ and 
Let $(G=(V,E),\psi)$ be a gain graph such that $\psi(e)$ is identity for every $e\in E$.
Suppose that $E$ is connected.
Then, the following holds.
\begin{itemize}
\item $\dim_{\mathbb{F}} \{D_e^d\mid e\in E\}=d|V|-d$;
\item For any $x\in (\mathbb{F}^d)^V$ with $x(V\setminus \{i,j\})=0$,  $x\in \spa\{D_e^d\mid e\in E\}$ if and only if
$x(i)+x(j)=0$.
\end{itemize}
\end{lemma}
\begin{proof}
%We may assume that $\psi:E\rightarrow \Gamma$ 
We may assume that $\Gamma$ is a subgroup of $\dF^{\times}$.
The first part directly follows from the above discussion on frame matroids and the canonical linear representation.
Indeed, since $D_e^d$ is the direct sum of $d$ copies of $D_e$,
we have 
$\dim_{\mathbb{F}}\{D_e^d\mid e\in E\}=
d \dim_{\mathbb{F}}\{D_e\mid e\in E\}=
d(|V|-c(E)+\alpha_\Gamma(E))=d|V|-d$.

The second part also follows from the above discussion.
We first consider the case of $d=1$.
Let $e^*=(i,j)$ be a new edge from $i$ to $j$ with the gain $\psi(e^*)=g$ for some $g\in \mathbb{F}\setminus \{0\}$,
and let $(G^*,\psi)$ be the gain graph obtained from $(G,\psi)$ by adding $e^*$.
In the canonical representation of $\mathbf{F}(G^*,\psi)$,
$e^*$ is associated with a vector $x_{e^*}\in \mathbb{F}^V$ with 
$x_{e^*}(V\setminus \{i,j\})=0$, $x_{e^*}(j)=1$ and $x_{e^*}(i)=-g$.
%We then insert a new edge $f=(i,j)$ with $\psi(f)=a$.

Since $E$ is connected, $E\cup \{e^*\}$ has a cycle passing through $e^*$.
We thus have $\langle E\cup \{e^*\}\rangle_{\psi,v}=\langle g\rangle$ for any $v\in V$ 
by Proposition~\ref{lem:checking_label}. 
This implies that 
\begin{equation*}
\dim_{\mathbb{F}} \{D_e\mid e\in E\cup \{e^*\}\}=g_{\Gamma}(E\cup\{e^*\})=
\begin{cases}
|V|-1 & \text{ if } g=1 \\
|V| & \text{ otherwise}.
\end{cases}
\end{equation*}
Therefore,
$\spa\{D_e\mid e\in E\}$ contains $x\in \mathbb{F}^V\setminus \{0\}$ 
with  $x(V\setminus \{i,j\})=0$ and $x(i)+gx(j)=0$ if and only if $g=1$.
Equivalently, $\spa\{D_e\mid e\in E\}$ contains $x\in \mathbb{F}^V\setminus \{0\}$ 
with  $x(V\setminus \{i,j\})=0$  if and only if $x(i)+x(j)=0$.

Thus, if we consider the direct sum of $d$ copies of $\dF^d$,
we conclude that  
$\spa\{D_e^d\mid e\in E\}$ contains $x\in (\mathbb{F}^d)^V\setminus \{0\}$ 
with  $x(V\setminus \{i,j\})=0$ if and only if  $x(i)+x(j)=0$.
\end{proof}

\subsection{Linear matroids induced by group representations}
\label{subsec:dowling_linear}
In this section, we shall extend the representation theory of 
the union of Dowling geometries. 
The idea of our construction is that, 
instead of coefficients $\psi(e)\in \mathbb{F}^{\times}$ of $\alpha_e$ in (\ref{eq:dowling_union_rep}),
we shall make use of linear representations of groups.
We then have a linear matroid induced by a group $\Gamma$, where $\Gamma$ is not restricted to finite abelian groups.
We show that resulting linear matroids are special cases of matroids given in 
\textsection\ref{sec:fractional_lifting}.

Let $\Gamma$ be a group equipped with a linear representation 
$\rho:\Gamma\rightarrow GL(\mathbb{F}^d)$ 
on a vector space of finite dimension $d$ over a field $\mathbb{F}$.
Let $(G=(V,E),\psi)$ be a finite $\Gamma$-gain graph.

As in the previous subsection, let $\mathbb{K}$ be an extension of $\mathbb{F}$ 
that contains an algebraically independent set
$\{\alpha_e^i\mid i=1,\dots, d, e\in E\}$  over $\mathbb{F}$,
and let $\alpha_e=(\alpha_e^1,\dots, \alpha_e^d)^{\top}\in \mathbb{K}^d$.

With each $e=(i,j)\in E$,
we assign a vector $x_{e,\psi}\in (\mathbb{K}^d)^V$ 
defined by 
\begin{equation}
\label{eq:dowling_new_rep}
x_{e,\psi}(v)=\begin{cases}
-\rho(\psi(e))\alpha_e & \text{ if } v=i \\
\alpha_e & \text{ if } v=j \\
0 & \text{ otherwise}
\end{cases}
\end{equation}
if $e=(i,j)$ is not a loop, and 
\begin{equation}
\label{eq:dowling_new_rep_loop}
x_{e,\psi}(v)=\begin{cases}
(I_d-\rho(\psi(e)))\alpha_e & \text{ if } v=i \\
0 & \text{ otherwise},
\end{cases}
\end{equation}
if $e$ is a loop.
%Then we are interested in the linear matroid on $E$ represented by $x_e$.
The linear matroid induced on $\{x_{e,\psi} \mid e\in E\}$ is denoted by $\mathbf{D}_{\rho}(G,\psi)$.

Note that $\mathbf{D}_{\rho}(G,\psi)$ is the generic matroid obtained from a linear polymatroid with 
a linear representation $e\mapsto A_{e,\psi}$ defined by
\begin{equation}
\label{eq:A_non_loop}
A_{e,\psi}=\left\{x\in (\mathbb{F}^d)^V\Big\vert  
\begin{array}{l}
x(i)+\rho(\psi(e)) x(j)=0, \\ 
x(V\setminus \{i,j\})=0
\end{array}
\right\}
\end{equation}
for each non-loop edge $e\in E$, and 
\begin{equation}
\label{eq:A_loop}
A_{e,\psi}=\left\{x\in (\mathbb{F}^d)^V\Big\vert 
\exists \alpha\in \mathbb{F}^d\colon
\begin{array}{l} 
x(i)=(I_d-\rho(\psi(e)))\alpha, \\
x(V\setminus \{i\})=0
\end{array}
\right\}
\end{equation}
for a loop $e$, by extending the underlying field from $\mathbb{F}$ to $\mathbb{K}$.
Note also that $A_{e,\psi}$ is invariant with the choice of orientation of each edge, as each $\rho(\psi(e))$ is invertible.

Let  $\mathbf{DP}_{\rho}(G,\psi)$ be  the linear polymatroid on $E$ represented by
$e\mapsto A_{e,\psi}$.
Clearly, each $A_{e,\psi}$ depends on the gain function $\psi$,
but, as shown below, the associated polymatroid is actually invariant up to equivalence.

\begin{lemma}
\label{lem:linear_equivalence}
Let $\psi$ and $\psi'$ be equivalent gain functions.
Then,
\[
 \dim_{\mathbb{F}} \{A_{e,\psi}\mid e\in E \} = \dim_{\mathbb{F}} \{A_{e,\psi'}\mid e\in E \}.
\]
\end{lemma}
\begin{proof}
Let us simply denote $d_{\psi}=\dim_{\mathbb{F}}\{A_{e,\psi}\mid e\in E\}$.
It is sufficient to show that $d_{\psi}$ is invariant from any switch operation.

Suppose that $\psi'$ is obtained from $\psi$ by a switch operation at $v\in V$ with 
$\gamma\in \Gamma$. 
Since $A_{e,\psi}$ is invariant with the choice of the edge orientation, we may assume that 
all of edges are oriented from $v$.
Then, $\psi'(e)=\gamma\psi(e)$ if $e$ is incident to $v$, $\psi'(e)=\gamma\psi(e)\gamma^{-1}$ if $e$ is a loop at $v$,
and otherwise $\psi'(e)=\psi(e)$.

Consider a bijective linear transformation 
$T:(\mathbb{F}^d)^{V}\rightarrow (\mathbb{F}^d)^{V}$ defined by,
for each $x\in (\mathbb{F}^d)^{V}$,
\begin{equation*}
T(x)(w)=\begin{cases}
x(w) & \text{ if } w\in V\setminus \{v\} \\
\rho(\gamma) x(v) & \text{ if } w=v. 
\end{cases}
\end{equation*}
We then have 
$x(v)=\rho(\gamma)^{-1}T(x)(v)$ and $x(w)=T(x)(w)$ for  $w\in V\setminus \{v\}$.
Therefore, if $e$ is a non-loop edge oriented from $v$ to a vertex $j\in V$,
\begin{align*}
TA_{e,\psi}&=\{T(x) \in (\mathbb{F}^d)^V \mid x(v)+\rho(\psi(e))x(j)=0, x(V\setminus \{v,j\})=0\} \\
&=\{y \in (\mathbb{F}^d)^V \mid y(v)+\rho(\gamma \psi(e))y(j)=0, y(V\setminus \{v,j\})=0\}.
\end{align*}
As $\psi'(e)=\gamma \psi(e)$, 
we obtain that $TA_{e,\psi}=A_{e,\psi'}$.
Similarly, if $e$ is a loop attached to $v$,
\begin{align*}
TA_{e,\psi}&=\{y \in (\mathbb{K}^d)^V \mid \exists\alpha\in \mathbb{K}^d\colon \rho(\gamma^{-1})y(v)=(I_d-\rho(\psi(e)))\alpha, y(V\setminus \{v\})=0\} \\
&=\{y \in (\mathbb{K}^d)^V \mid \exists\alpha\in \mathbb{K}^d\colon y(v)=(I_d-\rho(\gamma\psi(e)\gamma^{-1}))\alpha, y(V\setminus \{v\})=0\} \\
&=A_{e,\psi'},
\end{align*}
where $\psi'(e)=\gamma\psi(e)\gamma^{-1}$.

If $e$ is not incident to $v$, then we clearly have $TA_{e,\psi}=A_{e,\psi}=A_{e,\psi'}$.
Thus $d_{\psi}$ is invariant from any switch operation.
\end{proof}

%By Lemma~\ref{lem:linear_equivalence}, $\mathbf{DP}_{\rho}(G,\psi)$ 
%is invariant with respect to the choice of equivalent gain functions $\psi$.
%Since the matroid  $\mathbf{D}_{\rho}(G,\psi)$ is obtained from $\mathbf{DP}_{\rho}(G,\psi)$,
%we also found that  $\mathbf{D}_{\rho}(G,\psi)$ is invariant up to equivalent gain functions $\psi$.

\subsection{Combinatorial characterization}
\label{subsec:dowling_combinatorial}
We now show that the linear matroid $\mathbf{D}_{\rho}(G,\psi)$ is indeed equal to 
a special case of matroids given in \textsection\ref{subsec:fractional_lifting}.
Recall that, in \textsection\ref{subsec:fractional_lifting} (Example~\ref{ex:gain2}), 
$d_{\rho}:2^{\Gamma}\rightarrow \mathbb{Z}$ is given  by
\begin{equation}
\label{eq:d_rho_group}
d_{\rho}(X)=\dim_{\dF}\{{\rm image}(I_d-\rho(\gamma))\mid \gamma\in X\} \qquad (X\subseteq \Gamma).
\end{equation}
Let $(G=(V,E),\psi)$ be a $\Gamma$-gain graph.
As a special case of $f_{\mu}$ given in (\ref{eq:gain_rank_ex}),
we shall define a set function $f_{\rho}$ on $E$  by
\begin{align}
\label{eq:f_rho}
f_{\rho}(F)&=d|V(F)|-dc(F)+\sum_{X\in C(F)}d_{\rho}\langle X\rangle \qquad (F\subseteq E).
\end{align}
By Theorem~\ref{thm:submodularity1}, $f_{\rho}$ is 
a normalized integer-valued monotone submodular function, and thus 
$\mathbf{P}(f_{\rho})=(E,f_{\rho})$ is a polymatroid and 
$\mathbf{M}(f_{\rho})=(E,f_{\rho}^{\mathbf{1}})$ is a matroid on $E$.

We are now ready to state our main theorems.
The first theorem asserts the equivalence of $\mathbf{P}(f_{\rho})$ and $\mathbf{DP}_{\rho}(G,\psi)$,
while the second implies the equivalence of $\mathbf{M}(f_{\rho})$ and $\mathbf{D}_{\rho}(G,\psi)$.
\begin{theorem}
\label{thm:main_poly}
Let $\Gamma$ be a group equipped with a linear representation 
$\rho: \Gamma\rightarrow GL(\mathbb{F}^d)$.
Let $(G=(V,E),\psi)$ be a $\Gamma$-gain graph.
Then, 
\begin{equation}
f_{\rho}(E)=\dim_{\mathbb{F}}\{A_{e,\psi}\mid e\in E\}.
\end{equation}
\end{theorem}
\begin{proof}
For any $F\subseteq E$, let $A_F=\spa\{A_{e,\psi}\mid e\in F\}$.
By definition, it is easy to check that $A_E=\bigoplus_{X\in C(E)}A_X$.
Moreover, $f_{\rho}(E)=\sum_{X\in C(E)}f_{\rho}(X)$.
Hence, it suffices to show the statement when $G$ is connected.

Let $T$ be a spanning tree in $E$.
By Proposition~\ref{prop:tree_identity} and Lemma~\ref{lem:linear_equivalence}, 
we may assume that $\psi(e)$ is identity for $e\in T$.
By Proposition~\ref{lem:checking_label},
$\langle E\rangle_{\psi,v}=\langle \psi(e)\mid e\in E\setminus T\rangle$ for any $v\in V$.
Hence, as $\mu$ is invariant under taking closure, $\mu\langle E \rangle=\mu(\{\psi(e)\mid e\in E\setminus T \})$.
Thus,
\[
 f_{\rho}(E)=d|V|-d+\dim_{\dF}\{{\rm image}(I_d-\rho(\psi(e)))\mid e\in E\setminus T\}.
\]

By Lemma~\ref{lem:graphic},
we have that (i) $\dim_{\dF} A_T=d|V|-d$
and (ii) for any $i,j\in V$ and any $x\in (\mathbb{F}^d)^V$ with $x(V\setminus \{i,j\})=0$, 
$x\in A_T$ if and only if $x(i)+x(j)=0$. 
This means that each quotient space $A_{e,\psi}/A_T$ for $e=(i,j)\in E\setminus T$ is written by 
\begin{align*}
A_{e,\psi}/A_T=\{x+A_T \mid \exists\alpha\in \mathbb{F}^d\colon x(j)=(I_d-\rho(\psi(e)))\alpha, x(V\setminus \{j\})=0 \} \\
=\{x+A_T \mid \exists\alpha\in \mathbb{F}^d\colon x(1)=(I_d-\rho(\psi(e)))\alpha, x(V\setminus \{1\})=0 \} 
\end{align*}
where $1$ denotes one specific vertex in $V$.
%Thus, $A_{e,\psi}/A_T$ is isomprphic to  $\im(\rho(\psi(e))-I_d)$.
%\begin{align*}
%\{x\in (\mathbb{F}^d)^V \mid \exists\alpha\in \mathbb{F}^d\colon x(1)=(\rho(\psi(e))-I_d)\alpha, x(V\setminus \{1\})=0 \},
%\end{align*}
%which is also isomorphic to $\im(\rho(\psi(e))-I_d)$.
Therefore, $\spa\{A_{e,\psi}/A_T\mid e\in E\setminus T\}$ is isomorphic to
\begin{align*}
{\rm span} \{x\in (\mathbb{F}^d)^V \mid \exists e\in E\setminus F, \exists\alpha\in \mathbb{F}^d\colon 
x(1)=(I_d-\rho(\psi(e)))\alpha, x(V\setminus \{1\})=0 \}
\end{align*} 
which is isomorphic to $\spa \{{\rm image}(I_d-\rho(\psi(e)))\mid e\in E\setminus T\}$. 
%Consequently, we obtain
%\begin{equation*}
%\dim(\spa\{A_{e,\psi}/A_T\mid e\in E\setminus T\})=\dim(\spa\{{\rm im}(I_d-\rho(\psi(e)))\mid e\in E\setminus T\}).
%\end{equation*}
Thus, we obtain $\dim_{\dF}\{A_{e,\psi}\mid e\in E\}=
\dim_{\dF} A_T+\dim_{\dF}\{A_{e,\psi}/A_T\mid e\in E\setminus T\}
=d|V|-d+\dim_{\dF}\{{\rm image}(I_d-\rho(\psi(e)))\mid e\in E\setminus T\}=f_{\rho}(E)$,
completing the proof of Theorem~\ref{thm:main_poly}.
\end{proof}

\begin{theorem}
\label{thm:main_rank}
Let $\Gamma$ be a group equipped with a linear representation 
$\rho: \Gamma\rightarrow GL(\mathbb{F}^d)$.
Let $(G=(V,E),\psi)$ be a $\Gamma$-gain graph.
Then, $f^{\mathbf{1}}_{\rho}$ is the rank function of $\mathbf{D}_{\rho}(G,\psi)$.
\end{theorem}
\begin{proof}
By Theorem~\ref{thm:main_poly}, we have
$f_{\rho}(F)=\dim_{\dK}{\rm span}\{A_{e,\psi}\mid e\in F\}$ for any $F\subseteq E$
(by restricting the statement to the graph $(V,F)$).
Since $x_{e,\psi}$ of (\ref{eq:dowling_new_rep})(\ref{eq:dowling_new_rep_loop}) is taken from $A_{e,\psi}$
so that $\{x_{e,\psi}\mid e\in E\}$ is in generic position,
Theorem~\ref{theorem:flat_matroid} implies that 
the rank of $F\subseteq E$ in $\mathbf{D}_{\rho}(G)$ is equal to
$\min \{|X|+\dim_{\mathbb{K}}\{A_{e,\psi}\mid e\in F\setminus X\}\mid X\subseteq F\}$,
which is equal to  $f^{\mathbf{1}}_{\rho}(F)$ 
by definition (\ref{eq:rist}).
\end{proof}

\begin{corollary}
\label{cor:main_matroid} 
Let $\Gamma$ be a group equipped with a linear representation 
$\rho: \Gamma\rightarrow GL(\mathbb{F}^d)$.
Let $(G=(V,E),\psi)$ be a $\Gamma$-gain graph.
Then, 
$\{ x_{e,\psi}\mid e\in E \})$ is linearly independent if and only if  
$|F|\leq d|V(F)|-dc(F)+\sum_{X\in C(F)}d_{\rho}\langle X\rangle$ for any  $F\subseteq E$.
\end{corollary}

\section{Applications}
\label{sec:applications}
As applications,
we shall address two problems from discrete geometry.
In the first problem we discuss the symmetric redrawing problem of 
symmetrically embedding graphs, called symmetric frameworks.
We shall extend Whiteley's  parallel redrawing theorem to symmetric setting.
In the second problem we discuss the symmetry-forced rigidity of symmetric frameworks
and extend Laman's theorem concerning the rigidity of graphs.  

The section is organized as follows.
We shall first introduce notions of symmetric graphs and symmetric frameworks 
in \textsection\ref{subsec:symmetric_graphs} 
and \textsection\ref{subsec:symmetric_framework}, respectively.
Then, we will discuss the symmetric parallel redrawing problem in \textsection\ref{subsec:drawing}
and the symmetry-forced rigidity in \textsection\ref{subsec:rigidity}.

%and In this section, we shall present define the notion of  symmetric infinitesimal rigidity,
% introduced by Schulze and Whiteley~\cite{schulze2011orbit}.
%In \textsection~\ref{subsec:symmetric_graphs}, 
%we first introduce ${\cal S}$-symmetric graphs, whose automorphism group has 
%a subgroup isomorphic to ${\cal S}$.
%In \textsection~\ref{subsec:inf}, we shall review the conventional notion of infinitesimal rigidity.
%In \textsection~\ref{subsec:symmetric_inf}, we introduce the symmetric infinitesimal rigidity of 
%symmetric frameworks, which only concerns with infinitesimal motions invariant from ${\cal S}$.
%In \textsection~\ref{subsec:orbit_rigidity}, we shall introduce orbit rigidity matrices,
%which are main tools for investigating symmetric infinitesimal rigidity in the subsequent sections.
%In \textsection~\ref{subsec:necessity}, we prove a necessary condition for 
%symmetric frameworks to be symmetrically infinitesimally rigid.

\subsection{Symmetric graphs}
\label{subsec:symmetric_graphs}

Let $H$ be a simple graph, which may not be finite.
An {\em automorphism} of $H$ is a permutation $\pi:V(H)\rightarrow V(H)$ such that
$\{u,v\}\in E(H)$ if and only if $\{\pi(u),\pi(v)\}\in E(H)$.
The set of all automorphisms of $H$ forms a subgroup of the symmetric group of $V(H)$, 
known as the {\em automorphism group} ${\rm Aut}(H)$ of $H$.
An {\em action} of a group $\Gamma$ on $H$ is a group homomorphism $\theta:\Gamma \rightarrow {\rm Aut}(H)$.
An action $\theta$ is called {\em free} if  $\theta(\gamma)(v)\neq v$ for any $v\in V$ and  
any non-identity $\gamma\in \Gamma$.
We say that a graph $H$ is {\em $(\Gamma,\theta)$-symmetric} (or simply {\em $\Gamma$-symmetric}) 
if $\Gamma$ acts on $H$ by $\theta$.
In the subsequent discussion, we only consider free actions, and we 
omit to specify the action $\theta$, if it is clear from the context.
We then denote $\theta(\gamma)(v)$ by  $\gamma v$. 

%Note that, for any $g\in {\cal S}$ and $u,v\in V$,
%\begin{equation}
%\label{eq:symmetry_vertex}
%\{u,v\}\in E(\Gamma) \ \Longleftrightarrow \ \{gu,gv\}\in E(\Gamma),
%\end{equation}
%and hence there is an induced  action of ${\cal S}$ on $E(\Gamma)$ defined by $g\cdot \{u,v\}=\{gu,gv\}$.

For a $\Gamma$-symmetric graph $H$,
the {\em quotient graph} $H/\Gamma$ is a multigraph on the set $V(H)/\Gamma$ 
of vertex orbits, together with the set $E(H)/\Gamma$ of edge orbits as the edge set (with respect to $\theta$). 
An edge orbit may be represented by a loop in $H/\Gamma$.
Figure~\ref{fig:dihedral_group} illustrates an example when $\Gamma$ is the dihedral group of order $4$.

Several distinct graphs may have the same quotient graph. 
However, if we assume that the underlying action is free, 
then a gain labeling  makes the relation one-to-one.
To see this, we arbitrary choose a vertex $v$ as a representative vertex from each vertex orbit.
Then, each orbit is written by $\Gamma v=\{gv\mid g\in \Gamma\}$.
If the action is free, 
an edge orbit connecting $\Gamma u$ and $\Gamma v$ in $H/\Gamma$ 
can be written by $\{\{gu,ghv\}\mid g\in \Gamma \}$ 
for a unique $h\in \Gamma$.
We then orient the edge orbit from $\Gamma u$ to $\Gamma v$ in $H/\Gamma$
and assign to it gain $h$.
In this way, we obtain {\em the quotient $\Gamma$-gain graph}, denoted by $(H/\Gamma,\psi)$. 

Conversely, let $(G,\psi)$ be a finite $\Gamma$-gain graph for a group $\Gamma$.
We simply denote the pair $(g,v)$ of $g\in \Gamma$ and $v\in V(G)$  by $gv$.
The {\em covering graph} (also known as the derived graph) of $(G,\psi)$
is the simple graph with the vertex set $\Gamma\times V(G)=\{gv\mid g\in \Gamma, v\in V(G)\}$
and the edge set $\{\{gu,g\psi(e)v\}\mid e=(u,v)\in E(G), g\in \Gamma\}$.
%We denote it by $\Gamma G$.

Clearly, $\Gamma$ freely acts on the covering graph with the action 
$\theta$ defined by $\theta(g):v \mapsto gv$ for $g\in \Gamma$,
under which the quotient graph comes back to $(G,\psi)$.
In this way, there is a one-to-one correspondence between $\Gamma$-gain graphs and $\Gamma$-symmetric graphs with free actions.
For more properties of covering graphs, see e.g.,~\cite{Biggs,GrossTucker}.

\begin{figure}[th]
\centering
\begin{minipage}{0.4\textwidth}
\centering
\includegraphics[scale=0.5]{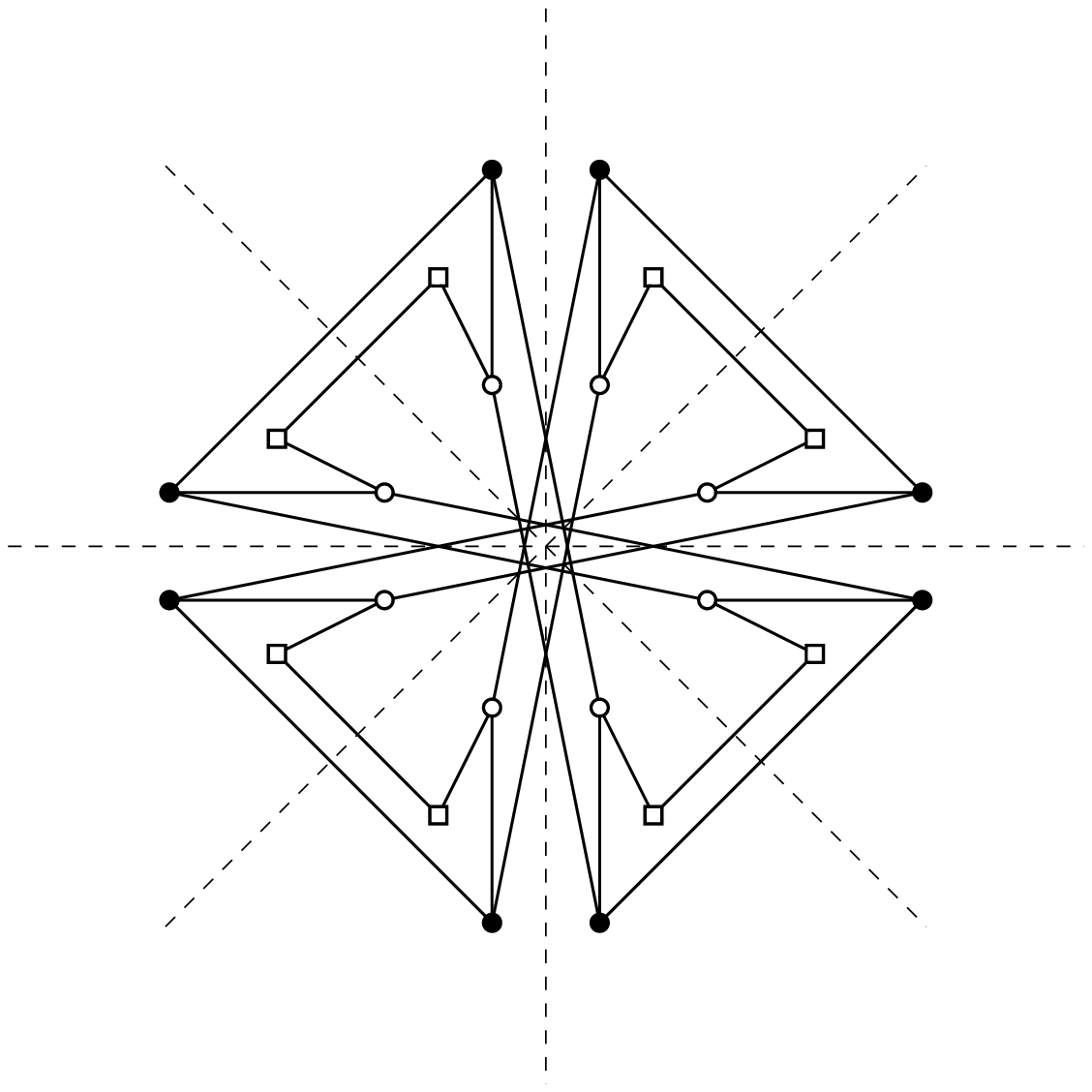}
\end{minipage}
%\begin{minipage}{0.4\textwidth}
%\centering
%\includegraphics[scale=0.49]{dihedral_graph2.eps}
%\end{minipage}
%\begin{minipage}{0.24\textwidth}
%\centering
%\includegraphics[scale=0.5]{dihedral_graph3.eps}
%\end{minipage}
 \begin{minipage}{0.24\textwidth}
\centering
\includegraphics[scale=0.5]{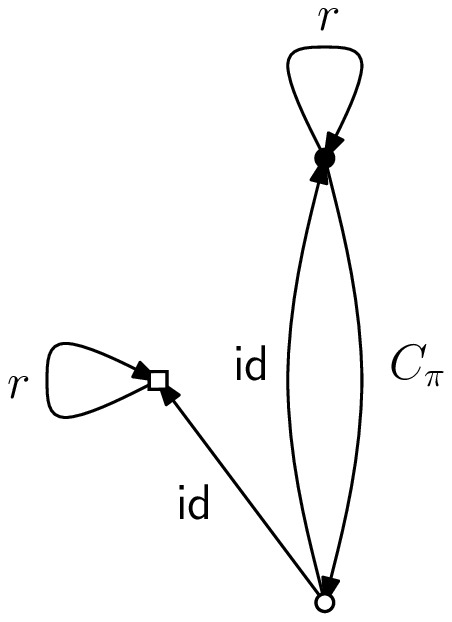}
\end{minipage}
\caption{A symmetric graph and the quotient gain graph with dihedral group symmetry of order $4$.}
\label{fig:dihedral_group}
\end{figure}

\subsection{Finite Symmetric frameworks}
\label{subsec:symmetric_framework}
A {\em $d$-dimensional framework} (or, simply, a framework) is a pair $(H,p)$
of a simple undirected graph $H$ and a mapping $p:V(H)\rightarrow \mathbb{R}^d$, called a {\em point-configuration},
which may be regarded as a straight-line realization of $G$ in $\mathbb{R}^d$.
%We denote the set $\{p(v)\colon v\in V(H)\}$ of points by $p(\Gamma)$.
In this paper, we are interested in symmetrically embedded  symmetric graphs in the Euclidean space.
Thus, throughout applications in \textsection\ref{sec:applications}, $\Gamma$ denotes a subgroup of matrix group $GL(\mathbb{R}^d)$.

Let $H$ be a  $(\Gamma,\theta)$-symmetric graph, where $\Gamma$ freely acts on $H$ through $\theta$.
A function $f:V(H)\rightarrow \mathbb{R}^d$ is called {\em $(\Gamma,\theta)$-symmetric} 
(or simply {\em {\em $\Gamma$}-symmetric}) if
\begin{equation}
\label{eq:symmetric_func}
\gamma f(v)=f(\gamma v) \qquad \forall \gamma\in \Gamma \text{ and } \forall v\in V(H).
\end{equation}
The pair $(H, p)$ is said to be a {\em $(\Gamma,\theta)$-symmetric framework} 
(or simply {\em $\Gamma$-symmetric framework})
if $H$ and $p$ are $(\Gamma,\theta)$-symmetric.
 
It is convenient to fix a representative vertex $v$ of each vertex orbit $\Gamma v$,
and define the {\em quotient} $f/\Gamma:V/\Gamma\rightarrow \mathbb{R}^d$ of $f$ to be 
$f(v)=f/\Gamma(\Gamma v)$ for each representative vertex $v$. 

In the subsequent discussion, we shall only consider finite frameworks, where $H$ is a finite graph.
(In \textsection\ref{sec:further}, we will discuss infinite frameworks with crystallographic symmetry.)
Namely, we shall restrict our attention to {\em discrete point groups} ${\cal P}$,
which are finite discrete subgroups of the {\em orthogonal group} ${\cal O}(\mathbb{R}^d)$,
i.e., the set of $d\times d$ orthogonal matrices.

For a discrete point group $\Gamma$, let $\mathbb{Q}_{\Gamma}$ be the field
 generated by $\mathbb{Q}$ and the entries of matrices contained in $\Gamma$.
For a $\Gamma$-gain graph $(G,\psi)$, a mapping $f:V(G)\rightarrow \mathbb{R}^d$ is said to be 
{\em $\Gamma$-generic} if the set of coordinates of the image of $f$ is algebraically independent over $\mathbb{Q}_{\Gamma}$.
Also, for a $\Gamma$-symmetric graph $H$, 
a $\Gamma$-symmetric function $f:V(H)\rightarrow \mathbb{R}^d$ is said to be 
{\em $\Gamma$-generic} if $f/\Gamma$ is $\Gamma$-generic.

\subsection{Symmetric parallel redrawing problem}
\label{subsec:drawing}
\subsubsection{Parallel redrawing}
Let $(H,p)$ be a finite $d$-dimensional framework.
We shall consider $(H,p)$ as a {\em drawing} of the graph $H$ in $\mathbb{R}^d$ with straight-line edges. 

A framework $(H,q)$ is called a {\em parallel redrawing} of $(H,p)$ if
$q(i)-q(j)$ is parallel to $p(i)-p(j)$ for all $\{i,j\}\in E(H)$.
No matter how the underlying graph is dense, any framework  admits parallel redrawings,
since  a translation of $(H,p)$ or a dilation of $(H,p)$ is always a redrawing.
A drawing $(H,p)$ is said to be {\em robust} if 
any redrawing of $(H,p)$ is a consequence of translations and dilation of $(H,p)$.
In the parallel redrawing problem, we are asked whether $(H,p)$ is robust or not.

In the context of rigidity theory, this concept is known as the {\em direction-rigidity} of 
$d$-dimensional bar-joint frameworks $(H,p)$,
where we are interested in direction-constraint, rather than 
conventional length-constraint (which we will discuss in the next subsection), 
or the mixture of length and direction constraints 
(see, e.g.,~\cite{Whitley:1997,jackson2010globally,Servatius:1999}).

Let us take a look at the formal definition.
We define a {\em relocation} of $(H,p)$ by $m:V(H)\rightarrow \mathbb{R}^d$ 
such that 
\begin{align*}
\text{$m(i)-m(j)$ is parallel to  $p(i)-p(j)$} \qquad \forall \{i,j\}\in E(H), 
\end{align*}
or, equivalently,
\begin{equation}
\label{eq:parallel_const1}
\langle m(i)-m(j),\alpha\rangle=0 \quad \forall \alpha\in \mathbb{R}^d \text{ such that } \langle p(i)-p(j),\alpha\rangle=0
\end{equation}
for all $\{i,j\}\in E(H)$, where $\langle \cdot, \cdot\rangle$ denotes the standard inner product in the Euclidean space. 

For $t\in \mathbb{R}^d$, 
let us define a constant map $m_t:V(H)\rightarrow \mathbb{R}^d$ by
$m_t(v)=t$ for $v\in V(H)$.
Then, $m_t$ is a relocation of $(H,p)$, known as a {\em translation}. 
On the other hand, let $m_{\rm di}(v)=p(v)$ for $v\in V(H)$.
Then, $m_{\rm di}$ is also a relocation of $(H,p)$, known as a {\em dilation}.
In general, a relocation is called {\em trivial} if it is a linear combination of $m_{\rm di}$ and $m_t$ 
for $t\in \mathbb{R}^d$, 
and $(H,p)$ is called {\em robust} if all possible relocations are trivial.
The set of trivial relocations of $(H,p)$ forms a linear subspace of $(\mathbb{R}^d)^V$,
denoted by ${\rm tri}(H,p)$, which has dimension $d+1$ unless 
$\{p(v)\mid v\in V(H)\}$ is a point.

In~\cite{Whiteley:89,Whitley:1997}, Whiteley showed a combinatorial characterization of 
robust frameworks on generic point-configurations,  
as a corollary of a combinatorial characterization of reconstructivity of pictures appeared 
in scene analysis.
The goal of this section is to extend this result to 
the {\em symmetric parallel redrawing problem} of symmetric frameworks.

Let $\Gamma$ be a discrete point group in $d$-dimensional Euclidean space $\mathbb{R}^d$,
and suppose that $(H,p)$  is $\Gamma$-symmetric.
It is then natural to ask whether there is a redrawing 
preserving the symmetry.
Namely, we shall take into account only {\em $\Gamma$-symmetric relocations}.
It is straightforward from (\ref{eq:symmetric_func})(\ref{eq:parallel_const1}) to see the following:
\begin{proposition}
Let $(H,p)$ be a $\Gamma$-symmetric framework,
and let $m:V\rightarrow \mathbb{R}^d$ be a $\Gamma$-symmetric relocation of $(H,p)$.
Then, $(H,p+m)$ is $\Gamma$-symmetric and a parallel redrawing of $(H,p)$.
Conversely, if $(H,q)$ is a $\Gamma$-symmetric parallel redrawing of $(H,p)$,
then $q-p$ is a $\Gamma$-symmetric relocation of $(H,p)$.
\end{proposition}
We thus say that $(H,p)$ is {\em symmetrically robust} 
if any $\Gamma$-symmetric relocation of $(H,p)$ is trivial.
The set of $\Gamma$-symmetric trivial relocations 
forms a linear subspace of ${\rm tri}(H,p)$, which is denoted by ${\rm tri}_\Gamma(H,p)$.

Recall that $\dim_{\mathbb{R}} ({\rm tri}(H,p))=d+1$.
However, 
not every translation $m_t$ is $\Gamma$-symmetric;
$m_t$  is $\Gamma$-symmetric if and only if
$\gamma t=t$ for all $\gamma\in \Gamma$, or equivalently 
$t\in \bigcap_{\gamma\in \Gamma}{\rm ker}(\gamma-I_d)$.
Thus, if $\{p(v)\mid v\in V(H)\}$  is not the orbit of single point,
$(H,p)$ is symmetrically robust if and only if 
the dimension of the space of $\Gamma$-symmetric relocations is equal to 
\begin{equation}
\label{eq:parallel_const4}
\dim_{\dR} {\rm tri}_\Gamma(H,p)=1+\dim_{\mathbb{R}} \bigcap_{\gamma\in \Gamma}{\rm ker}(\gamma-I_d).
\end{equation}
Note that the dilation is always $\Gamma$-symmetric, which is indeed crucial in the proof 
of our main claim below.

\subsubsection{Symmetric parallel redrawing polymatroids}
\label{subsec:parallel_polymatroid}
For a symmetric relocation of $\Gamma$-symmetric framework $(H,p)$,
the system (\ref{eq:parallel_const1}) is 
apparently redundant due to the $\Gamma$-symmetry of $p$ and $m$, 
and the redundancy can be eliminated by using the quotient $\Gamma$-gain graph $(H/\Gamma,\psi)$. 

To see this, let us take a representative vertex $v\in V(H)$ from each vertex orbit $\Gamma v\in V(H/\Gamma)$, 
as mentioned in \textsection\ref{subsec:symmetric_framework}.
Then, there is a natural one-to-one correspondence between $p$ and its quotient $p/\Gamma$ 
(resp., $m$ and $m/\Gamma$) through $p(v)=p/\Gamma(\Gamma v)$ (resp., $m(v)=m/\Gamma(\Gamma v)$).
Recall also that each edge orbit connecting from $\Gamma i$ to $\Gamma j$ is written by 
$\Gamma e=\{\{\gamma i, \gamma \psi_e j\}\colon \gamma\in \Gamma\}$, where $\psi_e$ is the gain of $\Gamma e$ in the quotient gain graph.
Hence, (\ref{eq:parallel_const1}) is written by 
\begin{equation*}
\langle m(\gamma i)-m(\gamma \psi_e j),\alpha\rangle=0 \quad \forall \alpha\in \mathbb{R}^d 
\text{ such that } \langle p(\gamma i)-p(\gamma \psi_e j),\alpha\rangle=0
\end{equation*}
for each edge $\{\gamma i, \gamma \psi_e j\}$ in each edge orbit $\Gamma e$.
Since $\Gamma$ consists of orthogonal matrices, the conditions over edges in $\Gamma e$ 
can be reduced to  one condition,
\begin{equation*}
\langle m(i)-m(\psi_e j),\alpha\rangle=0 \quad \forall \alpha\in \mathbb{R}^d
\text{ such that } \langle p(i)-p(\psi_e j),\alpha\rangle=0.
\end{equation*}
By the $\Gamma$-symmetry of $p$ and $m$, this is further converted to
\begin{equation}
\label{eq:parallel_const_q}
\langle m/\Gamma(\Gamma i)-\psi_e m/\Gamma(\Gamma j),\alpha\rangle=0 \quad \forall \alpha\in \mathbb{R}^d
\text{ such that } \langle p/\Gamma(\Gamma i)-\psi_e p/\Gamma(\Gamma j),\alpha\rangle=0.
\end{equation}
Therefore, to analyze the space of relocations $m$, it suffices to 
analyze the dimension of $m/\Gamma$ satisfying (\ref{eq:parallel_const_q}) for all edge orbit 
$\Gamma e=(\Gamma i, \Gamma j)$.

Thus, by simplifying notations, 
the problem can be considered in  a general $\Gamma$-gain graph $(G,\psi)$, 
and our goal is to understand
the space of $m\in (\mathbb{R}^d)^{V(G)}$ satisfying 
\begin{equation}
\label{eq:parallel_const5}
\langle m(i),- \psi_e m(j),\alpha\rangle=0 \quad \forall \alpha\in \mathbb{R}^d
\text{ such that } \langle p(i)-\psi_e p(j),\alpha\rangle=0
\end{equation}
for every $e=(i,j)\in E(G)$ with $\psi(e)=\psi_e$. 
To do that, we associate  a $(d-1)$-dimensional linear subspace $P_{e,\psi}(p)$ with each 
edge orbit $e=(i,j)\in E(G)$ defined by
\begin{equation}
\label{eq:direction_flat}
P_{e,\psi}(p)=A_{e,\psi}\cap \{x\in (\mathbb{R}^d)^{V(G)}\mid
\langle p(i)-\psi_ep(j),x(i)\rangle=0\}
\end{equation}
where $A_{e,\psi}$ is, as defined in (\ref{eq:A_non_loop})(\ref{eq:A_loop}),
\begin{equation}
\label{eq:direction_flat_A1}
A_{e,\psi}=\left\{x\in (\mathbb{R}^d)^{V(G)}\Big\vert 
\begin{array}{l}
x(i)+\psi_e x(j)=0, \\
x(V\setminus \{i,j\})=0
\end{array}
\right\}
\end{equation}
or 
\begin{equation}
\label{eq:direction_flat_A2}
A_{e,\psi}=\left\{x\in (\mathbb{R}^d)^{V(G)}\Big\vert \exists \alpha\in \mathbb{R}^d\colon 
\begin{array}{l}
x(i)=(I_d-\psi_e)\alpha, \\ 
x(V\setminus \{i\})=0
\end{array}
\right\}
\end{equation}
depending on whether $e$ is a non-loop or a loop, respectively.
Observe then that $m\in (\mathbb{R}^d)^{V(G)}$ satisfies (\ref{eq:parallel_const5}) for all $e\in E(G)$ 
if and only if 
$m$ is in the orthogonal complement of $\spa\{P_{e,\psi}(p)\mid e\in E(G)\}$, because,
for any $x\in P_{e,\psi}(p)$, we have
\begin{align*}
\langle m, x\rangle&=\langle m(i),x(i)\rangle + \langle m(j),x(j)\rangle \\
&= \langle m(i), x(i)\rangle - \langle m(j), \psi_e^{-1}x(i)\rangle \\
&= \langle m(i)-\psi_e m(j), x(i)\rangle
\end{align*}
with $\langle p(i)-\psi_e p(j),x(i)\rangle =0$.
In total, we proved the following.
\begin{theorem}
\label{thm:parallel_eq}
Let $(H,p)$ be a $\Gamma$-symmetric framework, 
and $(H/\Gamma,\psi)$ the quotient $\Gamma$-gain graph of $H$.
Then, $(H,p)$ is symmetrically robust if and only if 
\[
 \dim_{\mathbb{R}}\{P_{e,\psi}(p/\Gamma)\mid e\in E(H/\Gamma)\}
=d|V/\Gamma|-1-\dim_{\mathbb{R}} \bigcap_{\gamma\in \Gamma}{\rm} ( {\rm ker} (I_d-\gamma). 
\]
\end{theorem}

\subsubsection{Combinatorial characterization}
By Theorem~\ref{thm:parallel_eq}, 
it now suffices to analyze the polymatroid of $\Gamma$-gain graphs $(G,\psi)$ with linear representation 
$e \mapsto P_{e,\psi}(p)$ for $p:V(G)\rightarrow \mathbb{R}^d$, which we call the {\em $\Gamma$-symmetric parallel redrawing polymatroid}
of $(G,\psi)$ (with respect to $p$).
The following theorem provides a combinatorial characterization of this polymatroid.
\begin{theorem}
\label{theorem:direction}
Let $\Gamma$ be a discrete point group 
with the natural representation $\rho:\gamma \in \Gamma \mapsto \gamma \in GL(\mathbb{R}^d)$,
$(G=(V,E),\psi)$ a $\Gamma$-gain graph,
and  $p:V\rightarrow \mathbb{R}^d$ a $\Gamma$-generic mapping. 
Define $f_{\rho}$ as in (\ref{eq:f_rho}) by 
\begin{equation*}
f_{\rho}(F)=d|V(F)|-dc(F)+\sum_{X\in C(F)}\dim_{\dR}\{{\rm image}(I_d-\rho(\gamma))\mid \gamma \in X\} 
\qquad (F\subseteq E)
\end{equation*}
and  $h_{\rho}:2^{E}\rightarrow \mathbb{Z}$ by
\begin{equation}
\label{eq:parallel_func}
h_{\rho}(F)=f_{\rho}(F)-1 \qquad (F\subseteq E).
\end{equation}
Then,
\[
 \dim_{\dR}\{P_{e,\psi}(p)\mid e\in E \}=\hat{h}_{\rho}(E).
\] 
In other words, for almost all $p$, the $\Gamma$-symmetric parallel redrawing polymatroid is equal to 
the polymatroid induced by $h_{\rho}$.
\end{theorem}
\begin{proof}
%We prove 
%
%$\rank(\overline{{\cal A}_F(\bp)})=\hat{f'}(F)$ for any nonempty $F\subseteq E$ 
%(see (\ref{eq:g_hat}) for the definition of $\hat{f}$).
%As usual, one direction is easy; for any $F\subseteq E$, $\rank(\overline{{\cal A}_F(p)})\leq \hat{f'}(F)$.
%We only prove the nontrivial direction ``$\geq$''.
The proof idea is from the alternative proof of Laman's theorem by  Lov{\'a}sz and Yemini\cite{lovasz:1982}.

Applying Theorem~\ref{thm:main_poly} with $f_{\rho}$, 
the polymatroid 
$\mathbf{P}(f_{\rho})=(E,f_{\rho})$ is equal to the linear polymatroid 
$\mathbf{DP}_{\rho}(G,\psi)$ with the linear representation $e\mapsto A_{e,\psi}$ 
(given in (\ref{eq:direction_flat_A1})(\ref{eq:direction_flat_A2})).
 
We define a hyperplane ${\cal H}$ of $(\mathbb{R}^d)^V$ 
(i.e., a ($d|V|-1$)-dimensional subspace) by 
\begin{equation*}
{\cal H}=\{x\in (\mathbb{R}^d)^V\mid \langle p,x\rangle =0 \}.
\end{equation*}
Then, observe that $P_{e,\psi}(p)=A_{e,\psi}\cap {\cal H}$ for every $e\in E$. 
Indeed, 
for any  $e=(i,j)\in E$ and any $x\in A_{e,\psi}$, we have
$\langle p,x\rangle=\langle p(i),x(i)\rangle+\langle p(j),x(j)\rangle=\langle p(i),x(i)\rangle+\langle p(j),-\psi_e^{-1}x(i)\rangle
=\langle p(i)-\psi_ep(j),x(i)\rangle$, 
which implies that $x\in {\cal H}$ if and only if $\langle p(i)-\psi_ep(j),x(i)\rangle=0$.
Therefore, as $p$ is $\Gamma$-generic, 
we conclude that 
the $\Gamma$-symmetric parallel redrawing polymatroid of $(G,\psi)$ 
is obtained from $\mathbf{DP}_{\rho}(G,\psi)$ by a Dilworth truncation, given in \textsection\ref{subsec:linear_truncation}.
By Theorem~\ref{theorem:truncation} and Theorem~\ref{thm:main_poly}, we  obtain
\begin{align*}
&\dim_{\dR} \{P_{e,\psi}(p)\mid e\in E\} \\
&=\min\{ \sum_i(\dim_{\dR}\{ A_{e,\psi}\mid e\in E_i \}-1) \mid  \text{a partition} \{E_1,\dots, E_k\} 
\text{ of } E\}  \quad \text{(by Theorem~\ref{theorem:truncation})}\\
&=\min\{ \sum_i (f_{\rho}(E_i)-1) \mid \text{ a partition } \{E_1,\dots, E_k\} \text{ of } E \} \qquad \text{(by Theorem~\ref{thm:main_poly})} \\
&=\min\{ \sum_i h_{\rho}(E_i) \mid \text{ a partition } \{E_1,\dots, E_k\} \text{ of } E \} \qquad \text{(by (\ref{eq:parallel_func}))} \\
&=\hat{h}_{\rho}(E). \qquad \text{(by (\ref{eq:hat})).}
\end{align*}
\end{proof}

The following extends  the result of Whiteley~\cite{Whitley:1997} to the symmetric parallel redrawing problem,
which directly follows from Theorem~\ref{thm:parallel_eq} and Theorem~\ref{theorem:direction}.
\begin{corollary}
\label{cor:parallel}
Let $H$ be a $\Gamma$-symmetric graph for a discrete point group $\Gamma$ 
and $(H/\Gamma,\psi)$ be the quotient $\Gamma$-gain graph.
For almost all $\Gamma$-symmetric $p:V(H)\rightarrow \mathbb{R}^d$,  
$(H,p)$ is symmetrically robust if and only if 
the graph obtained from $H/\Gamma$ by replacing each edge $e\in E(H/\Gamma)$
by $d-1$ parallel copies contains 
an edge subset $I$ satisfying the following counting conditions:
\begin{itemize}
\item $|I|=d|V|-1-\dim_{\dR} \bigcap_{\gamma\in \Gamma}{\rm ker}(\gamma-I_d)$;
\item $|F|\leq d|V(F)|-dc(F)-1+\sum_{X\in C(F)}\dim_{\dR} \{{\rm image}(I_d-\gamma)\mid \gamma\in \langle X\rangle\}$ 
for any nonempty  $F\subseteq I$. 
\end{itemize}
\end{corollary}

\subsection{Symmetry-forced rigidity of symmetric frameworks} 
\label{subsec:rigidity}
%In this paper, we are interested in frameworks having point group symmetries.
We then move to another application of Theorem~\ref{thm:main_poly}, the infinitesimal rigidity of symmetric frameworks.
The papers by Schulze and Whiteley~\cite{schulze2011orbit} and Jord{\'a}n et al.~\cite{gain_sparsity}
contain a more detailed explanation on this topic.

\subsubsection{Symmetry-forced infinitesimal rigidity}
\label{subsec:sym_rigidity_def}
The {\em infinitesimal rigidity} concerns with the dimension of the space of infinitesimal motions.
An {\em infinitesimal motion} of a framework $(H,p)$ 
is defined as an assignment $m\colon V(H)\rightarrow \mathbb{R}^d$ such that 
\begin{equation}
\label{eq:inf}
\langle m(i)-m(j), p(i)-p(j)\rangle=0 \qquad \forall \{i,j\}\in E(H).
\end{equation}
The set of infinitesimal motions forms a linear space, denoted $L(H,p)$.

In general, for a set $P\subseteq \mathbb{R}^d$ of points,
an {\em infinitesimal isometry} of $P$ is defined by $m:P\rightarrow \mathbb{R}^d$ such that
\begin{equation*}
%\label{eq:inf}
\langle m(x)-m(y), x-y \rangle=0 \qquad \forall x, y\in P.
\end{equation*}
The set of infinitesimal isometries forms a linear space, denoted by ${\rm iso}(P)$.
Notice that, for a skew-symmetric matrix $S$ and $t\in \mathbb{R}^d$,
a mapping $m:P\rightarrow \mathbb{R}^d$ defined by 
\begin{equation*}
%\label{eq:trivial_motion}
m(x)=Sx+t \qquad (x\in P)
\end{equation*}
is an infinitesimal isometry of $P$.
Indeed, it is well-known that any infinitesimal isometry can be described in this form,
and  
 \begin{equation}
\label{eq:trivial_dim}
 \dim_{\dR} {\rm iso}(P)=d(k+1)-{k+1 \choose 2},
\end{equation}
where $k$ denotes the affine dimension of $P$.
For example, for $d=2$,  an infinitesimal isometry
is a linear combination of {\em translations} and  the {\em infinitesimal rotation} around the origin.

An infinitesimal motion $m:V(H)\rightarrow \mathbb{R}^d$ 
of a framework $(H,p)$ is said to be {\em trivial} if 
$m$ can be expressed by 
\begin{equation}
\label{eq:trivial_motion}
m(v)=Sp(v)+t \qquad (v\in V(H))
\end{equation}
for some skew-symmetric matrix $S$ and $t\in \mathbb{R}^d$.
The set of all trivial motions forms a linear subspace of $L(H,p)$, denoted by ${\rm tri}(H,p)$.
By definition, ${\rm tri}(H,p)$ is isomorphic to ${\rm iso}(\{p(v)\mid v\in V(H)\})$,
and hence (\ref{eq:trivial_dim}) gives the exact dimension of ${\rm tri}(H,p)$.
$(H,p)$ is called {\em infinitesimally rigid} if $L(H,p)={\rm tri}(H,p)$.

As in the parallel redrawing problem, 
we are interested in $\Gamma$-symmetric infinitesimal motions of symmetric frameworks.
%The set of ${\cal S}$-symmetric infinitesimal motions and 
%the set of trivial ones 
%form linear subspaces of $L(H,p)$ and ${\rm tri}(H,p)$, denoted 
%$L(H,p,{\cal S},\rho)$ and ${\rm tri}(H,p,{\cal S},\rho)$, respectively.
For a discrete point group $\Gamma$, 
a $\Gamma$-symmetric framework $(H,p)$ is said to be {\em symmetry-forced rigid}
if any $\Gamma$-symmetric infinitesimal motion of $(H,p)$ is trivial.
We should remark that, as in the case of the parallel redrawing problem, 
not every trivial infinitesimal motion is $\Gamma$-symmetric.
%Figure~\ref{fig:symmetric_inf} illustrates examples.
%and if $P=\{x\}$,
%\[
% \dim {\rm iso}(P)=\begin{cases}
%1 & \text{ if } k=1 \\
%0 & \text{ if } k\geq 2,
%\end{cases}
%\]
%
%\begin{figure}[th]
%\centering
%\begin{minipage}{0.3\textwidth}
%\centering
%\includegraphics[scale=0.35]{symmetric_inf1.eps}
%\par
%(a)
%\end{minipage}
%\begin{minipage}{0.3\textwidth}
%\centering
%\includegraphics[scale=0.35]{symmetric_inf2.eps}
%\par
%(b)
%\end{minipage}
%\begin{minipage}{0.3\textwidth}
%\centering
%\includegraphics[scale=0.35]{symmetric_inf3.eps}
%\par
%(c)
%\end{minipage}
%\caption{Three independent infinitesimal isometries in the plane, 
%among which (a) is symmetric with respect to the group of a vertical reflection, 
%(b) is symmetric with respect to the group of a horizontal reflection, 
%and (c) is symmetric with respect to the group of rotations.}
%\label{fig:symmetric_inf}
%\end{figure}

The following result of Schulze~\cite{schulze:2010} 
motivates us to look at symmetry-forced infinitesimal rigidity.
(The precise definition of some terminologies are omitted here.)
\begin{theorem}[Schulze~\cite{schulze:2010}]
\label{thm:finite_motion}
Let $\Gamma$ be a discrete point group, and 
$H$ be a $\Gamma$-symmetric graph.
Then, for any $\Gamma$-generic $p$,
$(H,p)$ has a nontrivial continuous motion that preserves 
the $\Gamma$-symmetry
if and only if $(H,p)$ has a nontrivial $\Gamma$-symmetric infinitesimal motion.
\end{theorem}

\subsubsection{Orbit rigidity matrix}
\label{subsec:orbit_rigidity}

Let $(H, p)$ be a $\Gamma$-symmetric framework.
Due to $\Gamma$-symmetry, the system (\ref{eq:inf}) of linear equations (with respect to $m$)
is redundant.
Schulze and Whiteley~\cite{schulze2011orbit} pointed out that 
the system can be reduced to $|E(H)/\Gamma|$ linear equations.

In fact, by the same manner as \textsection\ref{subsec:parallel_polymatroid}, it follows that 
$m:V(H)\rightarrow \mathbb{R}^d$ is a $\Gamma$-symmetric infinitesimal motion of $(H,p)$
if and only if 
\begin{equation}
\label{eq:inf_orbit}
\langle m/\Gamma(\Gamma i), p/\Gamma(\Gamma i)-\psi_e\cdot p/\Gamma(\Gamma j) \rangle+
\langle m/\Gamma (\Gamma j), p/\Gamma(\Gamma j)-\psi_e^{-1}\cdot p/\Gamma(\Gamma i)\rangle=0
\end{equation} 
for every oriented edge orbit $\Gamma e=(\Gamma i, \Gamma j)$ in the quotient gain graph $(H/\Gamma,\psi)$.
By regarding (\ref{eq:inf_orbit}) as a system of linear equations of $m/\Gamma$,
the corresponding $|E(H)/\Gamma|\times d|V(H)/\Gamma|$-matrix is called 
the {\em orbit rigidity matrix} by Schulze and Whiteley~\cite{schulze2011orbit}.

In general, for a $\Gamma$-gain graph $(G,\psi)$ and $p:V(G)\rightarrow \mathbb{R}^d$,
we are interested in the system of linear equations on $m\in (\mathbb{R}^d)^{V(G)}$ defined by 
\begin{equation}
\label{eq:inf_orbit2}
\langle m( i), p(i)-\psi_ep(j) \rangle+
\langle m( j), p(j)-\psi_e^{-1}p(i) \rangle=0 \quad \forall e=(i,j)\in E(G).
\end{equation} 
To analyze the solution space of (\ref{eq:inf_orbit2}),
we then associate a 1-dimensional linear space $R_{e,\psi}$ with each $e=(i,j)\in E(G)$,
\begin{equation}
\label{eq:rigidity_flat1}
R_{e,\psi}(p)=\left\{x\in (\mathbb{R}^d)^{V(G)} \Bigg\vert 
\exists t \in \mathbb{R} \colon 
\begin{array}{l}
x(i)=t(p(i)-\psi_e p(j)), \\ 
x(j)=t(p(j)-\psi_e^{-1} p(i)),  \\
x(V\setminus \{i,j\})=0 
\end{array}
\right\}
\end{equation}
if $e$ is a non-loop edge, and 
\begin{equation}
\label{eq:rigidity_flat2}
R_{e,\psi}(p)=\left\{x\in (\mathbb{R}^d)^{V(G)} \Big\vert \exists t \in \mathbb{R} \colon  
\begin{array}{l}
x(i)=t(2I_d-\psi_e -\psi_e^{-1})p(i), \\
x(V\setminus \{i\})=0
\end{array}
\right\}
\end{equation}
if $e$ is a loop attached to $i$.
Observe then that $m$ satisfies  (\ref{eq:inf_orbit2}) if and only if $m$ is in the orthogonal complement of 
$\spa\{R_{e,\psi}(p)\mid e\in E(G)\}$.
This implies the following.
\begin{proposition}[Schulze and Whiteley~\cite{schulze2011orbit}]
\label{prop:inf}
Let $(H, p)$ be a $\Gamma$-symmetric framework 
and $(H/\Gamma,\psi)$ be the quotient $\Gamma$-gain graph.
Then, the dimension of the space of $\Gamma$-symmetric infinitesimal motions of $(H,p)$ is equal to
\[
 d|V(H)/\Gamma|-\dim_{\dR} \{R_{e,\psi}(p/\Gamma)\mid e\in E(H/\Gamma)\}.
\]
\end{proposition}
The detailed description and examples can be found in \cite{schulze:2010,schulze2011orbit}.
A combinatorial necessity condition of symmetry-forced rigidity 
can be found in \cite{gain_sparsity}. 

\subsubsection{Combinatorial characterization}
The following theorem provides a combinatorial characterization of the linear matroid
induced on $\{R_{e,\psi}(p)\mid e\in E(G)\}$
for the spacial case of $d=2$ and rotation groups ${\cal C}_k$. 
\begin{theorem}
\label{theorem:rigidity}
Let ${\cal C}_k$ be the group of $k$-fold rotations around the origin in the plane.
Let $(G=(V,E),\psi)$ be a ${\cal C}_k$-gain graph and  
$p:V\rightarrow \mathbb{R}^2$ be a ${\cal C}_k$-generic mapping. 
%Define $h_{\Gamma}:2^{E}\rightarrow \mathbb{Z}$ by
%\begin{equation}
%\label{eq:parallel_func}
%h_{\rho}(F)=f_{\rho}(F)-1 \qquad (F\subseteq E),
%\end{equation}
%where $f_{\rho}$ is as defined in (\ref{eq:f_rho}) with the natural representation 
%$\rho:\gamma \in \Gamma \mapsto \gamma \in GL(\mathbb{R}^d)$.
Then,
\[
 \dim_{\dR}\{R_{e,\psi}(p)\mid e\in E \}=\hat{h}_{\rho}(E),
\] 
where $h_{\rho}$ is as defined in (\ref{eq:parallel_func}) with $\Gamma={\cal C}_k$.
\end{theorem}
\begin{proof}
%We prove 
%
%$\rank(\overline{{\cal A}_F(\bp)})=\hat{f'}(F)$ for any nonempty $F\subseteq E$ 
%(see (\ref{eq:g_hat}) for the definition of $\hat{f}$).
%As usual, one direction is easy; for any $F\subseteq E$, $\rank(\overline{{\cal A}_F(p)})\leq \hat{f'}(F)$.
%We only prove the nontrivial direction ``$\geq$''.
The proof technique is exactly 
the same as the proof of Theorem~\ref{theorem:direction}.

Let $C_{\pi/2}$ be the matrix of size $2\times 2$, representing 
the  $4$-fold rotation around the origin in the Euclidean plane.
 We  define a hyperplane ${\cal H}'$ of $(\mathbb{R}^2)^V$ by 
\begin{equation*}
{\cal H}'=\{x\in (\mathbb{R}^2)^V\mid \sum_{v\in V}\langle C_{\pi/2}p(v),x(v)\rangle =0 \}
\end{equation*}
where $p\in (\mathbb{R}^2)^V$ is ${\cal C}_k$-generic as defined in the statement.

Let $A_{e,\psi}$ be the 2-dimensional linear subspace defined in 
(\ref{eq:direction_flat_A1})(\ref{eq:direction_flat_A2}) with $d=2$.
Then, observe that,  for each non-loop $e=(i,j)\in E$ and for any $x\in A_{e,\psi}$, 
\begin{align*}
\sum_{v\in V}\langle C_{\pi/2}p(v),x(v)\rangle&=\langle C_{\pi/2}p(i),x(i)\rangle+\langle C_{\pi/2}p(j),x(j)\rangle \\
&=\langle C_{\pi/2}p(i),x(i)\rangle+\langle C_{\pi/2}p(j),-\psi_e^{-1}x(i)\rangle \\
&=\langle C_{\pi/2}(p(i)-\psi_ep(j)),x(i)\rangle 
\end{align*}
where we used the fact that $C_{\pi/2}$ commutes with any element of ${\cal C}_k$.
This implies that $x\in {\cal H}'$ if and only if $x(i)\in \spa\{p(i)-\psi_e p(j)\}$ 
for every $e=(i,j)\in E$. 
In other words, $R_{e,\psi}(p)=A_{e,\psi}\cap {\cal H}'$.
The same analysis works in case of loops $e$.
 
Since $p$ is ${\cal C}_k$-generic,  
we conclude that 
the linear matroid induced on $\{R_{e,\psi}(p)\mid e\in E\}$ 
is obtained from $\mathbf{DP}_{\rho}(G,\psi)$ by a Dilworth truncation.
Since $\mathbf{DP}_{\rho}(G,\psi)=\mathbf{P}(f_{\rho})$ by Theorem~\ref{thm:main_poly},
Theorem~\ref{theorem:truncation} implies the statement.
%(For a concrete description, see the last paragraph of the proof of Theorem~\ref{theorem:direction}.)
\end{proof}

Combining Proposition~\ref{prop:inf} and Theorem~\ref{theorem:rigidity}, we conclude that 
the row matroid of an orbit rigidity matrix is the matroid induced by $h_{\rho}$,
if $d=2$ and the underlying symmetry is a group of rotations.  
The same characterization was obtained in \cite{malestein2011generic,gain_sparsity} by different approaches.

\section{Matroids Induced by Group Actions on Exterior Product Spaces}
\label{sec:action}
In this section, we give another application to rigidity theory,
where we extend Tay's theorem~\cite{tay:84} on generic rigidity of body-bar frameworks to a symmetric setting.
The result solves a conjecture given in \cite{ross2011}. 
For this extension,  we shall first investigate 
the case when a group is represented in the exterior product of vector spaces.

\subsection{Restriction to decomposable $k$-vectors}
\label{subsec:decomposable}
In the subsequent discussion of this section, 
the underlying group $\Gamma$ is equipped with a linear representation $\rho:\Gamma\rightarrow GL(\mathbb{F}^d)$
over a filed $\mathbb{F}$.
As before, we denote by $\mathbb{K}$ the field obtained from $\mathbb{F}$ by transcendental extensions.

Let  $\bigwedge^k \mathbb{F}^{d}$ be the $k$-th exterior power of $\mathbb{F}^d$.
Recall that $\bigwedge^k \mathbb{F}^{d}$ is a ${d\choose k}$-dimensional linear space,
and so each entry of an element of $\bigwedge^k\mathbb{F}^d$ can be naturally indexed by 
a $k$-tuple $(i_1,\dots, i_k)$ with $1\leq i_1<\dots <i_k\leq d$.
An element of $\bigwedge^k \mathbb{F}^{d}$ is called a {\em $k$-vector},
and a $k$-vector is said to be {\em decomposable} if it can be written of the form  $v_1\wedge \dots \wedge v_k$
for some $v_1,\dots, v_k\in \mathbb{F}^d$.
%Recall that $\bigwedge^k \mathbb{F}^d$ is spanned by decomposable $k$-vectors.

Let us consider a natural action of $\Gamma$ on $\bigwedge^k\mathbb{F}^d$,
that is, $\gamma\in \Gamma$ acts on a decomposable element $v_1\wedge \dots \wedge v_k$ by
$\rho(\gamma)v_1\wedge \dots \wedge \rho(\gamma)v_k$ 
and extends linearly to the other elements.
This leads to a linear mapping from $\bigwedge^k \mathbb{F}^d$ to  $\bigwedge^k \mathbb{F}^d$.
It is known that this action is a well-defined group representation of $\Gamma$ over $GL(\bigwedge^k \mathbb{F}^d)$.
In other words, there is a unique representation 
$\rho^{(k)}:\Gamma\rightarrow GL(\bigwedge^k \mathbb{F}^d)$
such that $\rho^{(k)}(\gamma)(v_1\wedge\dots \wedge v_k)=\rho(\gamma)v_1\wedge \dots \wedge \rho(\gamma)v_k$
for each $\gamma\in \Gamma$ and each $v_1\wedge\dots\wedge v_k$ (see, e.g.,\cite[Chapter 7]{HodgePedoe}).

Note that $\rho^{(k)}(\gamma)$ is a matrix of size ${d\choose k}\times {d\choose k}$.
To see a specific expression of the entries, let us simply denote $N=\rho(\gamma)$.
For $1\leq i_1<\dots<i_k\leq d$ and $1\leq j_1<\dots <j_k\leq d$,
let $N_{i_1\dots i_k}^{j_1\dots j_k}$ be the submatrix of $N$ induced by 
the $i_1$-th, $\dots, i_k$-th rows and the $j_1$-th $\dots, j_k$-th columns. 
If we index each column and each row of $N^{(k)}$ by a $k$-tuple $(i_1,\dots, i_k)$ 
according to the index ordering of elements in $\bigwedge^k \mathbb{F}^d$, 
we have 
\[
 N^{(k)}[(j_1,\dots, j_k),(i_1,\dots, i_k)]={\rm det}\ n_{i_1\dots i_k}^{j_1\dots j_k}, 
\]
where $N^{(k)}[(j_1,\dots, j_k),(i_1,\dots, i_k)]$ denotes the entry at 
$(j_1,\dots, j_k)$-th row  and $(i_1,\dots, i_k)$-th column.

Using representation $\rho^{(k)}$, 
we now consider a special case of matroids given in \textsection\ref{subsec:dowling_linear}.
Let $(G=(V,E),\psi)$ be a $\Gamma$-gain graph with a gain function $\psi:e \mapsto \psi_e$.
For each $e=(i,j)\in E$, let us assign $x_{e,\psi}^{(k)}\in (\bigwedge^k \mathbb{K}^{d})^V$ as follows:
\begin{equation}
\label{eq:exterior_vector}
x_{e,\psi}^{(k)}(v)=
\begin{cases}
-\rho^{(k)}(\psi_e)\alpha_e & \text{ if } v=i \\
\alpha_e & \text{ if } v=j \\
0 & \text{ otherwise}
\end{cases}
\end{equation}
\begin{equation}
\label{eq:exterior_vector_loop}
x_{e,\psi}^{(k)}(v)=
\begin{cases}
(\rho^{(k)}(\psi_e)-I_{d\choose k})\alpha_e & \text{ if } v=i \\
0 & \text{ otherwise}
\end{cases}
\end{equation}
depending on whether $e$ is a non-loop edge or a loop, 
where $\alpha_e=(\alpha_e^1,\dots, \alpha_e^{d\choose k})^{\top}\in \bigwedge^k \mathbb{K}^{d}$ such that 
$\{\alpha_e^i\mid 1\leq i\leq {d\choose k}, e\in E\}$ is algebraically independent over $\mathbb{F}$.
Then, by Corollary~\ref{cor:main_matroid}, $\{x_{e,\psi}^{(k)}\mid e\in E\}$ is linearly independent if and only if 
$|F|\leq f_{\rho^{(k)}}(F)$ for any  $F\subseteq E$
where 
\begin{align}
f_{\rho^{(k)}}(F)&={d\choose k}|V(F)|-{d\choose k}c(F)+\sum_{X\in C(F)}d_{\rho^{(k)}}\langle X\rangle \qquad
(F\subseteq E) \\
d_{\rho^{(k)}}(X)&=\dim_{\dR} \{{\rm image}(\rho^{(k)}(\gamma)-I_{d\choose k})\mid \gamma\in X\} \qquad (X\subseteq \Gamma). 
\label{eq:exterior_d}
\end{align}
Note that, for any $\gamma\in \Gamma$, 
we have 
\[
 {\rm image}(\rho^{(k)}(\gamma)-I_{d\choose k})={\rm span} \{ \rho(\gamma)p_1\wedge \dots \wedge \rho(\gamma)p_k 
-p_1\wedge\dots \wedge p_k\mid p_1,\dots, p_k\in \mathbb{F}^d\},
\]
and hence $d_{\rho^{(k)}}$ can be rewritten in terms of decomposable $k$-vectors  
by putting this into (\ref{eq:exterior_d}).
%\begin{equation}
%\label{eq:exterior_d}
%d_{\rho^{(k)}}(F)=\sum_{X\in C(F)}
%\dim({\rm span}\{ \rho(\gamma)p_1\wedge \dots \wedge \rho(\gamma)p_k\colon p_1,\dots, p_k\in \mathbb{F}^d, 
%\gamma\in \langle F\rangle_v\}).
%\end{equation}
It is hence natural to ask whether the matroid $\mathbf{M}(f_{\rho^{(k)}})$ has a linear representation 
in terms of decomposable $k$-vectors, that is, a representation given by
\begin{equation}
\label{eq:decomposable_vector}
\hat{x}_{e,\psi}^{(k)}(v)=
\begin{cases}
-\rho(\psi_e)p_{e,1}\wedge \dots \wedge \rho(\psi_e)p_{e,k} & \text{ if } v=i \\
p_{e,1}\wedge\dots\wedge p_{e,k} & \text{ if } v=j \\
0 & \text{ otherwise}
\end{cases}
\end{equation}
and
\begin{equation}
\label{eq:decomposable_vector_loop}
\hat{x}_{e,\psi}^{(k)}(v)=
\begin{cases}
\rho(\psi_e)p_{e,1}\wedge \dots \wedge \rho(\psi_e)p_{e,k}-p_{e,1}\wedge\dots\wedge p_{e,k} & \text{ if } v=i \\
0 & \text{ otherwise}
\end{cases}
\end{equation}
for some $p_{e,1},\dots, p_{e,k}\in \mathbb{K}^d$.
The next theorem asserts that (\ref{eq:decomposable_vector})(\ref{eq:decomposable_vector_loop}) indeed
define a linear representation of $\mathbf{M}(f_{\rho^{(k)}})$.

\begin{theorem}
\label{thm:exterior}
For some $\{p_{e,i}\in \mathbb{K}^d\mid e\in E, 1\leq i\leq k\}$,
the linear matroid induced on $\{\hat{x}_{e,\psi}^{(k)}\mid e\in E\}$
is equal to the matroid $\mathbf{M}(f_{\rho^{(k)}})$.
\end{theorem}
\begin{proof}
Let 
\begin{equation*}
A_{e,\psi}^{(k)}=\left\{x\in (\mbox{$\bigwedge^k \mathbb{K}^{d}$})^V \Big\vert
\begin{array}{l}
x(i)+\rho^{(k)}(\psi_e)x(j)=0, \\ 
x(V\setminus \{i,j\})=0
\end{array}
\right\} 
\end{equation*}
for a non-loop edge $e\in E$, and
\begin{equation*}
A_{e,\psi}^{(k)}=\left\{x\in (\mbox{$\bigwedge^k \mathbb{K}^{d}$})^V \Big\vert 
\exists \alpha\in (\mbox{$\bigwedge^k \mathbb{K}^{d}$})^V\colon
\begin{array}{l}
x(i)=(I_d-\rho^{(k)}(\psi_e))\alpha, \\ 
x(V\setminus \{i,j\})=0
\end{array}
\right\} 
\end{equation*}
for a loop $e\in E$.

Let ${\rm Gr}(d,k)$ be the set of all decomposable $k$-vectors in $\bigwedge^k \mathbb{K}^d$.
${\rm Gr}(d,k)$ is known as the {\em Grassmannian} in the literature
and is an irreducible rational variety spanning $\bigwedge^k\mathbb{K}^d$.
We shall define a subset $\hat{A}_{e,\psi}$ of $A_{e,\psi}$ by 
\begin{align*}
\hat{A}_{e,\psi}^{(k)}=A_{e,\psi}^{(k)}\cap \{x\in (\mbox{$\bigwedge^k \mathbb{K}^{d}$})^V\mid x(j)\in {\rm Gr}(d,k)\}.
\end{align*}
for a non-loop edge $e=(i,j)$, and
\begin{equation*}
\hat{A}_{e,\psi}^{(k)}=\left\{x\in (\mbox{$\bigwedge^k \mathbb{K}^{d}$})^V \Big\vert 
\exists \alpha\in {\rm Gr}(d,k) \colon
\begin{array}{l}
x(i)=(I_d-\rho^{(k)}(\psi_e))\alpha, \\ 
x(V\setminus \{i,j\})=0
\end{array}
\right\} 
\end{equation*}
for a loop $e$.
By Theorem~\ref{thm:main_poly}, we know that 
$f_{\rho^{(k)}}(E)=\dim_{\dR}\{A_{e,\psi}^{(k)}\mid e\in E\}$,
and each linear representation of $\mathbf{M}_{\rho^{(k)}}$ is obtained by taking a representative vector 
$x_{e,\psi}^{(k)}$ from each $A_{e,\psi}^{(k)}$ in generic position.
Thus, to show the statement, 
it suffices to show that a representative vector $x_{e,\psi}^{(k)}$ 
can be taken from $\hat{A}_{e,\psi}^{(k)}$ so that 
$\{x_{e,\psi}^{(k)}\mid e\in E\}$ is in generic position in the sense of (\ref{eq:generic_position}). 

%If $Gr(d,k)=\bigwedge^k\mathbb{R}^d$, then $\hat{A}_{e,\psi}^{(k)}=A_{e,\psi}^{(k)}$ 
%and hence representative vectors can be taken from $\hat{A}_{e,\psi}^{(k)}$ in generic position.
 
Suppose that $E$ has no loop.
Then, each $\hat{A}_{e,\psi}^{(k)}$ is (linearly) isomorphic to ${\rm Gr}(d,k)$ by a projection to $x(j)$. 
Notice that the condition (\ref{eq:generic_position}) of genericity is written in terms of linear dependencies.
Since ${\rm Gr}(k,d)$ is an irreducible rational variety, 
the linear isomorphism between $\hat{A}_{e,\psi}^{(k)}$ and ${\rm Gr}(d,k)$ 
implies that a representative vertex can be taken from $\hat{A}_{e,\psi}^{(k)}$ in generic position.
(A more detailed description for a special case can be found 
in \cite[Theorem 3.1]{tanigawa_truncation}, and the exactly same argument can be applied here.)

%Even if $G$ contains a loop, the same argument can be applied.
If $e$ is a loop, then $A_{e,\psi}^{(k)}$ and $\hat{A}_{e,\psi}^{(k)}$ are linearly isomorphic to
$(\rho^{(k)}(\psi_e)-I_d) (\bigwedge^k \mathbb{K}^d)$ and 
$(\rho^{(k)}(\psi_e)-I_d) {\rm Gr}(d,k)$, respectively. 
Since $\rho^{(k)}(\psi_e)-I_d$ is a linear operator, we can apply the same argument.
%
%As $\bigwedge^k\mathbb{R}^d=\spa {\rm Gr}(d,k)$, we have
%$\{(\rho^{(k)}(\psi_e)-I_d)\alpha\mid \alpha\in \bigwedge^k \mathbb{K}^d\}=
%{\rm span} \{(\rho^{(k)}(\psi_e)-I_d)\alpha\mid \alpha\in {\rm Gr}(d,k)\}$.  
%
%
%Suppose that $E$ contains a loop $e$.
%Then, 
%(if $\mathbb{K}$ is large enough). 
%Therefore, we can always take a representative vector $x_{e,\psi}^{(k)}$ from $A_{e,\psi}^{(k)}$ 
%without decreasing the dimension of $\spa\{x_{e,\psi}^{(k)}\mid e\in E\}$.
\end{proof}

\subsection{Symmetry-forced Rigidity of Body-bar Frameworks}
\label{subsec:symmetric_body_bar}
As an application,
we consider  matroids arose in the rigidity of symmetric body-bar frameworks,
which are  structures consisting of
rigid bodies linked by bars in $\mathbb{R}^d$.
%Tay showed that, if bars are placed generically, 
%the rigidity can be characterized by a simple combinatorial condition of the underlying graph.
%In this section, we shall extend Tay's result to symmetric frameworks.

Let ${\rm Aff}(\mathbb{R}^d)$ be 
the group of invertible affine transformations.
It is well-known that ${\rm Aff}(\mathbb{R}^d)=GL(\mathbb{R}^d)\ltimes \mathbb{R}^d$,
that is, 
the semidirect product of $GL(\mathbb{R}^d)$ and $\mathbb{R}^d$,
and  each element $\gamma=(A_{\gamma},t_{\gamma})\in {\rm Aff}(\mathbb{R}^d)$ acts on $\mathbb{R}^d$ 
by $\gamma \cdot q=A_{\gamma}q+t_{\gamma}$ for $q\in \mathbb{R}^d$.

The $d$-dimensional Euclidean group ${\cal E}(d)$ is a subgroup of ${\rm Aff}(\mathbb{R}^d)$,
where $(A_{\gamma},t_{\gamma})\in {\rm Aff}(\mathbb{R}^d)$ is in ${\cal E}(d)$ 
if and only if  $A_{\gamma}\in {\cal O}(\mathbb{R}^d)$. 
A {\em space group (or crystallographic group)} $\Gamma$ is a discrete cocompact subgroup of ${\cal E}(d)$,
i.e., $\mathbb{R}^d/\Gamma$ is compact.
Throughout this subsection, $\Gamma$ denotes either a space group or
a discrete point group, where $t=0$ in case of a point group.
%
%
%a discrete subgroup of the Euclidean group $E(d)$.
%Thus, $\gamma\in \Gamma$ can be represented by a pair $(A,t)$ of a $d\times d$-matrix $A\in O(d,\mathbb{R})$
%and $t\in \mathbb{R}^d$, or equivalently by a $(d+1)\times (d+1)$-matrix of the form 
%$\begin{pmatrix} A & t \\ 0 & 1\end{pmatrix}\in GL(d+1,\mathbb{R})$ with $A\in O(d)$ and $t\in \mathbb{R}^d$.
%
%Given a $\Gamma$-gain graph $(G,\psi)$ with a crystallographic group $\Gamma$,
%we shall consider the matroids given in the last section with respect to  
%the representation $\rho$ of $\Gamma$ given by (\ref{eq:argumented}).

We now  briefly take a look at how the linear matroid given in the last subsection 
\textsection\ref{subsec:decomposable} arises in the context of rigidity 
of body-bar frameworks.
The following modeling is based on \cite{borcea2011periodic}.
A body-bar framework is a structure consisting of rigid bodies connected by bars, 
and it can be represented by a triple $(H,B,q)$, where 
\begin{itemize}
\item $H$ is an undirected graph whose vertex is corresponding to a body and 
whose edge is corresponding to a bar liking the corresponding two bodies;
\item $B$ indicates the location of each body corresponding to each vertex $v\in V$ by 
$B(v)=(A_v,p_v)\in {\cal O}(\mathbb{R}^d)\ltimes \mathbb{R}^d$
(i.e., each body is identified with a Cartesian (local) coordinate system);
\item $q$ indicates the location of each bar in each local coordinate system as follows; 
 for each $e\in E(H)$ and an endvertex $v$ of $e$, $q(e,v)\in \mathbb{R}^d$ 
denotes the coordinate of the endpoint of the bar corresponding to $e$ in the coordinate system of the body $v$. 
Thus, the coordinate in the global system is equal to $A_vq(e,v)+p_v$, denoted by $\tilde{q}(e,v)$.
\end{itemize}
$B$ and $q$ are called a {\em body-configuration} and a {\em bar-configuration}, respectively.

When bodies are moving, each bar constraints the distance between the endpoints.
Such a length constraint can be written by 
\begin{equation}
\label{eq:body_bar_const1}
 \langle \tilde{q}(e,i)-\tilde{q}(e,j), \tilde{q}(e,i)-\tilde{q}(e,j) \rangle = \ell_e 
\qquad \forall e=\{i,j\}\in E(H) 
\end{equation}
by some specific bar-length $\ell_e$.

We consider a symmetric version of body-bar frameworks, 
where a body-bar framework $(H,B,q)$ is {\em $\Gamma$-symmetric} if 
$H$ is a $\Gamma$-symmetric graph (with a specific free action $\theta$) 
and $B$ and $q$ are subject to $\Gamma$-symmetry;
for any $v\in V(H)$, $e\in E(H)$, and $\gamma=(A_{\gamma},t_{\gamma})\in \Gamma$.
\begin{align*}
B(\gamma v)&=\gamma B(v) = (A_{\gamma}A_v,A_{\gamma}p_v+t_{\gamma}) \\
q(\gamma e, \gamma v)&=q(e,v).
\end{align*}
Indeed, in the global coordinate system, we have 
$\tilde{q}(\gamma e, \gamma v)=A_{\gamma}A_vq(\gamma e,\gamma v)+A_{\gamma}p_v+t
=A_{\gamma}A_vq(e,v)+A_{\gamma}p_v+t=\gamma \tilde{q}(e,v)$, 
and thus the definition implies the $\Gamma$-symmetry of $\tilde{q}$ in the global system.

%As in the case of bar-joint frameworks, the length constraint becomes redundant if we focus on the
%symmetric infinitesimal rigidity, that is, 
%an infinitesimal motion is also subjection to $\Gamma$-symmetry.
%Accordingly, we can compactly write the system of constraints by using the quotient $\Gamma$-gain graph 
%$(H/\Gamma,\psi)$. 
By using $\Gamma$-symmetry of $q$ and $B$, the system (\ref{eq:body_bar_const1})
can be reduced to the following system of equations,
\begin{equation}
\label{eq:body_bar_const2}
 \langle \tilde{q}(e,i)-\psi_e \tilde{q}(e,j), \tilde{q}(e,i)-\psi_e \tilde{q}(e,j) \rangle = \ell_e 
\qquad  \forall \Gamma e=(\Gamma i,\Gamma j)\in E(H/\Gamma),
\end{equation}
where $\psi(\Gamma e)=\psi_e$ denotes the gain of $\Gamma e$ in the quotient $\Gamma$-gain graph.

Thus, the analysis can be accomplished on the quotient graph, and 
we may consider the problem even in a general $\Gamma$-gain graph $(G=(V,E),\psi)$.
Namely, given a $\Gamma$-gain graph $(G,\psi)$, 
$B(v)=(A_v,p_v)$ for $v\in V$,  
and $q_{e,i}, q_{e,j}\in\mathbb{R}^d$ for $e=(i,j)\in E$, we consider
the system
\begin{equation}
\label{eq:body_bar_const22}
 \langle \tilde{q}_{e,i}-\psi_e\cdot \tilde{q}_{e,j}, \tilde{q}_{e,i}-\psi_e\cdot \tilde{q}_{e,j} \rangle = \ell_e 
\qquad  \forall e=(i,j)\in E. 
\end{equation}
By taking the derivative with respect to $B$, 
we shall again investigate the infinitesimal rigidity.
To see this, let us focus on the equation for $e=(i,j)\in E$ and simply denote $q_{e,i}$ by $q_i$.
Also, we denote $\psi_e=(A_{\psi_e},t_{\psi_e})$.
Then  (\ref{eq:body_bar_const22}) is rewritten as
\begin{equation*}
\label{eq:body_bar_const3}
\langle 
A_i q_i+p_i-(A_{\psi_e}A_j q_j+A_{\psi_e}p_j+t_{\psi_e}), 
A_iq_i+p_i-(A_{\psi_e}A_jq_j+A_{\psi_e}p_j+t_{\psi_e}) 
\rangle = \ell_e. 
\end{equation*}

To analyze infinitesimal motions of bodies under bar-constraints,
we take the derivative with respect to $(A_v,p_v)$ for $v\in V$, leading to
\begin{equation}
\label{eq:body_bar_const4}
 \langle 
A_iq_i+p_i-(A_{\psi_e}A_jq_j+A_{\psi_e}p_j+t_{\psi_e}), 
\dot{A}_i q_i+\dot{p}_i-(A_{\psi_e}\dot{A}_jq_j+A_{\psi_e}\dot{p}_j) 
\rangle = 0.
\end{equation}
Without loss of generality, we may take $(A_v, p_v)=(I_d,0)$  for all $v\in V$.
Then, $\dot{A}_v$ is a $d\times d$ skew-symmetric matrix over $\mathbb{R}^d$
since the tangent space of ${\cal O}(\mathbb{R}^d)$ at the identity consists of skew-symmetric matrices.
Therefore, (\ref{eq:body_bar_const4}) is a linear equation of variables $(\dot{A}_i,\dot{p}_i)$ and  
$(\dot{A}_j,\dot{p}_j)$ written by
\begin{equation}
\label{eq:body_bar_const42}
 \langle 
q_i-(A_{\psi_e}q_j+t_{\psi_e}), 
\dot{A}_i q_i+\dot{p}_i-(A_{\psi_e}\dot{A}_jq_j+A_{\psi_e}\dot{p}_j) 
\rangle = 0.
\end{equation}  
Since the set ${\cal S}_d$ of skew-symmetric matrices is isomorphic to $\mathbb{R}^{d\choose 2}$,
we can represent each $\dot{A}\in {\cal S}_d$ by a vector $\omega\in \mathbb{R}^{d \choose 2}$.
It is known that, 
by choosing a one-to-one correspondence between ${\cal S}_d$ and $\mathbb{R}^{d\choose 2}$ in  an appropriate manner,
we have 
\[
 \langle h, \dot{A}q\rangle =\langle h\wedge q, \omega\rangle  
\]
for any $h,q\in \mathbb{R}^d$ and any $\dot{A}\in {\cal S}_d$ corresponding to $\omega$.
Therefore, replacing $\dot{A}_i$ and $\dot{A}_j$ with the corresponding 
$\omega_i\in \mathbb{R}^{d\choose 2}$ and $\omega_j\in \mathbb{R}^{d\choose 2}$,
we have the following two relations to simplify (\ref{eq:body_bar_const42}),
\[
 \langle q_i-A_{\psi_e} q_j-t_{\psi_e},\dot{A}_iq_i\rangle 
=\langle (q_i-A_{\psi_e}q_j-t_{\psi_e})\wedge q_i,\omega_i\rangle 
=-\langle (A_{\psi_e} q_j+t_{\psi_e})\wedge q_i,\omega_i\rangle 
\]
\[
 \langle q_i-A_{\psi_e} q_j-t_{\psi_e},A_{\psi_e}\dot{A}_iq_j\rangle 
=\langle (A_{\psi_e}^{-1}(q_i-A_{\psi_e}q_j-t_{\psi_e}))\wedge q_j,\omega_j\rangle 
=\langle (A_{\psi_e}^{-1}(q_i-t_{\psi_e}))\wedge q_j,\omega_j\rangle.
\]
Thus, (\ref{eq:body_bar_const42}) can be written by
%\begin{equation*}
%\langle q_i-(A_{\psi_e}q_j+t_{\psi_e}),
%\dot{p}_i\rangle
%-\langle A_{\psi_e}^{-1}(q_i-t_{\psi_e})-q_j,
%\dot{p}_j\rangle
%-\langle (A_{\psi_e} q_j+t_{\psi_e})\wedge q_i,\omega_i\rangle
%-\langle (A_{\psi_e}^{-1}(q_i-t_{\psi_e}))\wedge q_j,\omega_j\rangle=0
%\end{equation*}
%or equivalently, 
\begin{equation}
\label{eq:body_bar_const5}
\langle q_i-\psi_eq_j,
\dot{p}_i\rangle
-\langle \psi_e^{-1}q_i-q_j,
\dot{p}_j\rangle
-\langle (\psi_eq_j)\wedge q_i,\omega_i\rangle
-\langle (\psi_e^{-1}q_i)\wedge q_j,\omega_j\rangle=0.
\end{equation}
The pair $(\omega_i,\dot{p}_i)$  is conventionally called 
an {\em infinitesimal motion} (or a {\em screw motion}) of body $i$,
and the set of all infinitesimal motions of each body forms a ${d+1 \choose 2}$-dimensional linear space, 
which can be identified with $\bigwedge^2 \mathbb{R}^{d+1}$. 

Thus, our problem is formulated as follows.
For a $\Gamma$-gain graph $(G,\psi)$ and $q_{e,i}, q_{e,j}\in \mathbb{R}^d$ for each $e=(i,j)\in E$,
an {\em infinitesimal motion} is defined by $s:i\in V\mapsto s_i=(\omega_i,\dot{p}_i)\in \bigwedge^2\mathbb{R}^{d+1}$
satisfying 
\begin{equation}
\label{eq:body_bar_const5_2}
\langle q_{e,i}-\psi_eq_{e,j},
\dot{p}_i\rangle
-\langle \psi_e^{-1}q_{e,i}-q_{e,j},
\dot{p}_j\rangle
-\langle (\psi_eq_{e,j})\wedge q_{e,i},\omega_i\rangle
-\langle (\psi_e^{-1}q_{e,i})\wedge q_{e,j},\omega_j\rangle=0.
\end{equation}
over all $e=(i,j)\in E$, 
and we are asked to compute the dimension of the space of infinitesimal motions.

(\ref{eq:body_bar_const5_2}) can be further simplified. 
To see this, let $\rho:\Gamma\rightarrow GL(\mathbb{R}^d)$ be 
the linear representation of $\Gamma$ by augmented $(d+1)\times (d+1)$-matrices,
i.e., $\rho(\gamma)=\begin{pmatrix} A_{\gamma} & t_{\gamma} \\ 0 & 1\end{pmatrix}$ 
for $\gamma=(A_{\gamma},t_{\gamma})\in \Gamma$.
Also, for $q\in \mathbb{R}^d$, denote $\bar{q}=\begin{pmatrix} q \\ 1 \end{pmatrix}\in \mathbb{R}^{d+1}$.
Since $\bar{q}_1\wedge \bar{q}_2=(q_1\wedge q_2, q_1-q_2)$ for any $q_1,q_2\in \mathbb{R}^d$,
if we denote the pair $(\omega_j,\dot{p}_j)$ by $s_j\in \bigwedge^2 \mathbb{R}^{d+1}$, 
(\ref{eq:body_bar_const5_2}) becomes
\begin{equation*}
\langle 
\bar{q}_{e,i}\wedge \rho(\psi_e)\bar{q}_{e,j}, s_i\rangle
-
\langle 
\rho(\psi_e)^{-1}\bar{q}_{e,i}\wedge \bar{q}_{e,j}, s_j\rangle
=0.
\end{equation*}
We may also replace $\bar{q}_{e,i}$ by $\rho(\psi_e)\bar{q}_{e,i}$,
leading to a system of linear equations of $s:i\in V\mapsto s_i\in \bigwedge^2 \mathbb{R}^{d+1}$,
\begin{equation}
\label{eq:body_bar_row2}
\langle 
 \rho(\psi_e)\bar{q}_{e,i}\wedge \rho(\psi_e)\bar{q}_{e,j}, s_i\rangle
-
\langle 
\bar{q}_{e,i}\wedge \bar{q}_{e,j}, s_j\rangle
=0 \qquad \forall e=(i,j)\in E.
\end{equation}

%Computing the dimension of the kernel of (\ref{eq:body_bar_row2}) is equivalent to computing that of the following system for a $\Gamma$-gain graph and $p_{e,i},p_{e,j}\in \mathbb{R}^{d+1}$ for $e=(i,j)\in E(G)$:
%\begin{equation}
%\label{eq:body_bar_row3}
%\langle 
% \rho(\psi_e)p_{e,i} \wedge \rho(\psi_e)p_{e,j}, s(i)\rangle
%-
%\langle 
%p_{e,i}\wedge p_{e,j}, s(j)\rangle
%=0 \qquad \forall e=(i,j)\in E(G).
%\end{equation} 

Define $\hat{x}_{e,\psi}^{(2)}$ by 
\begin{equation*}
\hat{x}_{e,\psi}^{(2)}(v)=
\begin{cases}
-\rho(\psi_e)\bar{q}_{e,i}\wedge \rho(\psi_e)\bar{q}_{e,j} & \text{ if } v=i \\
\bar{q}_{e,i}\wedge \bar{q}_{e,j} & \text{ if } v=j \\
0 & \text{ otherwise}
\end{cases}
\end{equation*}
for a non-loop edge $e=(i,j)$, and for each loop attached to a vertex $i$
\begin{equation*}
\hat{x}_{e,\psi}^{(2)}(v)=
\begin{cases}
\rho(\psi_e)\bar{q}_{e,i} \wedge \rho(\psi_e)\bar{q}_{e,j}-
\bar{q}_{e,i}\wedge \bar{q}_{e,j} & \text{ if } v=i \\
0 & \text{ otherwise}.
\end{cases}
\end{equation*}
Observe that $s$ is a solution of (\ref{eq:body_bar_row2}) if and only if 
$s$ is in the orthogonal complement of $\spa\{\hat{x}_{e,\psi}^{(2)}\mid e\in E\}$.
However, since $\hat{x}_{e,\psi}^{(2)}$ is a special case of 
(\ref{eq:decomposable_vector})(\ref{eq:decomposable_vector_loop}) given in the last subsection, 
we can apply Theorem~\ref{thm:exterior} to compute the exact value of 
$\dim_{\dR}\{\hat{x}_{e,\psi}^{(2)}\mid e\in E\}$ if a bar-configuration $q$ is generic.
(Although the last coordinate is restricted to 1 in each $\bar{q}_{e,i}$, 
we can still apply Theorem~\ref{thm:exterior}, 
as $\dim_{\dR}\{\hat{x}_{e,\psi}^{(2)}\mid e\in E\}$ is invariant up to scalar multiples of $\bar{q}_{e,i}$.) 
In terms of the infinitesimal rigidity of $\Gamma$-symmetric body-bar frameworks,
we proved  the following.
\begin{theorem}
\label{thm:body_bar}
Let $\Gamma$ be a discrete point group or a space group,
and $H$ a $\Gamma$-symmetric graph.
Then, for almost all body-configurations $B$ and bar-configurations $q$, 
the $\Gamma$-symmetric body-bar framework $(H,B,q)$ is symmetry-forced infinitesimally rigid if and only if
the quotient $\Gamma$-gain graph contains an edge subset $I$ satisfying the following,
\begin{itemize}
\item $|I|=D|V|-D+\dim_{\dR}\{{\rm image}(I_D-\rho^{(2)}(\gamma))\mid \gamma \in \Gamma\}$;
\item for any $F\subseteq I$, $|F|\leq D|V(F)|-Dc(F)
+\sum_{X\in C(F)}\dim_{\dR}\{{\rm image}(I_D-\rho^{(2)}(\gamma))\mid \gamma \in \langle X\rangle \})$
\end{itemize}
where $D={d+1\choose 2}$, 
$\rho$ is a linear representation of $\Gamma$
by augmented $(d+1)\times (d+1)$-matrices $\rho(\gamma)=\begin{pmatrix} A_{\gamma} & t_{\gamma} \\ 0 & 1\end{pmatrix}$, 
and $\rho^{(2)}:\Gamma\rightarrow GL(\bigwedge^2\mathbb{R}^{d+1})$ is a linear representation of $\Gamma$ defined by 
 $\rho^{(2)}(\gamma)(v_1\wedge v_2)=\rho(\gamma)v_1\wedge\rho(\gamma)v_2$ for $v_1,v_2\in \mathbb{R}^{d+1}$ and 
$\gamma\in \Gamma$.
\end{theorem}

As a special case when $\Gamma$ is a group of translations, Theorem~\ref{thm:body_bar} verifies
a conjecture by Ross~\cite{ross2012rigidity}.

%Let $\mathbf{L}_{b}(G)$ denotes the  
%linear matroid induced on \ref{eq:body_bar_row}.
%Our goal is to show that $\mathbf{L}_b(G)$ is equal to $\mathbf{M}(f_D)$.
%Notice that $\mathbf{L}_b(G)$ is the generic matroid obtained from the linear polymatroid induced by 
%flats,
%\begin{equation}
%\label{eq:A_non_loop}
%\hat{A}_{e,\psi}^{(k)}=\{(\Bvector{}{0},\Bvector{\cdots}{\cdots}, \Bvector{}{0},  
%\Bvector{u}{q_1^{\uparrow}\wedge \psi_eq_2^{\uparrow}},  \Bvector{}{0}, \Bvector{\cdots}{\cdots}, 
%\Bvector{}{0}, 
%\Bvector{v}{-\psi_e^{-1}q_1^{\uparrow}\wedge q_2^{\uparrow}}, \Bvector{}{0}, \Bvector{\cdots}{\cdots}, \Bvector{}{0}) \colon 
%q_1,q_2\in \mathbb{R}^d\},
%\end{equation}
%and for a loop $e$, 
%\begin{equation}
%\label{eq:A_loop}
%\hat{A}_{e,\psi}=\{x\in (\mathbb{R}^D)^V\mid \}
%
%\{(\Bvector{}{0},\Bvector{\cdots}{\cdots}, \Bvector{}{0},  
%\Bvector{a}{(I-\rho{\psi}(e)^{-1})\alpha},  \Bvector{}{0}, \Bvector{\cdots}{\cdots}, 
%\Bvector{}{0}) \colon 
%\alpha\in \mathbb{K}^d\}\subset L_a^d.
%\end{equation}

\section{Generalization of Lift Matroids}
\label{sec:lift}
In \cite{zaslavsky1991biased}, Zaslavsky also introduced another matroid of gain graphs, 
called {\em lift matroids}. 
This matroid is a special case of  {\em elementary lifts} of graphic matroids,
(see e.g.,~\cite{oxley} for elementary lifts).
It was shown by Zaslavsky~\cite{zaslavsky2003biased} that a lift matroid is representable over $\mathbb{F}$ 
if the underlying group is isomorphic to an additive subgroup of $\mathbb{F}$.
In this section, we shall propose an extension of lift matroids. 

\subsection{Lift matroids}
Let $(G=(V,E),\psi)$ be  a $\Gamma$-gain graph.
In the {\em lift matroid} $\mathbf{L}(G,\psi)$ of $(G,\psi)$, $F\subseteq E$ is independent if and only if 
there is at most one cycle, which  is unbalanced if exists~\cite{zaslavsky1991biased}.
Therefore, if we define $\ell_{\Gamma}:2^E\rightarrow \mathbb{Z}$ by
\begin{equation}
\label{eq:lift}
\ell_{\Gamma}(F)=|V(F)|-c(F)+\alpha_{\Gamma}(F) \qquad (F\subseteq E),
\end{equation}
where $\alpha_{\Gamma}$ is as defined in (\ref{eq:a2}),
then $\ell_{\Gamma}$ is the rank function of $\mathbf{L}(G,\psi)$.

Suppose that $\Gamma$ is an additive subgroup of $\mathbb{F}$.
We shall add a special new element $\ast$ to $V$,
and consider a linear representation  given by 
$e\in E\mapsto L_e\subseteq \mathbb{F}^{V\cup\{\ast\}}$ with
\begin{equation*}
L_e=\left\{x\in\mathbb{F}^{V\cup \{\ast\}} \Bigg\vert 
\begin{array}{l}
x(i)+x(j)=0, \\ 
\psi(e)x(i)+x(\ast)=0, \\
x(V\setminus \{i,j\})=0
\end{array}
\right\}.
\end{equation*}
This gives a linear representation of $\mathbf{L}(G,\psi)$,
called the {\em canonical representation} of $\mathbf{L}(G,\psi)$~\cite[Theorem~4.1]{zaslavsky2003biased}.
It is also known that any representation of $\mathbf{L}(K_n^{\bullet},\psi^{\bullet})$ 
is of this form
(see \cite[\textsection4]{zaslavsky2003biased} for more detail).

\subsection{Generalized lift matroids}
The idea of our extension of lift matroids is the same as the case of frame matroids;
instead of $\alpha_{\Gamma}(F)$, we consider a submodular function over $\Gamma$.

Suppose that $(G=(V,E),\psi)$ is a $\Gamma$-gain graph with an abelian group $\Gamma$.
We consider a symmetric polymatroidal function $\mu:2^{\Gamma}\rightarrow \mathbb{R}_+$
over $\Gamma$ (see \textsection\ref{subsec:fractional_lifting} for definition).
%Since $\Gamma$ is abelian, we automatically have the invariance under conjugates.

For $F\subseteq E$, we define $\llangle F\rrangle$ by
\[
 \llangle F\rrangle=\langle \psi(W)\mid W\in \pi_1(F,v), v\in V\rangle,
\]
i.e., the group generated by gains of all closed walks in $F$.
Since $\Gamma$ is abelian, if $F$ is connected, 
$\llangle F\rrangle=\langle F\rangle_v$ for any $v\in V(F)$ by Proposition~\ref{prop:conjugate}.
We then define $\ell_{\mu}:2^E\rightarrow \mathbb{R}$ by
\begin{equation}
\label{eq:generalized_lift}
 \ell_{\mu}(F)=|V(F)|-c(F)+\mu\llangle F\rrangle \qquad (F\subseteq E),
\end{equation}
where $\mu\llangle F\rrangle$ is an abbreviation of $\mu(\llangle F\rrangle)$.
As in Theorem~\ref{thm:submodularity1}, we have the following.
\begin{theorem}
\label{thm:submodularity_lift}
Let $(G=(V,E),\psi)$ be a $\Gamma$-gain graph with an abelian group $\Gamma$,
and $\mu$ be a symmetric polymatroidal function over $\Gamma$.
If $\mu(\gamma)\leq 1$ for every $\gamma\in \Gamma$,
then $\ell_{\mu}$ is monotone submodular.
\end{theorem}
\begin{proof}
For each $X\subseteq E$ and $e=(i,j)\in E\setminus X$, 
let $\Delta(X,e)=\ell_{\mu}(X\cup \{e\})-\ell_{\mu}(X)$,
and denote by $X_i$ and $X_j$ the connected components of $X$ for which $i\in V(X_i)$ and $j\in V(X_j)$,
each of which is an empty set if such a component does not  exist.
By a simple calculation, we have the following relation: 
\begin{align}
\label{eq:g4}
\Delta(X,e)=\begin{cases}
\mu\llangle X\cup\{e\}\rrangle-\mu \llangle X\rrangle & \text{ if $e$ is a loop or } X_i=X_j\neq\emptyset  \\
\mu\llangle X\cup\{e\}\rrangle+1-\mu \llangle X\rrangle & \text{ otherwise}. 
\end{cases}
\end{align}
However, since $\Gamma$ is abelian, 
it can be easily checked that $\llangle X\cup \{e\}\rrangle=\llangle X\rrangle$
if the later condition holds in (\ref{eq:g4}).
Therefore, we actually have
\begin{align}
\label{eq:g5}
\Delta(X,e)=\begin{cases}
\mu\llangle X\cup\{e\}\rrangle-\mu \llangle X\rrangle & \text{ if $e$ is a loop or } X_i=X_j\neq \emptyset  \\
1 & \text{ otherwise}. 
\end{cases}
\end{align}

By (\ref{eq:g5}), the monotonicity of $\mu$ over $\Gamma$ implies that 
$\Delta(X,e)\geq 0$. 
Thus, $\ell_{\mu}$ is monotone.

To see the submodularity, we claim the following.
\begin{claim}
\label{claim:submo_lifting}
Let  $X\subseteq E$, $e=(i,j)\in E\setminus X$,
and $F$ a maximal forest in $X$.
Suppose that $\psi(e)$ is identity for $e\in F$.
Then, 
\begin{align*}
\Delta(X,e)=\begin{cases}
\mu(\{\psi(f)\mid f\in X\cup\{e\}\})-\mu(\{\psi(f)\mid f\in X\}) & 
\text{ if $e$ is a loop or } X_i=X_j\neq \emptyset  \\
			 1 & \text{ otherwise}. 
\end{cases}
\end{align*}
Moreover, $\Delta(X,e)\leq 1$.
\end{claim}
\begin{proof}
By Proposition~\ref{lem:checking_label}, 
\begin{equation}
\label{eq:submo_lifting}
\llangle X\rrangle=\langle \psi(f)\mid f\in X\rangle 
\quad \text{ and } \quad
\llangle X\cup\{e\}\rrangle=\langle \psi(f)\mid f\in X\cup\{e\}\rangle.
\end{equation}
By the invariance of $\mu$ under taking closure,
putting (\ref{eq:submo_lifting}) into (\ref{eq:g5}),
we obtain the former relation of the statement.

To see the latter claim, observe that, if $e$ is a loop or $X_i=X_j\neq \emptyset$ holds,
then
$\Delta(X,e)=
\mu(\{\psi(f)\mid f\in X\cup\{e\}\})-\mu(\{\psi(f)\mid f\in X\})
\leq \mu(\psi(e))-\mu(\emptyset)\leq 1$,
where the second inequality follows from the submodularity of $\mu$ over $\Gamma$
and the third one follows from $\mu(\psi(e))\leq 1$ and $\mu(\emptyset)=0$.
\end{proof}

To see the submodularity of $\ell_{\mu}$,
let us check $\Delta(X,e)\geq \Delta(Y,e)$ for any $X\subseteq Y\subseteq E$ and $e\in E\setminus Y$.
Since $\Delta(Y,e)\leq 1$ by Claim~\ref{claim:submo_lifting}, 
it suffices to consider the case when $\Delta(X,e)<1$, i.e., $e$ is a loop or $X_i=X_j\neq \emptyset$.
If $X_i=X_j\neq \emptyset$, then $Y_i=Y_j\neq \emptyset$ as well.
Hence, 
$\Delta(X,e)\geq \Delta(Y,e)$ directly follows from Claim~\ref{claim:submo_lifting} 
and the submodularity (\ref{eq:submodular_ineq}) of $\mu$ over $\Gamma$.
\end{proof}

As in the case of gain matroids, let us focus on rational functions $\mu$,
i.e., $\mu:2^{\Gamma}\rightarrow \{0,\frac{k}{d},\dots, \frac{d-1}{d}k,k\}$ for some positive integers
$d$ and $k$. Then, 
$d\ell_{\mu}$ is a normalized integer-valued monotone submodular function,
and hence $(E,d\ell_{\mu})$ is a polymatroid.    

\begin{example}
\label{ex:lifting1}
Let us consider $\mathbb{Z}^d$-gain graph $(G=(V,E),\psi)$.
If we define $\mu$ by $\mu(X)=\dim_{\mathbb{R}} X$ for $X\subseteq \mathbb{Z}^d$,
then $\mu$ is a symmetric polymatroidal function.
Hence, $\ell_{\mu}(F)=|V(F)|-c(F)+\mu\llangle F\rrangle$ is monotone submodular.
Since $\ell_{\mu}(e)\leq 1$ for $e\in E$, $\ell_{\mu}$ is indeed a rank function of a matroid on $E$.
\end{example}

\begin{example}
\label{ex:lifting2}
Let us consider $\mathbb{Z}^d$-gain graph $(G=(V,E),\psi)$, again.
Define $\mu$ by 
$\mu(X)=\dim_{\mathbb{R}}\{\alpha\otimes \gamma\mid \alpha\in \mathbb{R}^d, \gamma\in X\}/d$ for $X\subseteq \mathbb{Z}^d$,
where $\alpha\otimes \gamma$ denotes the tensor product of $\alpha$ and $\gamma$.
Then, $\mu$ is a symmetric polymatroidal function with $\mu(\gamma)\leq 1$ for every $\gamma\in \Gamma$.
Therefore, $d\ell_{\mu}$ is monotone submodular.
Actually, the resulting  polymatroid is just the sum of $d$ copies of the matroid given in Example~\ref{ex:lifting1}.
\end{example}

\begin{remark}
Note that lifting matroids can be defined on $\Gamma$-gain graphs with any group $\Gamma$,
whereas we assumed in the above extension that $\Gamma$ is abelian. 
In fact, Theorem~\ref{thm:submodularity_lift} holds even for non-abelian group $\Gamma$, if 
$\mu\llangle \cdot \rrangle$ is invariant under switchings, which is the case of lifting matroids.
%This property follows, e.g., if $\mu:2^\Gamma\rightarrow \mathbb{R}_+$ satisfies $\mu(X\cup \gamma X\gamma^{-1})=\mu(X)$
%for any $X\subseteq \Gamma$ and any $\gamma \in \Gamma$.
\end{remark}

\subsection{Linear representations of generalized lift matroids}
We now give an extension of the canonical representation of $\mathbf{L}(G,\psi)$.
Let $(G,\psi)$ be a $\Gamma$-gain graph,
and suppose that $\Gamma$ is an additive subgroup of a vector space $\mathbb{F}^{t}$ over $\mathbb{F}$.

For a bilinear map $b:\mathbb{F}^d \times \mathbb{F}^t \rightarrow \mathbb{F}^{k}$,
we define $\mu_b:2^{\Gamma}\rightarrow \mathbb{Z}$ as follows:
\begin{equation}
\mu_b(X)=\dim_{\dF}\{b(\alpha,\gamma)\mid \alpha\in \mathbb{F}^d,\gamma\in X\} \qquad (X\subseteq \Gamma).
\end{equation}
Then, it is easy to check that $\mu_b$ is a symmetric polymatroid function over $\Gamma$.
Also, for any $\gamma\in \Gamma$, we have $\mu_b(\gamma)\leq d$.
Therefore, by Theorem~\ref{thm:submodularity_lift}, the following  function $f_b$
 induces a polymatroid of a $\Gamma$-gain graph $(G=(V,E),\phi)$,
\begin{equation}
f_b(F)=d|V(F)|-dc(F)+\mu_b\llangle F\rrangle \qquad (F\subseteq E).
\end{equation}
For example, if setting $b:\mathbb{F}\times \mathbb{F}^d\rightarrow \mathbb{F}^d$ to be 
$b:(\alpha,\gamma)\mapsto \alpha\gamma$, we have the case of Example~\ref{ex:lifting1}. 

We now show a linear representation of the (poly)matroid induced by $f_b$.
With each edge $e=(i,j)\in E$, we associate a 
linear space 
\begin{equation*}
L_{e,\psi}=\left\{x\in (\mathbb{F}^d)^{V} \oplus \mathbb{F}^{k}\Bigg\vert 
\begin{array}{l}
x(i)+x(j)=0, \\
b(x(i),\psi(e))+x(\ast)=0,\\ 
x(V\setminus \{i,j\})=0
\end{array}
\right\}
\end{equation*}
if $e$ is not a loop, and 
\begin{equation*}
 L_{e,\psi}=\left\{x\in (\mathbb{F}^d)^{V}\oplus \mathbb{F}^{k}\Big\vert \exists \alpha\in \mathbb{F}^d \colon
\begin{array}{l}
x(\ast)=-b(\alpha,\psi(e)), \\ 
x(V)=0
\end{array}
\right\}.
\end{equation*}
if $e$ is a loop, where $(\mathbb{F}^d)^{V} \oplus \mathbb{F}^{k}$ is an abbreviation of  
$(\mathbb{F}^d)^{V} \oplus (\mathbb{F}^{k})^{\{\ast\}}$ used throughout subsequent discussions. 

We consider a linear polymatroid induced on $\{L_{e,\psi}\mid e\in E\}$.
Clearly, it depends on $\psi$, but 
as in Lemma~\ref{lem:linear_equivalence} the rank of the polymatroid is invariant up to equivalence.
\begin{lemma}
\label{lem:linear_equivalence2}
Let $\psi$ and $\psi'$ be equivalent gain functions.
Then, 
$\dim_{\dF}\{L_{e,\psi}\mid e\in E\})=\dim_{\dF}\{L_{e,\psi'}\mid e\in E\}$.
\end{lemma}
\begin{proof}
%Let us simply denote $d_{\psi}=\dim(\spa\{L_{e,\psi}\mid e\in E\})$.
It is sufficient to show that the dimension is invariant from any switch operation.

Suppose that $\psi'$ is obtained from $\psi$ by a switch operation at $v$ with $\gamma\in \Gamma$.
We may assume that all of edges incident to $v$ is oriented to $v$.
Then,  $\psi'(e)=\psi(e)-\gamma$ if $e$ is incident to $v$;
otherwise $\psi'(e)=\psi(e)$.
Note that, since $\Gamma$ is abelian, $\psi'(e)=\psi(e)$ for any loop $e$.

Consider a bijective linear transformation 
$T:(\mathbb{F}^d)^{V}\oplus \mathbb{F}^{k}\rightarrow (\mathbb{F}^d)^{V}\oplus \mathbb{F}^{k}$ defined by
$T(x)(w)=x(w)$ for $w\in V$ and $T(x)(\ast)=b(x(v),\gamma)+x(\ast)$ for the special vertex $\ast$.
We then have 
\begin{align*}
TL_{e,\psi}&=
\left\{T(x)\in (\mathbb{F}^d)^V\oplus \mathbb{F}^k \Bigg\vert 
\begin{array}{l}
b(x(i),\psi(e))+x(\ast)=0, \\
x(i)+x(j)=0, \\ 
x(V\setminus \{i,j\})=0
\end{array}
\right\} \\
&=\left\{y \in (\mathbb{F}^d)^V\oplus \mathbb{F}^k \Bigg\vert
\begin{array}{l}
b(y(i),\psi(e))-b(y(v),\gamma)+y(\ast)=0,\\ 
y(i)+y(j)=0, \\ 
y(V\setminus \{i,j\})=0
\end{array}
\right\}.
\end{align*}
%as $y(\ast)=T(x)(\ast)=f(x(v),\gamma)+x(\ast)=f(y(v),\gamma)+x(\ast)$.
Since $b(y(v),\gamma)=0$ if $v\neq i,j$,
we have $b(y(i),\psi(e))-b(y(v),\gamma)=b(y(i),\psi'(e))$ for any $e=(i,j)\in E$.
Thus, $TL_{e,\psi}=L_{e,\psi'}$,
implying the lemma.
\end{proof}

The following theorem is a counterpart of Theorem~\ref{thm:main_poly}, whose proof is almost identical.
\begin{theorem}
\label{thm:lifting_rep}
Let $(G=(V,E),\psi)$ be a $\Gamma$-gain graph
with  an additive subgroup $\Gamma$ of $\mathbb{F}^t$. 
Define $f_b$ and $L_{e,\psi}$ as above.
Then, 
\[
 f_b(E)=\dim_{\dF}\{L_{e,\psi}\mid e\in E\}.
\]
\end{theorem}
\begin{proof}
Let $T$ be a maximal forest in $E$.
By   Proposition~\ref{prop:tree_identity} and Lemma~\ref{lem:linear_equivalence2}, 
we may assume that $\psi(e)=0$ for $e\in T$.
Since $\Gamma$ is abelian, Proposition~\ref{prop:conjugate} and Proposition~\ref{lem:checking_label} imply that 
$\llangle E\rrangle=\langle \psi(e)\mid e\in E\setminus T\rangle$,
and hence 
\[
 f_b(E)=d|V(E)|-dc(E)+\dim_{\dF}\{b(\alpha,\psi(e))\mid \alpha\in\mathbb{F}^d, e\in E\setminus T\}.
\]

Let $L_T=\spa\{L_{e,\psi}\mid e\in T\}$.
By Lemma~\ref{lem:graphic}, it follows that 
(i) $\dim_{\dF} L_T=d|V(E)|-dc(E)$ and (ii)
each quotient space $L_{e,\psi}/L_T$ for $e\in E\setminus T$ is written by
\begin{align*}
\{x+L_T\mid \exists\alpha\in \mathbb{F}^d\colon b(\alpha, \psi(e))+x(\ast)=0, 
x(V)=0 \}, 
\end{align*}
which is  isomorphic to
$\{b(\alpha, \psi(e))\mid \alpha\in \mathbb{F}^d \}$. 
Therefore, $\dim_{\dF}\{L_{e,\psi}\mid e\in E\}=
\dim_{\dF} L_T+\dim_{\dF}\{L_{e,\psi}/L_T\mid e\in E\setminus T\}
=d|V(E)|-dc(E)+\dim_{\dF} \{b(\alpha,\psi(e))\mid \alpha\in \mathbb{F}^d, e\in E\setminus T\}=f_b(E)$.
\end{proof}

Let $(G,\psi)$ be a $\Gamma$-gain graph.
With each  $e=(i,j)\in E(G)$, we associate a vector 
$y_{e,\psi}$ from $L_{e,\psi}$ so that 
$\{y_{e,\psi}\mid e\in E(G)\}$ is in generic position, by extending 
the underlying field to $\mathbb{K}^d$. 
The following is an immediate consequence of Theorem~\ref{theorem:flat_matroid} and Theorem~\ref{thm:lifting_rep}.
\begin{corollary}
\label{cor:lifting_rep}
Let $(G,\psi)$ be a $\Gamma$-gain graph
with an additive subgroup $\Gamma$ of $\mathbb{F}^t$.
Let $b:\mathbb{F}^d\times \mathbb{F}^t\rightarrow \mathbb{F}^{k}$ be a bilinear map.
Then, $\{y_{e,\psi}\mid e\in E(G)\}$ is linearly independent in $\dK^d$ if and only if 
for any $F\subseteq E$ 
\[
 |F|\leq d|V(F)|-dc(F)+\dim_{\dF}\{b(\alpha,\gamma)\mid \alpha\in \mathbb{F}^d, \gamma\in\llangle F\rrangle\}.
\]
\end{corollary}

\subsection{Applications}
Let $(G,\psi)$ be a $\mathbb{Z}^d$-gain graph,
and let us define a bilinear map $b:\mathbb{R}^d\times \mathbb{R}^d\rightarrow \mathbb{R}^{d^2}$
by $b(\alpha, \gamma)=\alpha\otimes \gamma$.
For each $e=(i,j)\in E(G)$, we shall associate a vector $y_{e,\psi}\in (\mathbb{R}^d)^V\oplus \mathbb{R}^{d^2}$
with  
\begin{equation*}
 y_{e,\psi}(v)=\begin{cases}
\alpha_e & \text{ if } v=i \\
-\alpha_e & \text{ if } v=j \\
-\alpha_e\otimes \psi(e) & \text{ if } v=\ast \\
0 & \text{ otherwise}
\end{cases}
\end{equation*}
such that the set of all coordinates of $\alpha_e \ (e\in E(G))$ is algebraically independent
over $\mathbb{Q}$.
By Corollary~\ref{cor:lifting_rep},
$\{y_{e,\psi}\mid e\in E(G)\}$ is linearly independent if and only if 
for any $F\subseteq E$ 
\[
 |F|\leq d|V(F)|-dc(F)+d\dim_{\dR} \{\gamma\mid \gamma\in\llangle F\rrangle\}.
\]

As in \textsection~\ref{sec:applications}, 
it is easy to check that the restriction of $\{L_{e,\psi}\mid e\in E\}$ to a generic hyperplane
 gives rise to  the orbit rigidity matrix of a $\mathbb{Z}^2$-symmetric framework (called a periodic framework) when $d=2$
or to the linear representation of the $\mathbb{Z}^d$-symmetric parallel redrawing polymatroid of a 
$\mathbb{Z}^d$-symmetric framework  for general dimension $d$.
This implies that the independence in the associated linear (poly)matroid is characterized by the following counting condition;
For any nonempty $F\subseteq E(G)$ 
\[
 |F|\leq d|V(F)|-dc(F)+d\dim_{\dR}\{\gamma\mid \gamma\in\llangle F\rrangle\}-1.
\]
This is an alternative proof of results by Malestein and Theran~\cite{malestein2010} for $d=2$.
%Since we give a more general result in \textsection~?, we omit the detail here. 

\section{Toward unified matroids}
\label{sec:unified}
Although we have no clear idea on how to unify the extension of frame matroids and 
that of lift matroids via their rank functions,
the canonical representations tell us a natural approach to unify representation theory obtained so far.
To see this, in this section, we shall focus on subgroups of $GL(\mathbb{F}^d)\ltimes \mathbb{F}^d$,
that is, the semidirect product of $GL(\mathbb{F}^d)$ and $\mathbb{F}^d$
with product $(g,z)\cdot (g',z')=(g g', g z'+z)$.

Let $\Gamma$ be a subgroup of $GL(\mathbb{F}^d)\ltimes \mathbb{F}^d$.
The projection of $\Gamma$ to the first component, i.e., $(g,z)\mapsto g$, is a group homomorphism,
and hence the image $\{g\mid (g,z)\in \Gamma\}$ forms a subgroup of $GL(\mathbb{F}^d)$, 
called the {\em linear part} of $\Gamma$ and denoted by $\Gamma_1$.

Let $(G=(V,E),\psi)$ be a $\Gamma$-gain graph 
with a gain function $\psi=(\psi_1,\psi_2)$,
and let $b:\mathbb{F}^d \times \mathbb{F}^d \rightarrow \mathbb{F}^{k}$ be a bilinear map
such that $\Gamma_1$ is unitary with respect to $b$, i.e., $b(g x, y)=b(x,g^{-1} y)$ 
for any $g \in \Gamma_1$ and any $x,y$. 
Combining the idea of \textsection\ref{sec:dowling_extension} and \textsection\ref{sec:lift},
we now associate a linear subspace with each edge $e=(i,j)\in E$ as follows:
\begin{equation}
\label{eq:unified_U1}
U_{e,\psi}=\left\{x\in (\mathbb{F}^d)^{V}\oplus \mathbb{F}^{k} \Bigg\arrowvert 
\begin{array}{lll} 
x(i)+\psi_1(e)x(j)=0, \\
x(\ast)=-b(x(i),\psi_2(e)), \\
x(V\setminus \{i,j\})=0 \end{array}
\right\}
\end{equation}
if $e$ is not a loop, and 
\begin{equation}
\label{eq:unified_U2}
 U_{e,\psi}=\left\{x\in (\mathbb{F}^d)^{V}\oplus \mathbb{F}^{k} \Bigg\arrowvert \exists \alpha\in \mathbb{F}^d \colon
\begin{array}{lll} 
x(i)=(I_d-\psi_1(e))\alpha, \\ 
x(\ast)=-b(\alpha,\psi_2(e)), \\ 
x(V\setminus \{i\})=0
\end{array}
\right\}
\end{equation}
if $e$ is a loop attached to $i$.

Note that $U_{e,\psi}$ is invariant from the reorientation of $e$,
as 
\begin{align*}
U_{e,\psi}&=\{x\mid 
x(i)+\psi_1(e)x(j)=0,
b(x(i),\psi_2(e))+x(\ast)=0, x(V\setminus \{i,j\})=0\} \\
&=\{x\mid 
\psi_1(e)^{-1}x(i)+x(j)=0,
b(-\psi_1(e)x(j),\psi_2(e))+x(\ast)=0, x(V\setminus \{i,j\})=0\} \\
&=\{x\mid 
\psi_1(e)^{-1}x(i)+x(j)=0,
b(x(j),-\psi_1(e)^{-1}\psi_2(e))+x(\ast)=0, x(V\setminus \{i,j\})=0\}
\end{align*}
where $\psi(e)^{-1}=(\psi_1(e)^{-1},-\psi_1(e)^{-1}\psi_2(e))$.
Although $U_{e,\psi}$ depends on the choice of gain functions $\psi$, 
as in the previous cases, 
the rank of the polymatroid induced on $\{U_{e,\psi}\mid e\in E\}$ is invariant up to equivalence.
\begin{lemma}
\label{lem:linear_equivalence3}
Let $\psi$ and $\psi'$ be equivalent gain functions. Then, 
$\dim_{\dF} \{U_{e,\psi}\mid e\in E\}=\dim_{\dF}\{U_{e,\psi'}\mid e\in E\}$.
\end{lemma}
\begin{proof}
%Let us simply denote $d_{\psi}=\dim\{{\spa\{U_{e,\psi}\mid e\in E\}}\}$.
%It is sufficient to show that $d_{\psi}$ is invariant from any switch operation.

Suppose that $\psi'$ is obtained from $\psi$ by a switch operation at $v$ with 
$\gamma=(g,z)\in \Gamma$. 
Since $U_{e,\psi}$ is invariant from reorientation of $e$, we may assume that 
all of the edges incident to $v$ are oriented from $v$.
Then, $\psi'(e)=\gamma\psi(e)$ if $e$ is  a non-loop edge incident to $v$;
$\psi'(e)=\gamma\psi(e)\gamma^{-1}$ if $e$ is a loop incident to $v$;
otherwise $\psi'(e)=\psi(e)$.

Consider a bijective linear transformation 
$T:(\mathbb{F}^d)^{V}\oplus \mathbb{F}^{k}\rightarrow (\mathbb{F}^d)^{V}\oplus \mathbb{F}^{k}$ defined by,
for each $x\in (\mathbb{F}^d)^{V}\oplus \mathbb{F}^{k}$,
\begin{equation*}
T(x)(w)=\begin{cases}
x(w) & \text{ if } w\in V\setminus \{v\} \\
g x(v) & \text{ if } w=v \\
-b(x(v),g^{-1}z)+x(\ast) & \text{ if } w=\ast.
\end{cases}
\end{equation*}
We then have 
\begin{align*}
x(w)&=T(x)(w) \quad \text{ for } w\in V\setminus \{v\}, \\
x(v)&=g^{-1}T(x)(v), \\
x(\ast)&=T(x)(\ast)+b(x(v),g^{-1}z)=T(x)(\ast)+b(T(x)(v),z).
\end{align*}
Therefore, if $e$ is a non-loop edge oriented from $v$ to $j\in V$,
\begin{align*}
TU_{e,\psi}=\left\{y \in (\mathbb{F}^d)^V\oplus \mathbb{F}^k \Bigg\vert 
\begin{array}{lll}
y(v)+g\psi_1(e)y(j)=0, \\
b(y(v),g \psi_2(e)+z)+y(\ast)=0, \\
y(V\setminus \{v,j\})=0\end{array}
\right\}.
\end{align*}
As $\psi'(e)=(g\psi_1(e),g\psi_2(e)+z)$, 
we obtain that $TU_{e,\psi}=U_{e,\psi'}$.
Similarly, for  a loop $e$ attached to $v$,
\begin{align*}
TU_{e,\psi}&=\left\{y \in (\mathbb{F}^d)^V\oplus \mathbb{F}^k \Bigg\vert 
\exists\alpha\in \mathbb{R}^d\colon
\begin{array}{lll}
y(v)=g(I_d-\psi_1(e))\alpha, \\
y(\ast)=-b(\alpha,-\psi_1(e)^{-1}g^{-1}z+g^{-1}z+\psi_2(e)), \\
y(V\setminus \{v,j\})=0\end{array}
\right\} \\
&=\left\{y \in (\mathbb{F}^d)^V\oplus \mathbb{F}^k \Bigg\vert 
\exists \alpha\in \mathbb{R}^d\colon
\begin{array}{lll}
y(v)=(I_d-g\psi_1(e)g^{-1})\alpha, \\
y(\ast)=-b(\alpha,-g\psi_1(e)^{-1}g^{-1}z+g\psi_2(e)+z), \\
y(V\setminus \{v,j\})=0\end{array}
\right\} \\
&=U_{e,\psi'}.
\end{align*}

If $e=(i,j)$ is not incident to $v$, then we have $T(x)(i)=x(i)$, $T(x)(j)=x(j)$, and $T(x)(\ast)=x(\ast)$
by $x(v)=0$ for any $x\in U_{e,\psi}$,
and hence $TU_{e,\psi}=U_{e,\psi}=U_{e,\psi'}$. 
Thus, we obtain the lemma.
\end{proof}

By using Lemma~\ref{lem:linear_equivalence3}, we can now apply the same proof 
as Theorem~\ref{thm:main_poly} to show a combinatorial characterization of 
the polymatroid induced on $\{U_{e,\psi}\mid e\in E\}$.
To see this, we need a new terminology.
Consider $F\subseteq E$. 
Recall that $G[F]$ denotes the edge-induced subgraph $(V(F),F)$.
By Proposition~\ref{prop:tree_identity}, for a maximal forest $T$ of $F$, 
there is an equivalent gain function $\psi^{\circ}_F$ to $\psi$
such that $\psi^{\circ}_F(e)$ is identity  for all $e\in T$. 
A {\em compressed graph} by $F$ is defined as 
a $\Gamma$-gain graph $(G_F^{\circ},\psi^{\circ}_F)$, where $G^{\circ}_F$ is obtained from $G[F]$ 
by contracting each connected component to a single vertex,
where each edge $e$ of $F$ remains in $G^{\circ}_F$ as a loop with the gain $\psi^{\circ}_F(e)$.
By Proposition~\ref{prop:tree_identity} and Proposition~\ref{lem:checking_label}, 
$(G^{\circ}_F,\psi^{\circ}_F)$ is invariant from the choice of $T$ up to the equivalence of gain functions $\psi^{\circ}_F$. 

Applying the same proof as Theorem~\ref{thm:lifting_rep}, we now have the following result.
We omit the proof, which is identical to those of Theorem~\ref{thm:main_poly} and Theorem~\ref{thm:lifting_rep}.
\begin{theorem}
\label{thm:unified}
Let  $\Gamma$ be a subgroup of $GL(\mathbb{F}^d)\ltimes \mathbb{F}^d$,
and $\Gamma_1$ the projection of $\Gamma$ to $GL(\mathbb{F}^d)$.
Let $(G,\psi)$ be a $\Gamma$-gain graph, and $b:\mathbb{F}^d\times \mathbb{F}^d\rightarrow \mathbb{F}^k$
a bilinear map such that $\Gamma_1$ is unitary with respect to $b$. 
Then, for any $F\subseteq E(G)$,
\[
 \dim_{\dF}\{U_{e,\psi}\mid e\in F\}=d|V(F)|-dc(F)+\dim_{\dF}\{U_{e,\psi_F^{\circ}}\mid e\in E(G_F^{\circ})\}.
\]
\end{theorem}

\begin{remark}
Theorem~\ref{thm:unified} gives a good characterization of the dimension of 
$\spa\{U_{e,\psi}\mid e\in F\}$,
since computing $\dim_{\dF}\{U_{e,\psi_F^{\circ}}\mid e\in E(G_F^{\circ})\}$
can be reduced to the computation of the rank of a matrix of size $(dc(F)+k)\times d|F|$.
Hence, it is possible to compute $\dim_{\dF}\{U_{e,\psi}\mid e\in F\})$ deterministically in polynomial time.
\qed 
\end{remark}

%\begin{remark}
%The covering graph of each connected component of $(G^{\circ}_F,\psi^{\circ}_F)$ becomes the Cayley graph on
%generators $\{\psi_F^{\circ}(e)\mid e\in F\}$. \qed
%\end{remark}

\section{Further applications}
\label{sec:further}
As applications of Theorem~\ref{thm:unified}, we shall extend the result of \textsection\ref{sec:applications}
to symmetric frameworks with crystallographic symmetry.
For detailed analysis,
we first review basic facts on space groups in \textsection\ref{subsec:space}.
In \textsection\ref{subsec:cry_parallel} and \textsection\ref{subsec:cry_rigidity}, 
we discuss the parallel redrawing problem and the symmetry-forced rigidity, respectively.

\subsection{Space groups}
\label{subsec:space}
Recall that a {\em space group (or crystallographic group)} $\Gamma$ is a discrete cocompact subgroup of 
the Euclidean group ${\cal E}(d)$,
and  each element $(A,t)\in \Gamma$ acts on $\mathbb{R}^d$ by $(A,t)\cdot q=Aq+t$ for $q\in \mathbb{R}^d$.
An element of the form $(I_d,t)$ is called a {\em translation}, and is simply denoted by $t$.
As in the previous section, 
let $\Gamma_1=\{A_{\gamma}\mid \gamma\in \Gamma\}$, the projection to the first component.

%
%
%Let ${\rm Aff}(\mathbb{R}^d)$ be 
%the group of invertible affine transformations, 
%that is, ${\rm Aff}(\mathbb{R}^d)=GL(\mathbb{R}^d)\ltimes \mathbb{R}^d$,
%where each element $(A,t)\in {\rm Aff}(d)$ acts on $\mathbb{R}^d$ by $(A,t)\cdot q=Aq+t$ for $q\in \mathbb{R}^d$.
%An element of the form $(I_d,t)$ is called a {\em translation}, and is simply denoted by $t$.
%
%The $d$-dimensional Euclidean group $E(d)$ is a subgroup of ${\rm Aff}(\mathbb{R}^d)$,
%where $(A,t)\in {\rm Aff}(\mathbb{R}^d)$ is in $E(d)$ if and only if  $A\in {\cal O}(\mathbb{R}^d)$. 

The subgroup ${\cal L}_{\Gamma}$ consisting of all translations in $\Gamma$ is called the {\em lattice group} of $\Gamma$,
and it is known by Bieberbach's theorem that ${\cal L}_{\Gamma}$ is a normal subgroup of $\Gamma$ generated by $d$ linearly independent translations 
$t_1,\dots, t_d\in \mathbb{R}^d$.
The $d\times d$-matrix $B_{\Gamma}$ of the base transformation from the standard basis of $\mathbb{R}^d$
to $\{t_1,\dots, t_d\}$ is called a {\em lattice basis} of $\Gamma$ here. 
(Conventionally, a lattice basis of ${\cal L}_{\Gamma}$ means $\{t_1,\dots, t_d\}$, rather than $B_{\Gamma}$.)
Then, ${\cal L}_{\Gamma}=\{B_{\Gamma}z \mid z\in \mathbb{Z}^d\}$.

The quotient subgroup ${\cal K}_{\Gamma}=\Gamma/{\cal L}_{\Gamma}$ is known as {\em the point group of $\Gamma$}.
Since ${\cal K}$ acts on ${\cal L}_{\Gamma}$ and ${\cal L}_{\Gamma}$ is isomorphic to $\mathbb{Z}^d$, 
${\cal K}_{\Gamma}$ can be represented as integral matrices.
Therefore, in the subsequent discussion, ${\cal K}_{\Gamma}$ is regarded as a finite subgroup of $GL(\mathbb{Z}^d)$.
Note then that $B_{\Gamma}{\cal K}_{\Gamma}B_{\Gamma}^{-1}=\Gamma_1\subseteq {\cal O}(\mathbb{R}^d)$.
Indeed, using the lattice basis $B_{\Gamma}$,  
each element $\gamma=(A_{\gamma},t_{\gamma})$ of $\Gamma$  can be uniquely written by a triple 
$(K_{\gamma},z_{\gamma},c_{\gamma})\in GL(\mathbb{Z}^d)\times \mathbb{Z}^d \times [0,1)^d$,
where $A_{\gamma}=B_{\Gamma}K_{\gamma}B_{\Gamma}^{-1}$ and $t_{\gamma}=B_{\Gamma} (z_{\gamma}+c_{\gamma})$.
The representation in $GL(\mathbb{Z}^d)\times \mathbb{Z}^d \times [0,1)^d$ is sometimes called the {\em standard form}.
Note that a space group $\Gamma$ is determined by the standard form of each element and the lattice basis.

%A space group $\Gamma$ is called {\em symmorphic} if $c_{\gamma}=0$ for all $\gamma\in \Gamma$.
%If $\Gamma$ is symmorphic, we may write the standard form of each element by a pair in $GL(\mathbb{Z}^d)\times \mathbb{Z}^d$. 

By Bieberbach's theorem, two space groups $\Gamma$ and $\Gamma'$ are isomorphic if and only if they are conjugate 
by an affine transformation in ${\rm Aff}(\mathbb{R}^d)$.
Eliminating trivial motions, we focus on affine motions that change the lattice basis $B_{\Gamma}$ 
 without changing the linear part. 
%A motion of the space group among an equivalence class can be thus regarded as a motion of the lattice basis.
We hence define the {\em space of lattices} by 
\begin{equation}
{\rm Lat}(\Gamma)=\{B\in GL(\mathbb{R}^d) \mid \forall K_{\gamma}\in {\cal K}_{\Gamma}\colon BK_{\gamma}B^{-1}=B_{\Gamma}K_{\gamma}B_{\Gamma}^{-1}(=A_{\gamma})\},
\end{equation}
and here we say that $\Gamma$ and $\Gamma'$ are {\em equivalent} if $B_{\Gamma'}\in {\rm Lat}(\Gamma)$.
It is convenient to  consider a slightly larger set 
\begin{equation}
\label{eq:lattice_motion}
\overline{\rm Lat}(\Gamma)=\{B\in \mathbb{R}^{d\times d}\mid  \forall K_{\gamma}\in {\cal K}_{\Gamma}\colon BK_{\gamma}=A_{\gamma}B\}.
\end{equation}
Then, $\overline{{\rm Lat}}(\Gamma)$ is a linear space and 
${\rm Lat}(\Gamma)$ is a dense open subset of $\overline{{\rm Lat}}(\Gamma)$.
%For $B\in \overline{\rm Lat}(\Gamma)\setminus {\rm Lat}(\Gamma)$, 
%$\{(BK_{\gamma}B^{-1}, B(z_{\gamma}+c_{\gamma})) \mid \gamma\in \Gamma\}$ forms a lower dimensional space group,
%which we called a {\em collapse} of $\Gamma$.
%The {\em space  of lattice motions} is defined as $\overline{\rm Lat}(\Gamma)-B_{\Gamma}=\{M-B_{\Gamma}\mid M\in {\rm Lat}(\Gamma)\}$.
%$\dim_{\dR} \overline{\rm Lat}(\Gamma)$ is said to be {\em the degree of freedom of the lattice} of $\Gamma$.
%where the additional term expresses the degree of freedom of translations of the lattice 
%(i.e., the degree of freedom of the origin).

\subsection{Parallel redrawing with space group symmetry}
\label{subsec:cry_parallel}
Let us move to the $\Gamma$-symmetric parallel redrawing problem
for  a  space group $\Gamma$ with linear part $\Gamma_1$.
Let ${\cal L}_{\Gamma}$ be the lattice group of $\Gamma$ with a basis $B_{\Gamma}\in GL(\mathbb{R}^d)$.
Each element $\gamma\in \Gamma$ is denoted by $(A_{\gamma},t_{\gamma})\in {\cal O}(\mathbb{R}^d)\ltimes \mathbb{R}^d$, 
but we also use the standard form 
$(K_{\gamma},z_{\gamma},c_{\gamma})\in GL(\mathbb{Z}^d)\times \mathbb{Z}^d\times [0,1)^d$,
where $A_{\gamma}=B_{\Gamma}K_{\gamma}B_{\Gamma}^{-1}$ and $t_{\gamma}=B_{\Gamma} (z_{\gamma}+c_{\gamma})$.
%Also, let $\Gamma_1=\{A_{\gamma}\mid \gamma\in \Gamma\}$.

We consider a $\Gamma$-symmetric framework $(H,p)$, where 
$H$ is a $\Gamma$-symmetric graph and $p$ is a $\Gamma$-symmetric point-configuration.
We say that $(H,q)$ is a {\em symmetric parallel redrawing} of $(H,p)$
if $(H,q)$ is a parallel redrawing of $(H,p)$ and 
it is $\Gamma'$-symmetric for some equivalent space group $\Gamma'$ to $\Gamma$.

A {\em relocation} of $(H,p)$ is $m:V(H)\rightarrow \mathbb{R}^d$ such that
\begin{equation}
\label{eq:cry_parallel_const1}
m(i)-m(j) \text{ is parallel to } p(i)-p(j) \text{ for any } \{i,j\}\in E(H).
\end{equation} 
Regarding the  $\Gamma$-symmetry of $m$, we have two remarks.
(1) Since $m$ is a vector, rather than a point in the Euclidean space, 
only $\Gamma_1$ acts on the space of relocations.
(2) A framework can be also relocated by deforming the underlying lattice.
Thus, we say that a relocation $m$ is {\em $\Gamma$-symmetric} 
if there is $M\in \overline{\rm Lat}(\Gamma)$ such that $M+B_{\Gamma}\in {\rm Lat}(\Gamma)$ and 
\begin{equation}
\label{eq:cry_symmetric_motion}
 m(\gamma v)=A_{\gamma}m(v)+M(z_{\gamma}+c_{\gamma}) \qquad \forall v\in V(H), \forall \gamma\in \Gamma.
\end{equation}
The definition is justified by the following proposition.

\begin{proposition}
\label{prop:cry_parallel1}
Let $(H,p)$ be a $\Gamma$-symmetric framework with a  space group $\Gamma$.
If $m$ is a $\Gamma$-symmetric relocation of $(H,p)$, then $(H,p+m)$ is a symmetric parallel redrawing of $(H,p)$.
Conversely, if $(H,q)$ is a symmetric parallel redrawing, then 
$q-p$ is a $\Gamma$-symmetric relocation. 
\end{proposition}
\begin{proof}
Suppose that $m$ is a $\Gamma$-symmetric relocation.
Define $q$ by $q(v)=p(v)+m(v)$ for $v\in V(H)$.
Since $m$ is a relocation, $(H,q)$ is a parallel redrawing of $(H,p)$ by (\ref{eq:cry_parallel_const1}).
Also, since $m$ is $\Gamma$-symmetric,  there exists $M\in \overline{\rm Lat}(\Gamma)$
for which $M+B_{\Gamma}\in {\rm Lat}(\Gamma)$ and (\ref{eq:cry_symmetric_motion}) is satisfied.
Let $B=M+B_{\Gamma}$.
Then, for any $v\in V(H)$ and $\gamma\in \Gamma$,
\begin{align*}
q(\gamma v)&=p(\gamma v)+m(\gamma v)=A_{\gamma}p(v)+B_{\Gamma}(z_{\gamma}+c_{\gamma})
+A_{\gamma}m(v)+M(z_{\gamma}+c_{\gamma}) \\
&=A_{\gamma}(p(v)+m(v))+B(z_{\gamma}+c_{\gamma})=A_{\gamma}q(v)+B(z_{\gamma}+c_{\gamma}).
\end{align*}
Since $B\in {\rm Lat}(\Gamma)$, this implies that 
$(H,q)$ is $\Gamma'$-symmetric for an equivalent $\Gamma'$ to $\Gamma$.

Conversely, suppose that  $(H,q)$ is a symmetric parallel redrawing of $(H,p)$.
Then, $(H,q)$ is $\Gamma'$-symmetric for some equivalent $\Gamma'$ to $\Gamma$.
This means that there is $B\in {\rm Lat}(\Gamma)$ such that 
$B$ is a lattice basis of $\Gamma'$.
Setting $m=q-p$ and $M=B-B_{\Gamma}$, we see that $M\in \overline{\rm Lat}(\Gamma)$ and
for any $v\in V(H)$ and $\gamma \in \Gamma$, 
\begin{align*}
m(\gamma v)&=q(\gamma v)-p(\gamma v)=BK_{\gamma}B^{-1}q(v)+B(z_{\gamma}+c_{\gamma})-
(B_{\Gamma}K_{\gamma}B_{\Gamma}^{-1}p(v)+B_{\Gamma}(z_{\gamma}+c_{\gamma})) \\
&=B_{\Gamma}K_{\gamma}B_{\Gamma}^{-1}(q(v)-p(v))+(B-B_{\Gamma})(z_{\gamma}+c_{\gamma}) \\
&=A_{\gamma}m(v)+M(z_{\gamma}+c_{\gamma}),
\end{align*}
implying that $m$ is a $\Gamma$-symmetric relocation of $(H,p)$.
\end{proof}

As in the case of point group symmetry, 
a relocation $m$ is said to be {\em trivial} if $m$ is a linear combination of 
translations $m_t$ and a dilation $m_{\rm di}$.
Indeed, for any $t\in \bigcap_{\gamma\in \Gamma}(A_{\gamma}-I_d)$, 
translation $m_t$ defined by $m_t(v)=t$ for $v\in V(H)$ is a 
$\Gamma$-symmetric relocation with $M=0$;
on the other hand, a dilation $m_{\rm di}$ defined by $m_{\rm di}(v)=p(v)$ for $v\in V(H)$
is also a $\Gamma$-symmetric relocation with $M=B_{\Gamma}$.

Motivated by Proposition~\ref{prop:cry_parallel1} a $\Gamma$-symmetric $(H,p)$ is said to be 
{\em symmetrically robust} if 
all possible $\Gamma$-symmetric relocations are trivial.
%
%
%To see this, let $M_1,\dots, M_k$ be a basis of $\overline{{\rm Lat}}(\Gamma)$ taken from ${\rm Lat}(\Gamma)$.
Let us then show that checking the robustness can be reduced to computing 
the rank of linear polymatroids of quotient gain graphs.
We first remark that the condition for $M$ to be $M+B_{\Gamma}\in {\rm Lat}(\Gamma)$ can be ignored in the analysis.
\begin{lemma}
\label{lem:cry_collapse}
A $\Gamma$-symmetric $(H,p)$ is not symmetrically robust if and only if 
there are $m:V(H)\rightarrow \mathbb{R}^d$ and $M\in \overline{\rm Lat}(\Gamma)$ for which 
(\ref{eq:cry_parallel_const1}) and (\ref{eq:cry_symmetric_motion}) are satisfied and 
$m$ is not a linear combination of trivial relocations.
\end{lemma}
\begin{proof}
The sufficiency is trivial from the definition.
To see the necessity, suppose there are such $m$ and $M$.
Then, for any $\epsilon \in \mathbb{R}$, $\epsilon m$ and $\epsilon M$ satisfy (\ref{eq:cry_parallel_const1}) and (\ref{eq:cry_symmetric_motion}).
Moreover, if we take $\epsilon$ enough small, $B_{\Gamma}+\epsilon M\in {\rm Lat}(\Gamma)$ holds as $B_{\Gamma}$ is nonsingular.
Therefore $\epsilon m$ is a nontrivial $\Gamma$-symmetric relocation, and hence $(H,p)$ is not symmetrically robust.
\end{proof}

As in the finite case, we now simplify the system (\ref{eq:cry_parallel_const1}).
%To do that, let us take a representative vertex $v\in V(H)$ from each vertex orbit $\Gamma v\in V(H/\Gamma)$.
%Then, there is a natural one-to-one correspondence between $p$ and its quotient $p/\Gamma$ (resp., $m$ and $m/\Gamma$).
Recall that each edge orbit is written by 
$\Gamma e=\{(\gamma i, \gamma \psi_e j)\mid \gamma\in \Gamma\}$ with 
$\psi(e)=\psi_e=(A_{\psi_e},z_{\psi_e}+c_{\psi_e})$.
Thus, for each edge $(\gamma i, \gamma \psi_e j)$ in an edge orbit $\Gamma e$,
(\ref{eq:cry_parallel_const1}) is written by 
\begin{equation*}
\langle m(\gamma i)-m(\gamma \psi_e j),\alpha\rangle=0 \quad \forall \alpha\in \mathbb{R}^d 
\text{ such that } \langle p(\gamma i)-p(\gamma \psi_e j),\alpha\rangle=0.
\end{equation*}
These are indeed equivalent to one equation,
\begin{equation*}
\langle m(i)-m(\psi_e j),\alpha\rangle=0 \quad \forall \alpha\in \mathbb{R}^d
\text{ such that } \langle p(i)-p(\psi_e j),\alpha\rangle=0,
\end{equation*}
which is further converted to, by (\ref{eq:cry_symmetric_motion}),
\begin{equation*}
%\label{eq:cry_parallel_const4}
\langle m(i)- (A_{\psi_e}m (j)+M(z_{\psi_e}+c_{\psi_e})),\alpha\rangle=0 \quad \forall \alpha\in \mathbb{R}^d
\text{ such that } \langle p(i)-(A_{\psi_e}p(j)+t_{\psi_e}),\alpha\rangle=0.
\end{equation*}
for some $M\in \overline{\rm Lat}(\Gamma)$.

Therefore, by Lemma~\ref{lem:cry_collapse}, 
the problem can be considered in  a general $\Gamma$-gain graph $(G=(V,E),\psi)$, 
and we are interested in 
the space of $(m,M)\in (\mathbb{R}^d)^V \oplus \overline{\rm Lat}(\Gamma)$ satisfying 
\begin{equation}
\label{eq:cry_parallel_const4}
\langle m(i)- (A_{\psi_e}m (j)+M(z_{\psi_e}+c_{\psi_e})),\alpha\rangle=0 \quad \forall \alpha\in \mathbb{R}^d
\text{ such that } \langle p(i)-(A_{\psi_e}p(j)+t_{\psi_e}),\alpha\rangle=0
\end{equation}
for every $e=(i,j)\in E$. 
For further analysis, we shall take a basis $B_1,\dots, B_k\in \mathbb{R}^{d\times d}$ 
of $\overline{\rm Lat}(\Gamma)$,
where $k=\dim_{\dR} \overline{\rm Lat}(\Gamma)$. 
We then define a bilinear function $b_i:\mathbb{R}^d\times \mathbb{R}^d\rightarrow \mathbb{R}$ by
\begin{equation}
\label{eq:cry_bi}
b_i(\alpha,t)=\langle \alpha, B_iB_{\Gamma}^{-1}t\rangle \qquad ((\alpha,t)\in \mathbb{R}^d\times \mathbb{R}^d).
\end{equation}
Then, a bilinear map $b:\mathbb{R}^d\times \mathbb{R}^d\rightarrow \mathbb{R}^k$ is defined by 
$b(\alpha,t)=(b_1(\alpha, t),\dots, b_k(\alpha,t))^{\top}$.

Observe that $\Gamma_1$ is unitary with respect to $b_i$.
Indeed, for each $\gamma\in \Gamma$, we have $A_{\gamma}B_{\Gamma}=B_{\Gamma}K_{\gamma}$ and $A_{\gamma}B_i=B_iK_{\gamma}$ 
as $B_{\Gamma}, B_i\in\overline{\rm Lat}(\Gamma)$,
and hence  $b_i(A_{\gamma}\alpha,t)=\langle A_{\gamma}\alpha,B_iB_{\Gamma}^{-1}t\rangle
=\langle \alpha,A_{\gamma}^{-1}B_iB_{\Gamma}^{-1}t\rangle=
\langle \alpha,B_iK_{\gamma}^{-1}B_{\Gamma}^{-1}t\rangle=\langle \alpha,B_iB_{\Gamma}^{-1}A_{\gamma}^{-1}t\rangle=b_i(\alpha,A_{\gamma}^{-1}t)$.
Hence, $\Gamma_1$ is also unitary with respect to $b$.

We shall associate  a $(d-1)$-dimensional linear subspace $P_{e,\psi}'(p)$ with each 
edge $e=(i,j)\in E$, defined by
\begin{equation}
\label{eq:cry_direction_flat}
P_{e,\psi}'(p)=U_{e,\psi}\cap \{x\in (\mathbb{R}^d)^{V}\oplus \mathbb{R}^k \mid
\langle p(i)-(A_{\psi_e}p(j)+t_{\psi_e}),x(i)\rangle=0\}
\end{equation}
where  $U_{e,\psi}$ is, as defined in (\ref{eq:unified_U1})(\ref{eq:unified_U2}),
\begin{equation}
\label{eq:direction_flat_U1}
U_{e,\psi}=
\left\{x\in (\mathbb{R}^d)^{V}\oplus \mathbb{R}^k \Bigg\vert \begin{array}{lll}
x(i)+A_{\psi_e} x(j)=0, \\ 
b(x(i),t_{\psi_e})+x(\ast)=0, \\
x(V\setminus \{i,j\})=0 \end{array}
\right\}
\end{equation}
or 
\begin{equation}
\label{eq:direction_flat_U2}
U_{e,\psi}=\left\{x\in (\mathbb{R}^d)^{V}\oplus \mathbb{R}^k\Bigg\vert 
\exists \alpha\in \mathbb{R}^d\colon
\begin{array}{lll}
x(i)=(I_d-A_{\psi_e})\alpha, \\ 
x(\ast)=-b(\alpha,t_{\psi_e}), \\ 
x(V\setminus \{i,j\})=0 \end{array} \right \}
\end{equation}
depending on whether $e$ is a non-loop or a loop, respectively.
%Note that $\{(A_{\gamma}, B(z_{\gamma}+c_{\gamma}))\mid \gamma\in \Gamma\}$ forms the symmorphic space group of
%the arithmetic crystal classes to which  $\Gamma$ belongs.

\begin{lemma}
\label{lem:cry_parallel_eq}
Let $(G=(V,E),\psi)$ be a $\Gamma$-gain graph with a space group $\Gamma$.
Then, the dimension of the space of $(m,M)\in (\mathbb{R}^d)^V\oplus \overline{\rm Lat}(\Gamma)$ 
satisfying (\ref{eq:cry_parallel_const4}) is equal to
\[
d|V|+k-\dim_{\dR}\{P_{e,\psi}'(p)\mid e\in E\}
\]
where $k=\dim_{\dR} \overline{\rm Lat}(\Gamma)$.
\end{lemma}
\begin{proof}
%Note that the space of $(m,M)\in (\mathbb{R}^d)\oplus (\overline{\rm Lat}(\Gamma)-B_{\Gamma})$ satisfying (\ref{eq:cry_parallel_const4}) is 
%isomorphic to
%the space of $(m,B)\in (\mathbb{R}^d)\oplus \overline{\rm Lat}(\Gamma)$ satisfying 
%\begin{equation}
%\label{eq:cry_parallel_const5}
%\langle m(i)- (A_{\psi_e}m (j)+Bz_{\psi_e}),\alpha\rangle=0 \quad \forall \alpha\in \mathbb{R}^d
%\text{ such that } \langle p(i)-(A_{\psi_e}p(j)+t_{\psi_e}),\alpha\rangle=0
%\end{equation}
%for every $e=(i,j)\in E$. 
Since $\{B_1,\dots, B_k\}$ is a basis of $\overline{\rm Lat}(\Gamma)$, $\overline{\rm Lat}(\Gamma)$ is parameterized by $k$ parameters 
$a=(a_1,\dots,a_k)^{\top}\in \mathbb{R}^k$ such that $\overline{\rm Lat}(\Gamma)=\{\sum_{1\leq \ell\leq k}a_{\ell}B_{\ell}\mid a\in \dR^k\}$.
In other words, the space of $(m,M)$ satisfying (\ref{eq:cry_parallel_const4}) is 
isomorphic to the space of $(m,a)\in (\mathbb{R}^d)^V\oplus \mathbb{R}^k$ 
satisfying
\begin{equation}
\label{eq:cry_parallel_const6}
\langle m(i)- (A_{\psi_e}m (j)+\sum_{1\leq \ell\leq k} a_{\ell}B_{\ell} (z_{\psi_e}+c_{\psi_e})),
\alpha\rangle=0 \quad \forall \alpha\in \mathbb{R}^d
\text{ such that } \langle p(i)-(A_{\psi_e}p(j)+t_{\psi_e}),\alpha\rangle=0
\end{equation}
for every $e=(i,j)\in E$. 
Observe then that $(m,a)\in (\mathbb{R}^d)^{V}\oplus \mathbb{R}^k$ satisfies (\ref{eq:cry_parallel_const6}) if and only if 
$(m,a)$ is in the orthogonal complement of $\spa\{P_{e,\psi}'(p)\mid e\in E\}$, because,
for any $x\in P_{e,\psi}'(p)$, we have
\begin{align*}
\langle (m,a), x\rangle&=\langle m(i),x(i)\rangle + \langle m(j),x(j)\rangle + \langle a, x(\ast)\rangle \\
&= \langle m(i), x(i)\rangle - \langle m(j), A_{\psi_e}^{-1}x(i)\rangle - \langle a, b(x(i),t_{\psi_e})\rangle \\
&= \langle m(i)-A_{\psi_e} m(j), x(i)\rangle-\sum_{\ell} a_{\ell}b_{\ell}(x(i),t_{\psi_e}), \\
&= \langle m(i)-A_{\psi_e} m(j), x(i)\rangle-\sum_{\ell} a_{\ell}\langle B_{\ell}(z_{\psi_e}+c_{\psi_e}), x(i)\rangle \\
&=\langle m(i)-A_{\psi_e} m(j)-\sum_{\ell} a_{\ell} B_{\ell}(z_{\psi_e}+c_{\psi_e}), x(i)\rangle 
\end{align*}
with $\langle p(i)-(A_{\psi_e}p(j)+t_{\psi_e}),x(i)\rangle=0$
%, where we used $t_{\psi_e}=B_{\Gamma}z_{\psi_e}$.
%This completes the proof.
\end{proof}

Since the set of trivial relocations forms a linear space of dimension  
$(\dim_{\dR} \bigcap_{\gamma\in \Gamma}{\rm ker}(A_{\gamma}-I_d))+1$,  
Lemmas~\ref{lem:cry_collapse} and~\ref{lem:cry_parallel_eq} imply the following.

\begin{corollary}
\label{coro:cry_parallel_eq}
Let $(H,p)$ be a $\Gamma$-symmetric framework with a space group $\Gamma$, 
and $(H/\Gamma,\psi)$ be the quotient $\Gamma$-gain graph of $H$.
Then, $(H,p)$ is symmetrically robust if and only if 
\[
 \dim_{\dR}\{P_{e,\psi}'(p/\Gamma)\mid e\in E(H/\Gamma)\}=d|V/\Gamma|+k-1-
\dim_{\dR} \bigcap_{\gamma\in \Gamma}{\rm} {\rm ker} (A_{\gamma}-I_d). 
\]
\end{corollary}

\subsubsection{Combinatorial characterization}
By Corollary~\ref{coro:cry_parallel_eq}, 
it now suffices to analyze the {\em $\Gamma$-symmetric parallel redrawing polymatroid} of 
a $\Gamma$-gain graph $(G=(V,E),\psi)$,
that is, the linear polymatroid  with linear representation 
$e \mapsto P_{e,\psi}'(p)$.
The following theorem provides a combinatorial characterization of this polymatroid.

We say that the lattice of $\Gamma$ is  {\em generic} if
$B_{\Gamma}$ is expressed by $B_{\Gamma}=\sum_{i=1}^k s_{i}B_{i}$ 
such that $\{s_1,\dots, s_k\}$ is  algebraically independent 
over $\mathbb{Q}_{\Gamma_1}$. 
For a discrete point group ${\cal P}$, almost all space groups $\Gamma$ with $\Gamma_1={\cal P}$ have generic lattices.

\begin{theorem}
\label{theorem:cry_direction}
Let $(G=(V,E),\psi)$ be a $\Gamma$-gain graph for a space group $\Gamma$ with a generic lattice,
and $k=\dim_{\dR} \overline{\rm Lat}(\Gamma)$. 
Define $h_{\Gamma}$ by 
\[
 h_{\Gamma}(F)=d|V(F)|-dc(F)+\dim_{\dR}\{U_{e,\psi_F^{\circ}}\mid e\in E(G_F^{\circ}) \})-1,
\]
where $(G_F^{\circ},\psi_F^{\circ})$ is the compressed graph of $(G,\psi)$ by $F$ (defined 
in~\textsection$\ref{sec:unified}$).
Then, for almost all $p:V\rightarrow \mathbb{R}^d$,
\[
 \dim_{\dR}\{P_{e,\psi}'(p)\mid e\in E \}=\hat{h}_{\Gamma}(E).
\] 
In other words, the $\Gamma$-symmetric parallel redrawing polymatroid is equal to 
the polymatroid induced by $h_{\Gamma}$.
\end{theorem}
\begin{proof}
Let $h_{\Gamma}'=h_{\Gamma}+1$. 
Note that that the linear polymatroid $\mathbf{LP}(E,\Psi)$ of the linear representation 
$\Psi:e\mapsto U_{e,\psi}$  is a special case of
those given in \textsection\ref{sec:unified}, 
and hence Theorem~\ref{thm:unified} implies that 
$h_{\Gamma}'(F)=\dim_{\dR}\{U_{e,\psi}\mid e\in F\}$ for all $F\subseteq E$.

Since the lattice of $\Gamma$ is generic,
a lattice basis $B_{\Gamma}$ is written by $B_{\Gamma}=\sum_{i=1}^k s_iB_i$,
where $\{s_1,\dots, s_k\}$ is algebraically independent over $\mathbb{Q}_{\Gamma_1}$.
Let us take any $p:V\rightarrow \mathbb{R}^d$ such that
the coordinates of the image and $s_1,\dots, s_k$ are algebraically independent over $\mathbb{Q}_{\Gamma_1}$.
We define a hyperplane ${\cal H}$ of $(\mathbb{R}^d)^V\oplus \mathbb{R}^k$ by 
\begin{equation*}
{\cal H}=\{x\in (\mathbb{R}^d)^V\oplus \mathbb{R}^k \mid \sum_{v\in V}\langle p(v),x(v)\rangle+\langle s,x(\ast)\rangle =0 \}.
\end{equation*}
Then, observe that $P_{e,\psi}'(p)=U_{e,\psi}\cap {\cal H}$ for every $e\in E$, since, 
for any  $e=(i,j)\in E$ and any $x\in U_{e,\psi}$, we have 
\begin{align*}
\sum_{v\in V}\langle p(v),x(v)\rangle+\langle s,x(\ast)\rangle&=\langle p(i),x(i)\rangle+\langle p(j),x(j)\rangle+\langle s,x(\ast)\rangle \\
&=\langle p(i),x(i)\rangle+\langle p(j),-A_{\psi_e}^{-1}x(i)\rangle+\langle s, -b(x(i),t_{\psi_e})\rangle \\
&=\langle p(i),x(i)\rangle-\langle A_{\psi_e} p(j),x(i)\rangle-\sum_{\ell}s_{\ell}\langle 
B_{\ell}(z_{\psi_e}+c_{\psi_e}), x(i)\rangle \\
&=\langle p(i)-(A_{\psi_e}p(j)+t_{\psi_e}),x(i)\rangle, 
\end{align*}
by $t_{\psi_e}=B(z_{\psi_e}+c_{\psi_e})=(\sum_i s_iB_i)(z_{\psi_e}+c_{\psi_e})$.
Therefore, as the coordinates of the image of $p$ and $s_1,\dots, s_k$ are algebraically independent over  
$\mathbb{Q}_{\Gamma_1}$,
we conclude that 
the $\Gamma$-symmetric parallel redrawing polymatroid of $(G,\psi)$ 
is obtained from  $\mathbf{LP}(E,\Psi)$ by a Dilworth truncation, given in \textsection\ref{subsec:linear_truncation}. 
By Theorem~\ref{theorem:truncation}, we finally obtain
\begin{align*}
&\dim_{\dR}\{P_{e,\psi}'(p)\mid e\in E\} \\
&=
\min\{ \sum_i(\dim_{\dR}\{ U_{e,\psi}\mid e\in E_i \})-1) \mid  \text{a partition} \{E_1,\dots, E_k\} 
\text{ of } E\}  \\
&=\min\{ \sum_i (h_{\Gamma}'(E_i)-1) \mid \text{ a partition } \{E_1,\dots, E_k\} \text{ of } E \} \\
&=\min\{ \sum_i h_{\Gamma}(E_i) \mid \text{ a partition } \{E_1,\dots, E_k\} \text{ of } E \} \\
&=\hat{h}_{\Gamma}(E). 
\end{align*}
%
%Applying Theorem~\ref{theorem:truncation} and Theorem~\ref{thm:main_poly} as in the calculation of 
%Theorem~\ref{theorem:direction}, we obtain
%$\dim({\rm span}\{P_{e,\psi}'\mid e\in E\})=\hat{h}_{\rho}(E)$. 
\end{proof}

Combining Corollary~\ref{coro:cry_parallel_eq} and Theorem~\ref{theorem:cry_direction},
we complete characterizing the symmetric robustness of drawings with crystallographic symmetry.
\begin{corollary}
\label{cor:cry_parallel}
Let $H$ be a $\Gamma$-symmetric graph for a space group $\Gamma$ with a generic lattice, 
and $(G,\psi)$ be the quotient $\Gamma$-gain graph of $H$.
For almost all $\Gamma$-symmetric $p:V(H)\rightarrow \mathbb{R}^d$,  
$(H,p)$ is symmetrically robust if and only if 
the $\Gamma$-gain graph obtained from $G$ by replacing each edge $e\in E(G)$
by $d-1$ parallel copies contains 
an edge subset $I$ satisfying the following counting conditions;
\begin{itemize}
\item $|I|=d|V|+k-1-\dim_{\dR}\bigcap_{\gamma\in \Gamma}{\rm ker}(A_\gamma-I_d)$;
\item $|F|\leq d|V(F)|-dc(F)-1+\dim_{\dR}\{U_{e,\psi_F^{\circ}}\mid e\in E(G_F^{\circ}) \}$ 
for any nonempty  $F\subseteq I$,
\end{itemize}
where $(G_F^{\circ},\psi_F^{\circ})$ is the compressed graph of $G$ by $F$ and  
$k=\dim_{\dR}\overline{\rm Lat}(\Gamma)$.
\end{corollary}
%
%As a special case when $d=2$ and $\Gamma_1$ is a group of rotations, 
%Corollary~\ref{cor:cry_parallel} implies recent results by Malestein and Theran~\cite{malestein2011generic}.

\begin{remark}
As we have remarked, $\dim_{\dR}\{U_{e,\psi_F^{\circ}}\mid e\in E(G_F^{\circ})\}$
can be deterministically computed in polynomial time.
Thus, Corollary~\ref{cor:cry_parallel} gives a good characterization of the symmetric robustness.
Checking the counting condition can be reduced to a mathematical programming
given in (\ref{eq:poly_trun}), which can be deterministically done in polynomial time
 (see e.g.~\cite[Theorem 48.4]{Schrijver}).  \qed
\end{remark}

\subsection{Symmetry-forced rigidity with space group symmetry}
\label{subsec:cry_rigidity}
Let $C_{\pi/2}$ be the $2\times 2$-matrix representing the 4-fold rotation about the origin in $\mathbb{R}^2$.
In \textsection\ref{subsec:rigidity}, we have seen that the idea of characterizing 
robust drawings with point group symmetry can be directly applied to
characterizing the symmetry-forced infinitesimal rigidity of symmetric 2-dimensional frameworks, 
if the underlying point group commutes with $C_{\pi/2}$.
Here, we show an analogous fact in space groups.

The space group $\Gamma$ we can cope with here 
is the case when the linear part $\Gamma_1$ is a group of rotations 
about the origin.
More specifically, $\Gamma$ falls into five crystallographic group types, called 
${\sf p1}, {\sf p2}, {\sf p3}, {\sf p4}, {\sf p6}$ in terms of {\em Crystallographic notation}.
In the subsequent discussion, $\Gamma$ is assumed to be one of ${\sf p1}, {\sf p2}, {\sf p3}, {\sf p4}, {\sf p6}$.
%(The case of ${\sf p1}$ can be solved in the same manner, but we omit this easier case.)

Let $(H,p)$ be a $\Gamma$-symmetric framework
with a $\Gamma$-symmetric graph $H$ (with a specific free action $\theta$) 
and a $\Gamma$-symmetric point-configuration $p$.
Recall that an {\em infinitesimal motion} of $(H,p)$ is defined as $m:V(H)\rightarrow \mathbb{R}^2$ satisfying
\begin{equation}
\label{eq:cry_inf}
\langle m(i)-m(j),p(i)-p(j)\rangle=0 \qquad \{i,j\}\in E(H).
\end{equation}
As in the previous subsection, we are interested in $\Gamma$-symmetric motions,
where we say that an infinitesimal motion $m$ is {\em $\Gamma$-symmetric} if
there is $M\in \overline{\rm Lat}(\Gamma)$ such that 
\begin{equation}
\label{eq:cry_sym_inf}
m(\gamma v)=A_{\gamma}m(v)+Mz_{\gamma} \qquad \forall v\in V(H), \forall \gamma\in \Gamma
\end{equation}
(where $c_{\gamma}=0$ for any $\gamma\in \Gamma$ if $\Gamma\in \{{\sf p1},{\sf p2},{\sf p3},{\sf p4},{\sf p6}\}$).
Note that the space of infinitesimal lattice motions with fixed origin is now equal to $\overline{\rm Lat}(\Gamma)$. 

It can be observed that
%Indeed, for $t\in \mathbb{R}^2$, 
%the translation $m_t:V(H)\rightarrow \mathbb{R}^2x$ is an infinitesimal motion of $(H,p)$ with $M=0$,
%and the set of translations forms a two dimensinal linear space.
the infinitesimal rotation $m_r:V(H)\rightarrow \mathbb{R}^2$
defined by $m_r(v)=C_{\pi/2}p(v)$ is always a $\Gamma$-symmetric infinitesimal motion of  $(H,p)$.
To see this, let $M=C_{\pi/2}B_{\Gamma}$.
Then, since $C_{\pi/2}$ commutes with 
$A_{\gamma}=B_{\Gamma}K_{\gamma}B_{\Gamma}^{-1}$ for any $\gamma\in \Gamma$,
we have $C_{\pi/2}B_{\Gamma}K_{\gamma}B_{\Gamma}^{-1}C_{\pi/2}^{-1}=B_{\Gamma}K_{\gamma}B_{\Gamma}^{-1}$,
implying that $M=C_{\pi/2}B_{\Gamma}\in \overline{\rm Lat}(\Gamma)$.
Moreover, for any $\gamma\in \Gamma$ and $v\in V(H)$, we have
\begin{align*}
m_r(\gamma v)=C_{\pi/2}p(\gamma v)=C_{\pi/2}(A_{\gamma}p(v)+B_{\Gamma}z_{\gamma})
=A_{\gamma}m_r(v)+Mz_{\gamma},
\end{align*}
which implies that $m_r$ satisfies (\ref{eq:cry_sym_inf}) and is indeed a $\Gamma$-symmetric motion. 

Also it is easy to see that, 
for any $t\in \bigcap_{\gamma\in \Gamma}(A_{\gamma}-I_d)$, 
translation $m_t$ defined by $m_t(v)=t$ for $v\in V(H)$ is a 
$\Gamma$-symmetric motion with $M=0$.

%A $\Gamma$-symmetric motion is said to be {\em trivial} if
%it is a  linear combination of $m_t$ and $m_r$,
%and 
We say that $(H,p)$ is {\em infinitesimally rigid} 
if every possible $\Gamma$-symmetric infinitesimal motion is 
a linear combination of such translations $m_t$ and $m_r$.
%(This definition is applicable only to $\Gamma\in\{ {\sf p2},{\sf p3},{\sf p4},{\sf p6}\}$.)

As usual, taking a representative vertex $v$ from each vertex orbit $\Gamma v$,
(\ref{eq:cry_inf}) is reduced to the system,
\begin{equation}
\label{eq:cry_inf2}
\langle m(i)-m(\psi_e j),p(i)-p(\psi_e j)\rangle=0
\end{equation}
over all edge orbits from $\Gamma i$ to $\Gamma j$ with the gain $\psi_e$.
Thus, the problem can be considered in a general $\Gamma$-gain graph $(G=(V,E),\psi)$ with $p:V\rightarrow \mathbb{R}^2$,
and we are asked to compute the dimension of the linear space of $(m,M)\in (\mathbb{R}^2)^V\oplus \overline{\rm Lat}(\Gamma)$ satisfying
\begin{equation}
\label{eq:cry_inf3}
\langle m(i)-(A_{\psi_e}m(j)+Mz_{\psi_e}), p(i)-(A_{\psi_e}p(j)+t_{\psi_e})\rangle =0
\quad \forall (i,j)\in E.
\end{equation}

Let $B_1,\dots, B_k\in \mathbb{R}^{d\times d}$ be a basis of $\overline{\rm Lat}(\Gamma)$.
As in (\ref{eq:cry_bi}), we shall define a bilinear function 
$b_i:\mathbb{R}^d\times \mathbb{R}^d\rightarrow \mathbb{R}$ by 
$b_i(\alpha,t)=\langle \alpha, B_iB_{\Gamma}^{-1}t\rangle$ for $(\alpha, t)\in \mathbb{R}^d\times \mathbb{R}^d$,
and define $b:\mathbb{R}^d\times \mathbb{R}^d\rightarrow \mathbb{R}^k$ by 
$b=(b_1,\dots, b_k)^{\top}$.

To analyze the system (\ref{eq:cry_inf3}), 
we shall associate a 1-dimensional linear space with each $e=(i,j)\in E$ as follows:
\begin{equation}
R_{e,\psi}'(p)=U_{e,\psi}\cap \{x\in (\mathbb{R}^2)^V\oplus \mathbb{R}^k\mid  
\langle C_{\pi/2}(p(i)-(A_{\psi_e} p(j)+t_{\psi_e})),x(i)\rangle=0\},
\end{equation}
where $U_{e,\psi}$ is as defined  in (\ref{eq:direction_flat_U1})(\ref{eq:direction_flat_U2}).
%\[
% U_{e,\psi}=\{x\in (\mathbb{R}^2)^V\oplus (\mathbb{R}^k)\mid x(i)+A_{\psi_e}x(j), b(x(i),t_{\psi_e})+x(\ast)=0,
%x(V\setminus \{i,j\})=0\}.
%\]
Then, applying the same proof as that of Lemma~\ref{lem:cry_parallel_eq}, it is easy to check  the following.
\begin{lemma}
\label{lem:cry_rigidity1}
Let $\Gamma$ be a 2-dimensional space group whose linear part  $\Gamma_1$ is a group of rotations,
$(G=(V,E),\psi)$ a $\Gamma$-gain graph, and $p:V\rightarrow \mathbb{R}^2$.
Then, the space of $(m,M)\in (\mathbb{R}^d)^V\oplus \overline{\rm Lat}(\Gamma)$ satisfying 
(\ref{eq:cry_inf3}) is equal to
\[
 2|V|+k-\dim_{\dR}\{R_{e,\psi}'(p)\mid e\in E\}.
\]
where $k=\dim_{\dR} \overline{\rm Lat}(\Gamma)$.
\end{lemma}
\begin{theorem}
\label{thm:cry_rigidity}
Let $\Gamma$ be a 2-dimensional space group whose point group  $\Gamma_1$ is a group of rotations
and which has generic lattice  $B_{\Gamma}$,
and  $(G=(V,E),\psi)$ a $\Gamma$-gain graph.
Then, for almost all $p:V\rightarrow \mathbb{R}^2$,
$\dim_{\dR}\{R_{e,\psi}'(p)\mid e\in E\}=\hat{h}_{\Gamma}(E)$,
where $h_{\Gamma}$ is
\begin{equation*}
%\label{eq:cry_rigidity_h}
 h_{\Gamma}(F)=2|V(F)|-2c(F)+\dim_{\dR}\{U_{e,\psi_F^{\circ}}\mid e\in E(G_F^{\circ})\}-1 \quad (F\subseteq E).
\end{equation*}
\end{theorem}
\begin{proof}
%Since $\{B_1,\dots, B_k\}$ be a basis of $\overline{\rm Lat}(\Gamma)$,
%we can parameterize it by $\overline{\rm Lat}(\Gamma)=\{\}$
Recall that $C_{\pi/2}B_{\Gamma}\in \overline{\rm Lat}(\Gamma)$.
Hence, there is $s=(s_1,\dots, s_k)^{\top} \in \mathbb{R}^k$ such that 
$\sum_i s_iB_i=C_{\pi/2}B_{\Gamma}$.
Since the lattice is generic, 
we may assume that $\{s_1,\dots, s_k\}$ is algebraically independent over $\mathbb{Q}_{\Gamma_1}$.

Let us take any $p:V\rightarrow \mathbb{R}^2$ 
such that the coordinates of the image of $p$ and $s_1,\dots, s_k$ are algebraically independent 
over $\mathbb{Q}_{\Gamma_1}$.
We define a hyperplane ${\cal H}'$ of $(\mathbb{R}^2)^V\oplus \mathbb{R}^k$ by
\[
 {\cal H}'=\{x\in (\mathbb{R}^2)^V\oplus \mathbb{R}^k\mid \sum_{v\in V}\langle C_{\pi/2}p(v),x(v)\rangle+
\langle s,x(\ast)\rangle=0 \}.
\]
Then, it can be shown  that $R_{e,\psi}'(p)=U_{e,\psi}\cap {\cal H}'$ 
in the same analysis as the proof of Theorem~\ref{theorem:cry_direction}.
%Indeed, for each $e=(i,j)$ and each $x\in U_{e,\psi}$, we have
%\begin{align*}
%&\sum_{v\in V} \langle C_{\pi/2}p(v),x(v)\rangle + \langle s,x(\ast)\rangle \\
%&=\langle C_{\pi/2}p(i),x(i)\rangle + \langle C_{\pi/2}p(j),x(j)\rangle +\langle s, x(\ast) \rangle \\
%&=\langle C_{\pi/2}p(i),x(i)\rangle + \langle C_{\pi/2}p(j),-A_{\psi_e}^{-1}x(i)\rangle 
%+\sum_{\ell}b_{\ell}(x(i),t_{\psi_e})s_{\ell} \\
%&=\langle C_{\pi/2}(p(i)-A_{\psi,e}p(j)), x(i)\rangle  
%+\sum_{\ell}\langle B_{\ell}z_{\psi_e}, x(i)\rangle s_{\ell} \\ 
%&=\langle C_{\pi/2}(p(i)-A_{\psi,e}p(j)), x(i)\rangle  
%+\langle C_{\pi/2}B_{\Gamma}z_{\psi_e}, x(i)\rangle \\
%&=\langle C_{\pi/2}(p(i)-A_{\psi,e}p(j)-t_{\psi_e}), x(i)\rangle,  
%\end{align*}
%where the last two equations follow from 
%$t_{\psi_e}=B_{\Gamma}z_{\psi_e}$ and $C_{\pi/2}B_{\Gamma}=\sum_{\ell} s_{\ell}B_{\ell}$.
%This means that, for any $x\in U_{e,\psi}$, 
%$x\in H'$ if and only if 
%$\langle C_{\pi/2}(p(i)-A_{\psi,e}p(j)-t_{\psi_e}), x(i)\rangle=0$,
%implying $U_{e,\psi}\cap H'=R_{e,\psi}'(p)$.
%
Also, by Theorem~\ref{thm:unified}, 
$\dim_{\dR}\{U_{e,\psi}\mid e\in F\}=(h_{\Gamma}+1)(F)$
for any $F\subseteq E$.
Since ${\cal H}'$ is generic, by Theorem~\ref{theorem:truncation}, we obtain 
$\dim_{\dR}\{R_{e,\psi}'(p)\mid e\in F\}=\hat{h}_{\Gamma}(F)$.
\end{proof}
Lemma~\ref{lem:cry_rigidity1} and Theorem~\ref{thm:cry_rigidity} imply the following.
\begin{corollary}
\label{cor:cry_rigidity}
Let $\Gamma$ be a 2-dimensional space group whose linear part  $\Gamma_1$ is a group of rotations
and whose lattice is  generic.
Let $H$ be a $\Gamma$-symmetric graph.
Then, for almost all $\Gamma$-symmetric $p:V(H)\rightarrow \mathbb{R}^2$,
$(H,p)$ is symmetry-forced infinitesimally rigid if and only if 
the quotient $\Gamma$-gain graph $(G,\psi)$ contains an edge subset $I$ satisfying the following counting conditions:
\begin{itemize}
\item $|I|=2|V|+k-1-\dim_{\dR} \bigcap_{\gamma\in \Gamma}{\rm ker}(A_{\gamma}-I_d)$;
\item $|F|\leq 2|V(F)|-2c(F)+\dim_{\dR}\{U_{e,\psi_F^{\circ}}\mid e\in E(G_F^{\circ})\}-1$
for any nonempty $F\subseteq I$,
\end{itemize}
where $(G_F^{\circ},\psi_F^{\circ})$ is the compressed graph of $G$ by $F$ and  
$k=\dim_{\dR}\overline{\rm Lat}(\Gamma)$.
\end{corollary}
For $\Gamma={\sf p1},{\sf p2}, {\sf p3}, {\sf p4}, {\sf p6}$, $k=4,4,2,2,2$, respectively.

\section*{Acknowledgments}
We thank Tibor Jord{\'a}n and Vikt{\'o}ria Kaszanitzky for valuable discussions on 
count matroids of gain graphs.
The modeling of symmetric body-bar frameworks given in \textsection\ref{sec:action} 
is based on \cite{borcea2011periodic}.
We thank  Ciprian Borcea and Ileana Streinu for valuable discussions on 
this topic.

%The author is supported by JSPS Grant-in-Aid for Young Scientists (B).

\bibliographystyle{abbrv}
\bibliography{tani20120416}

\begin{thebibliography}{10}

\bibitem{asimov1978}
L.~Asimov and B.~Roth.
\newblock {The rigidity of graphs}.
\newblock {\em Tran.~Amer.~Math.~Soc.}, 245:279--289, 1978.

\bibitem{Biggs}
N.~Biggs.
\newblock {\em Algebraic Graph Theory}.
\newblock Cambridge University Press, 2 edition, 1994.

\bibitem{borcea2010}
C.~Borcea and I.~Streinu.
\newblock Periodic frameworks and flexibility.
\newblock {\em Proc.~R.~Soc.~Lond.~Ser.~A Math.~Phys.~Eng.~Sci.},
  466(2121):2633--2649, 2010.

\bibitem{borcea2011minimally}
C.~Borcea and I.~Streinu.
\newblock Minimally rigid periodic graphs.
\newblock {\em Bull.~Lond.~Math.~Soc.}, 43(6):1093--1103, 2011.

\bibitem{borcea2011periodic}
C.~Borcea, I.~Streinu, and S.~Tanigawa.
\newblock Periodic body-and-bar frameworks.
\newblock In {\em Proc.~24th ACM symposuim on Computational Geometry
  (SoCG2012)}, pages 347--356, 2012.

\bibitem{brylawski}
T.~Brylawski.
\newblock Constructions.
\newblock In N.~White, editor, {\em Theory of Matroids}, chapter~7. Cambridge
  University Press, 1986.

\bibitem{connelly2009symmetric}
R.~Connelly, P.~Fowler, S.~Guest, B.~Schulze, and W.~Whiteley.
\newblock When is a symmetric pin-jointed framework isostatic?
\newblock {\em Int.~J.~Solids Struct.}, 46(3-4):762--773, 2009.

\bibitem{dowling1973class}
T.~Dowling.
\newblock A class of geometric lattices based on finite groups.
\newblock {\em J.~Combin.~Theory Ser.~B}, 14(1):61--86, 1973.

\bibitem{edmonds:1968}
J.~Edmonds.
\newblock Matroid partition.
\newblock In {\em Mathematics of the Decision Sciences Part 1}, pages 335--345.
  AMS, 1968.

\bibitem{edmonds:1970}
J.~Edmonds.
\newblock Submodular functions, matroids, and certain polyhedra.
\newblock In R.~Guy, H.~Hanani, N.~Sauer, and J.~Sch{\"o}nheim, editors, {\em
  Combinatorial Structures and Their Applications}, pages 69--87, 1970.

\bibitem{fowler2000symmetry}
P.~Fowler and S.~Guest.
\newblock A symmetry extension of maxwell's rule for rigidity of frames.
\newblock {\em Int.~J.~Solids Struct.}, 37(12):1793--1804, 2000.

\bibitem{Frank2011}
A.~Frank.
\newblock {\em Connections in Combinatorial Optimization}.
\newblock Oxford Lecture Series in Mathematics and Its Applications. Oxford
  University Press, 2011.

\bibitem{fujishige}
S.~Fujishige.
\newblock {\em Submodular Functions and Optimization}.
\newblock Annals of Discrete Mathematics. Elsevier, 2nd edition, 2005.

\bibitem{gluck}
H.~Gluck.
\newblock {Almost all simply connected closed surfaces are rigid}.
\newblock In {\em Geometric topology}, volume 438 of {\em Lecture Notes in
  Mathematics}, pages 225--240. Springer, 1975.

\bibitem{GrossTucker}
J.~L. Gross and T.~W. Tucker.
\newblock {\em Topological Graph Theory}.
\newblock Dover, 1987.

\bibitem{HodgePedoe}
W.~V.~D. Hodge and D.~Pedoe.
\newblock {\em Methods of Algebraic Geometry}, volume~1.
\newblock Cambridge University Press, reissue edition, 3 1994.

\bibitem{jackson2010globally}
B.~Jackson and T.~Jord{\'a}n.
\newblock {Globally rigid circuits of the direction--length rigidity matroid}.
\newblock {\em J.~Combin.~Theory Ser.~B}, 100(1):1--22, 2010.

\bibitem{gain_sparsity}
T.~Jord{\'a}n, V.~E. Kaszanitzky, and S.~Tanigawa.
\newblock Gain-sparsity and symmetric rigidity in the plane.
\newblock manuscript.

\bibitem{laman:Rigidity:1970}
G.~Laman.
\newblock {On graphs and rigidity of plane skeletal structures}.
\newblock {\em J.~Eng.~Math.}, 4(4):331--340, 1970.

\bibitem{lovasz:1977}
L.~Lov{\'a}sz.
\newblock {Flats in matroids and geometric graphs}.
\newblock In {\em Combinatorial surveys: proceedings of the Sixth British
  Combinatorial Conference}, pages 45--86. Academic Press, 1977.

\bibitem{lovasz:1982}
L.~Lov{\'a}sz and Y.~Yemini.
\newblock {On generic rigidity in the plane}.
\newblock {\em SIAM J.~Algebr.~Discrete Methods}, 3:91--98, 1982.

\bibitem{malestein2011generic}
J.~Malestein and L.~Theran.
\newblock Generic rigidity of frameworks with orientation-preserving
  crystallographic symmetry.
\newblock {\em arXiv:1108.2518}, 2011.

\bibitem{malestein2010}
J.~Malestein and L.~Theran.
\newblock Generic combinatorial rigidity of periodic frameworks.
\newblock {\em Adv.~Math.}, (to apear).

\bibitem{Mason:1976}
J.~H. Mason.
\newblock Matroids as the study of geometrical configurations.
\newblock In M.~Aigner, editor, {\em Higher Combinatorics (Proceedings NATO
  Advanced Study Institute, 1976)}, pages 133--176. D.~Reidel, 1977.

\bibitem{Mason:1981}
J.~H. Mason.
\newblock Glueing matroids together: a study of dilworth truncations and
  matroid analogues of exterior and symmetric powers.
\newblock In L.~Lov{\'a}sz and V.~T. S{\'o}s, editors, {\em Algebraic Methods
  in Graph Theory Vol. II (Colloquium Szeged, 1978)}, pages 519--561.
  North-Holland, 1981.

\bibitem{owen2008frameworks}
J.~Owen and S.~Power.
\newblock Frameworks, symmetry and rigidity.
\newblock {\em arXiv:0812.3785}, 2008.

\bibitem{owen:2011}
J.~C. Owen and S.~C. Power.
\newblock Infinite bar-joint frameworks, crystals and operator theory.
\newblock {\em New York J.~Math.}, 17:445--490, 2011.

\bibitem{oxley}
J.~Oxley.
\newblock {\em {Matroid theory},}.
\newblock Oxford University Press, USA, 2nd edition, 2011.

\bibitem{power2012}
S.~C. Power.
\newblock Polynomials for crystal frameworks and the rigid unit mode spectrum.
\newblock Technical report,
  http://www.maths.lancs.ac.uk/~power/PowerPhilTran.pdf.

\bibitem{ross2011}
E.~Ross.
\newblock {\em Geometric and combinatorial rigidity of periodic frameworks as
  graphs on the torus}.
\newblock PhD thesis, York University, Toronto, May 2011.

\bibitem{ross2012rigidity}
E.~Ross.
\newblock The rigidity of periodic body-bar frameworks on the three-dimensional
  fixed torus.
\newblock Technical report, arXiv:1203.6611, 2012.

\bibitem{Schrijver}
A.~Schrijver.
\newblock {\em {Combinatorial optimization: polyhedra and efficiency}}.
\newblock Springer, 2003.

\bibitem{schulze}
B.~Schulze.
\newblock {\em Combinatorial and geometric rigidity with symmetric
  constraints}.
\newblock Ph. thesis, York University, 2009.

\bibitem{schulze2010symmetric}
B.~Schulze.
\newblock Symmetric versions of laman's theorem.
\newblock {\em Discrete Comput.~Geom.}, 44(4):946--972, 2010.

\bibitem{schulze:2010}
B.~Schulze.
\newblock Symmetry as a sufficient condition for a finite flex.
\newblock {\em SIAM J.~Discrete Math.}, 24(4):1291--1312, 2010.

\bibitem{Schulze2012protein}
B.~Schulze, A.~Sljoka, and W.~Whiteley.
\newblock How does symmetry impact the flexibility of proteins?
\newblock Technical report, 2012.

\bibitem{schulze2011orbit}
B.~Schulze and W.~Whiteley.
\newblock The orbit rigidity matrix of a symmetric framework.
\newblock {\em Discrete Comput.~Geom.}, 46(3):561--598, 2011.

\bibitem{Servatius:1999}
B.~Servatius and W.~Whiteley.
\newblock Constraining plane configurations in computer-aided design:
  combinatorics of directions and lengths.
\newblock {\em SIAM J.~Discrete Math.}, 12(1):136--153, 1999.

\bibitem{tanigawa_truncation}
S.~Tanigawa.
\newblock Generic rigidity matroids with dilworth truncations.
\newblock {\em SIAM J.~Discrete Math.}, 26:1412--1439, 2012.

\bibitem{tay:84}
T.~Tay.
\newblock {Rigidity of multi-graphs. I: Linking rigid bodies in $n$-space}.
\newblock {\em J.~Combin.~Theory Ser.~B}, 36(1):95--112, 1984.

\bibitem{whiteley:88}
W.~Whiteley.
\newblock {The union of matroids and the rigidity of frameworks}.
\newblock {\em SIAM J.~Discrete Math.}, 1(2):237--255, 1988.

\bibitem{Whiteley:89}
W.~Whiteley.
\newblock {A matroid on hypergraphs with applications in scene analysis and
  geometry}.
\newblock {\em SIAM J.~Discrete Math.}, 4:75--95, 1989.

\bibitem{Whitley:1997}
W.~Whiteley.
\newblock {Some matroids from discrete applied geometry}.
\newblock {\em Contemporary Mathematics}, 197:171--312, 1996.

\bibitem{whittle1989generalisation}
G.~Whittle.
\newblock A generalisation of the matroid lift construction.
\newblock {\em Tran.~Amer.~Math.~Soc.}, 316(1):141--159, 1989.

\bibitem{zaslavsky1989biased}
T.~Zaslavsky.
\newblock Biased graphs "{I}". bias, balance, and gains.
\newblock {\em J.~Combin.~Theory Ser.~B}, 47(1):32--52, 1989.

\bibitem{zaslavsky1991biased}
T.~Zaslavsky.
\newblock Biased graphs "{II}". the three matroids.
\newblock {\em J.~Combin.~Theory Ser.~B}, 51(1):46--72, 1991.

\bibitem{zaslavsky1994}
T.~Zaslavsky.
\newblock Frame matroids and biased graphs.
\newblock {\em Eur.~J.~Combin}, 15:303--307, 1994.

\bibitem{zaslavsky2003biased}
T.~Zaslavsky.
\newblock Biased graphs "{IV}": Geometrical realizations.
\newblock {\em J.~Combin.~Theory Ser.~B}, 89(2):231--297, 2003.

\end{thebibliography}

\end{document}